\documentclass[12pt]{amsart}
\usepackage{amsthm}
\usepackage{amssymb}
\usepackage{amsmath}
\theoremstyle{plain}
\newtheorem{proposition}{Proposition}
\newtheorem{theorem}[proposition]{Theorem}
\newtheorem{lemma}[proposition]{Lemma}
\newtheorem{corollary}[proposition]{Corollary}

\theoremstyle{definition}
\newtheorem{definition}[proposition]{Definition}

\theoremstyle{definition}

\newtheorem{remark}[proposition]{Remark}

\numberwithin{equation}{section}

\numberwithin{proposition}{section}

\gdef\myletter{}

\let\savetheequation\theequation
\def\theequation{\savetheequation\myletter}

\newcommand{\CC}{{\mathbb C}}

\newcommand{\RR}{{\mathbb R}}

\newcommand{\PP}{{\mathbb P}}

\renewcommand{\Re}{\mbox{Re}}

\renewcommand{\date}{\today}
\def \bar{\overline}
\def \hat{\widehat}

\begin{document}

%\vskip 3mm

\title[Berman-Boucksom]{\bf Weighted Pluripotential Theory Results of Berman-Boucksom}

\maketitle
\author{Compiled by N. Levenberg}

\date

\section{Introduction.}\label{sec:intro}

The main goal of these notes is to present a more-or-less self-contained discussion of some of the recent results and techniques of Berman-Boucksom in the setting of weighted pluripotential theory. We follow mainly the arguments of Berman-Boucksom in \cite{[BB]}, \cite{BBnew} and \cite{[BB2]} as well as the Berman paper \cite{[B1]} and the paper of Berman-Boucksom-Nystrom \cite{[BBN]} {\it These notes are for comprehension purposes only, not for publication; the key results are from \cite{[BB]}, \cite{BBnew}, \cite{[BB2]}, \cite{[BBN]}  and \cite{[B1]}}. In particular, these notes provide 
\begin{enumerate}
\item a proof of {\it Rumely's formula} relating the transfinite diameter of a compact set $K$ in $\CC^d$ with certain integrals involving the Robin function $\rho_K$ of $K$, as well as weighted versions of the formula;
\item a proof that {\it asymptotically Fekete arrays} distribute asymptotically to the Monge-Ampere measure $\mu_K:=\frac{1}{(2\pi)^d}(dd^cV_{K}^*)^d$ of the global extremal function $V_{K}^*$, and, more generally, that {\it asymptotically weighted Fekete arrays} distribute asymptotically to the weighted Monge-Ampere measure $\mu_{K,Q}:=\frac{1}{(2\pi)^d}(dd^cV_{K,Q}^*)^d$ of the weighted global extremal function $V_{K,Q}^*$;
\item a proof that (weighted) {\it optimal measures} distribute asymptotically like the (weighted) Monge-Ampere measure $\mu_K$ ($\mu_{K,Q}$);
\item general results on {\it strong Bergman asymptotics} for Bernstein-Markov pairs $(K,\mu)$ where $\mu$ is a measure on $K$, as well as for weighted Bernstein-Markov triples $(K,\mu,Q)$.
\end{enumerate}

We provide some  background results on pluripotential theory and weighted pluripotential theory. Many items which are nowadays considered fundamental (albeit not necessarily elementary to prove) are stated without proof but with references. In particular, Appendix B in \cite{safftotik} as well as \cite{bloom} are good sources for weighted pluripotential theory. A note to those not yet initiated in this area: the weighted theory is {\it essential} for proving even the unweighted versions of (1)-(4).

We would like to thank several people for valuable comments regarding portions of this material; in alphabetical order, here is a subset: Muhammed Ali Alan, Tom Bagby, Tom Bloom, Len Bos, Dan Coman and Eugene Poletsky. The biggest thanks, of course, go to Robert Berman and S\'ebastien Boucksom for their deep and beautiful work, as well as their patience with my questions. My main reason for writing these notes is to advertise their contributions to the pluripotential theory community.   

\setcounter{tocdepth}{2}
\tableofcontents

\section{Background.}\label{sec:back}

\subsection {Extremal functions and extremal measures.} In  pluripotential theory, one considers $E\subset \CC^d$ compact or, slightly more general, $E\subset \CC^d$ is a bounded Borel set. The global extremal function or global pluricomplex Green function of $E$ is given by $V^*_{E}(z):=\limsup_{\zeta \to z}V_{E}(\zeta)$ where
$$V_E(z):=\sup \{u(z):u\in L(\CC^d), \ u\leq 0 \ \hbox{on} \ E\}.$$
Here, $L(\CC^d)$ is the set of all plurisubharmonic (psh) functions $u$ on $\CC^d$ with the property that $u(z) - \log |z| = 0(1), \ |z| \to \infty$. Either $V^*_{E}\in L^+(\CC^d)$ where
$$L^+(\CC^d)=\{u\in L(\CC^d): u(z)\geq \log^+|z| + C\}$$
and $C$ is a constant depending on $u$, or $V^*_{E}\equiv +\infty$. This latter situation happens precisely when $E$ is pluripolar. If $E$ is compact,
$$V_E(z) = \sup \{\frac{1}{deg (p)}\log |p(z)|: ||p||_E \leq 1\}$$
where $p$ is a polynomial. We call 
$$\mu_E:= \frac{1}{(2\pi)^d}(dd^cV_E^*)^d$$
the {\it extremal measure} for $E$ if $E$ is not pluripolar. Here, $dd^c = 2i \partial \bar \partial$ although in Appendix 2 we incorporate a factor of $2\pi$. 

In the weighted theory, one restricts to {\it closed} sets but, for certain weights, these sets may be unbounded. To be precise, let $K\subset \CC^d$ be closed and let $w$ be an admissible weight function on $K$:  $w$ is a nonnegative, usc function with
$\{z\in K:w(z)>0\}$ nonpluripolar; if $K$ is unbounded, we require that $w$ satisfies the growth property
\begin{equation} \label{grprop} |z|w(z)\to 0 \ \hbox{as} \ |z|\to \infty, \ z\in K. 
\end{equation}
Let $Q:= -\log w$ -- we use $Q$ and $w$ interchangeably -- and define the weighted extremal function or weighted
pluricomplex Green function $V^*_{K,Q}(z):=\limsup_{\zeta \to z}V_{K,Q}(\zeta)$ where
$$V_{K,Q}(z):=\sup \{u(z):u\in L(\CC^d), \ u\leq Q \ \hbox{on} \ K\}. $$
We have $V^*_{K,Q}\in L^+(\CC^d)$. In the unbounded case, note that   property (\ref{grprop}) is equivalent to 
$$Q(z)-\log |z| \to \ +\infty \ \hbox{as} \ |z|\to \infty \ \hbox{through points in} \ K.$$ 
Due to this growth assumption for $Q$, $V_{K,Q}$ is well-defined and equals $V_{K\cap \mathcal B_R,Q}$ for $R>0$ sufficiently large where $\mathcal B_R=\{z:|z|\leq R\}$ (Definition 2.1 and Lemma 2.2 of Appendix B in \cite{safftotik}). It is known that the support
$$S_w:=\hbox{supp}(\mu_{K,Q})$$ 
of the {\it weighted extremal measure} 
$$\mu_{K,Q}:=\frac{1}{(2\pi)^d}(dd^cV_{K,Q}^*)^d$$
is compact -- this plays a very important role in what follows; 
\begin{equation} \label{suppw}
S_w \subset S_w^*:=\{z\in K: V_{K,Q}^*(z)\geq Q(z)\};
\end{equation}
moreover, 
$$V_{K,Q}^*=Q \ \hbox{q.e. on} \ S_w$$
(i.e., $V_{K,Q}^*=Q$ on $S_w - F$ where $F$ is pluripolar); and if $u\in L(\CC^d)$ satisfies $u\leq Q$ q.e. on $S_w$ then $u\leq V_{K,Q}^*$ on $\CC^d$. Indeed,
\begin{equation}
\label{equivdef}
V_{K,Q}(z)=\sup \{\frac{1}{deg(p)}\log |p(z)|:||w^{deg(p)}p||_{S_w}\leq 1, \  p \ \hbox{polynomial} \}
\end{equation}
and
$$||w^{deg(p)}p||_{S_w} = ||w^{deg(p)}p||_{K}.$$
Theorem 2.8 of Appendix B in \cite{safftotik} includes the slightly stronger statement that
$$V_{K,Q}^*(z)=\bigl[\sup \{\frac{1}{deg(p)}\log |p(z)|:||w^{deg(p)}p||_{K}^*\leq 1, \  p \ \hbox{polynomial} \}\bigr]^*$$ where
$$||w^{deg(p)}p||_{K}^*:=\inf \{||w^{deg(p)}p||_{K\setminus F}: F \subset K \ \hbox{pluripolar}\}.$$

The unweighted case is when $K$ is compact and $w\equiv 1$ ($Q\equiv 0$); we then write $V_K:=V_{K,0}$ to be consistent with the previous notation. 

Even in one variable ($d=1$) the weighted theory introduces new phenomena from the unweighted case. As an elementary example, $\mu_K$ puts no mass on the interior of $K$ (in one variable, the support of $\mu_K$ is the outer boundary of $K$); but this is not necessarily true in the weighted setting. As a simple but illustrative example, taking $K$ to be the closed unit ball $\{z:|z|\leq 1\}$ and $Q(z)=|z|^2$, it is easy to see that $V_{K,Q}=Q$ on the ball $\{z:|z|\leq 1/\sqrt 2\}$ and $V_{K,Q}(z)=\log |z| + 1/2 - \log (1/\sqrt 2)$ outside this ball. Indeed, taking $K=\CC^d$ and the same weight function $Q(z)=|z|^2$, one obtains the same weighted extremal function $V_{K,Q}$; this illustrates one of the results in subsection \ref{sec:bergcd}.

A compact set $K$ is called {\it regular} if $V_K=V_K^*$; i.e., $V_K$ is continuous; and $K$ is {\it locally regular} if for each $z\in K$, the set $K$ is locally regular at $z$; i.e., the sets $K\cap \overline{\mathcal B_r(z)}$ are regular for $r>0$ where $\mathcal B_r(z)$ denotes the ball of radius $r$ centered at $z$. If $K$ is locally regular and $Q$ is continuous, then $V_{K,Q}$ is continuous (cf, \cite{siciak}). Indeed, if $K$ is compact and $V_{K,Q}$ is continuous for {\it every} continuous admissible weight $Q$ on $K$, then $K$ is locally regular (Proposition 6.1 \cite{NQD}). If $K$ is an arbitrary compact set, then for $\epsilon >0$, 
$$K_{\epsilon}:=\{z\in \CC^d: {\rm dist}(z,K)\leq \epsilon\}$$
is a regular compact set and 
$$\lim_{\epsilon \downarrow 0}V_{K_{\epsilon}}=V_K;$$
indeed, $V_{K_{1/j}}\uparrow V_K$ as $j\uparrow \infty$ (cf., Corollary 5.1.5 of \cite{[K]}). In the weighted case, given a compact, nonpluripolar set $K$ and an admissible weight $w$ on $K$, we can find a sequence of locally regular compacta $\{K_j\}$ decreasing to $K$ and a sequence of weights $\{w_j\}$ with $w_j$ continuous and admissible on $K_j$ such that $w_j \downarrow w$ on $K$. In this setting,
\begin{equation} \label{wtdapprox}
V_{K_j,Q_j} \uparrow V_{K,Q}.
\end{equation}
A proof of (\ref{wtdapprox}) can be found in \cite{bloomlev2} (equation (7.4)). We give a construction of such locally regular compacta $\{K_j\}$  and weights $\{w_j\}$ in Proposition \ref{polypoly} of Appendix 1.

A bit of notation: if $Q$ is an admissible weight on $\CC^d$, then we write $V_Q:=V_{\CC^d,Q}$. Note that an admissible weight $Q$ on a closed set $K$ can always be extended in a trivial way to all of $\CC^d$ by setting $Q:=+\infty$ on $\CC^d\setminus K$ so that $V_Q=V_{K,Q}$. 

We often want to emphasize the relation between the weight $Q$ and the weighted extremal function $V_{K,Q}^*$, so we may write 
\begin{equation}
\label{pq}
P(Q) =P_K(Q) :=V_{K,Q}^*
\end{equation}
especially in section \ref{diffep}. Two elementary observations are that this  ``projection'' operation $P$ is increasing and concave. Precisely, 
if $Q_1\leq Q_2$ are admissible weights on $E$, then trivially $P(Q_1)\leq P(Q_2)$; and if $0\leq s \leq 1$ and $a,a'$ are admissible weights on $E$,
\begin{equation}
\label{projcon}
P(sa + (1-s)a') \geq sP(a) + (1-s)P(a').
\end{equation}
Note that $sa + (1-s)a'$, being a convex combination of $a, a'$, is an admissible weight on $E$. Then (\ref{projcon}) follows since the right-hand-side is a competitor for the weighted extremal function on the left-hand-side. Also, for future use, we show that $P$ is Lipschitz; i.e., if $a,b$ are  admissible weights on $E$ and $0\leq t\leq 1$ then on $\CC^d$, 
\begin{equation}
\label{loclip}
|P(a+t(b-a)) -P(a)|\leq Ct
\end{equation}
where $C=C(a,b)$. Similarly, if $u\in C(E)$, we have, for $t\in \RR$, 
\begin{equation}
\label{loclip2}
|P(a+tu) -P(a)|\leq C|t|
\end{equation}
where $C=C(u)$. 
To see these, first observe that 
$$P(a+t(b-a))\leq a+t(b-a)$$
on $\CC^d$ so that, on $D(0):=\{P(a)=a\}$, 
$$P(a+t(b-a))\leq a +|t|\sup_{D(0)}|b-a|$$
which implies, by the definition of $P(a)$, equation (\ref{suppw}) and the remark following, that 
$$P(a+t(b-a))\leq P(a) + |t|\sup_{D(0)}|b-a|$$
on $\CC^d$. Similarly,
$$P(a)\leq a$$
on $\CC^d$ so that 
$$P(a) \leq a= P(a+t(b-a))-t(b-a)$$
on $D(t):=\{P(a+t(b-a))=a+t(b-a)\}$ which implies
$$P(a)\leq P(a+t(b-a)) +|t|\sup_{D(t)}|b-a|$$
on $D(t)$. Thus again by the definition of $P(a+t(b-a))$, equation (\ref{suppw}) and the remark following,
$$P(a)\leq P(a+t(b-a)) +|t|\sup_{D(t)}|b-a|$$
on $\CC^d$. This gives (\ref{loclip}) with $C=\max[\sup_{D(0)}|b-a|,\sup_{D(t)}|b-a|]$ or (\ref{loclip2}) with $C=\sup_{E}|u|$. In the former case, if $E$ is unbounded, in order that $\max[\sup_{D(0)}|b-a|,\sup_{D(t)}|b-a|]$ is a {\it finite} constant which is independent of $t$, we assume that 
\begin{equation}
\label{unbhyp}
\cup_{0\leq t\leq 1} D(t) \ \hbox{is bounded and}  \ u:=b-a\in L^{\infty}(\cup_{0\leq t\leq 1} D(t)).
\end{equation}
Then (\ref{loclip}) holds. This observation will be used in the proof of Theorem \ref{keycn}. Of course in both cases, if $E$ is compact, $C$ is finite.

Associated to the asymptotic behavior of the weighted
pluricomplex Green function $V^*_{K,Q}$ is its {\it Robin function}
$$\rho_{K,Q}(z):=\limsup_{|\lambda| \to \infty} [V^*_{K,Q}(\lambda z)- \log |\lambda|].$$
We write $\rho_K:= \rho_{K,0}$ in the unweighted case. 
This is a logarithmically homogeneous psh function in $L(\CC^d)$. A ``projectivized'' version is 
\begin{equation}
\label{projrob} \tilde \rho_{K,Q}(z):=\limsup_{|\lambda| \to \infty} [V^*_{K,Q}(\lambda z)- \log |\lambda z|].
\end{equation}
Indeed, given any function $u\in L(\CC^d)$, we can form the Robin function
 $$\rho_u(z):=\limsup_{|\lambda| \to \infty} [u(\lambda z)- \log |\lambda|]$$
 and the ``projectivized'' Robin function
$$\tilde \rho_u(z):=\limsup_{|\lambda| \to \infty} [u(\lambda z)- \log |\lambda z|].$$
Since $\tilde \rho(tz)=\tilde \rho(z)$ for $t\in \CC\setminus \{0\}$, we can consider $\tilde \rho$ as a function on lines through the origin in $\CC^d$; i.e., as a function on $\PP^{d-1}$. In general, either $\rho_u\equiv -\infty$ or $\rho_u$ is a logarithmically homogeneous psh function in $L(\CC^d)$. This latter case always occurs if, e.g., $u\in L^+(\CC^d)$. In this case $\tilde \rho_u$ is bounded (above and below).

We record two results related to the classes $L(\CC^d)$ and $L^+(\CC^d)$. The first is a domination principle and the second is a comparison principle in $L^+(\CC^d)$. The first result is Lemma 6.5 in \cite{[BT]}.

\begin{proposition}
\label{lplusdom}
Let $u\in L^+(\CC^d)$ and $v\in L(\CC^d)$. If $v\leq u$ a.e-$(dd^cu)^d$, then $v\leq u$ in $\CC^d$.
\end{proposition}

\begin{proposition} 
\label{L+comp}
Let $u_1,u_2\in L^+(\CC^d)$. Then
$$\int_{\{u_1 < u_2\}}(dd^cu_2)^d \leq \int_{\{u_1 < u_2\}}(dd^cu_1)^d.$$
\end{proposition}

\noindent Note that the integrand may be unbounded but since $u_1,u_2\in L^+(\CC^d)$ each integral is finite (indeed, bounded above by $(2\pi)^d$). 

\begin{proof} By adding a constant to $u_1$, if necessary, we may assume $u_1\geq 0$. Then for $\epsilon >0$, we have 
$$\{(1+\epsilon)u_1 < u_2\} \subset \{u_1 < u_2\}$$ 
and $\{(1+\epsilon)u_1 < u_2\}$ is bounded. By the standard comparison theorem for locally bounded psh functions on bounded domains (cf., Theorem 3.7.1, \cite{[K]}),
\begin{equation}
\label{compus}
\int_{\{(1+\epsilon)u_1 < u_2\}}(dd^cu_2)^d \leq (1+\epsilon)^d\int_{\{(1+\epsilon)u_1 < u_2\}}(dd^cu_1)^d.
\end{equation}
Clearly
$$\bigcup_{j=1}^{\infty} \{(1+1/j)u_1 < u_2\} = \{u_1 < u_2\}$$
so applying (\ref{compus}) with $\epsilon = 1/j$, the result follows by monotone convergence upon letting $j\to \infty$.
\end{proof}

\subsection{Transfinite diameter.}

We let $\mathcal P_n$ denote the complex vector space of holomorphic polynomials of degree at most $n$ and we adopt the convention that 
$$N =N(n) :=\hbox{dim} \mathcal P_n =   {n+d \choose n}.$$
Thus 
$$\mathcal P_n= \hbox{span} \{e_1,...,e_N\}$$ 
where $\{e_j(z):=z^{\alpha(j)}\}$ are the standard basis monomials. For 
points $\zeta_1,...,\zeta_N\in \CC^d$, let
\begin{equation} \label{vdm}VDM(\zeta_1,...,\zeta_N)=\det [e_i(\zeta_j)]_{i,j=1,...,N}  \end{equation}
$$= \det
\left[
\begin{array}{ccccc}
 e_1(\zeta_1) &e_1(\zeta_2) &\ldots  &e_1(\zeta_N)\\
  \vdots  & \vdots & \ddots  & \vdots \\
e_N(\zeta_1) &e_N(\zeta_2) &\ldots  &e_N(\zeta_N)
\end{array}
\right]$$
and for a compact subset $K\subset \CC^d$ let
$$V_n =V_n(K):=\max_{\zeta_1,...,\zeta_N\in K}|VDM(\zeta_1,...,\zeta_N)|.$$
Then 
\begin{equation} \label{tdlim} \delta(K)=\delta^1(K)=\lim_{n\to \infty}V_{n}^{\frac{d+1}{dnN}} \end{equation} where $\frac{dnN}{d+1}$ is the sum of the degrees of a set of basis monomials for $ \mathcal P_n$, is the {\it transfinite diameter} of $K$. The temporary superscript ``$1$'' refers to a weight $w\equiv 1$. Zaharjuta \cite{zah} showed that the limit exists; we outline his proof in Appendix 1.

The {\it Rumely formula}, relating the transfinite diameter 
$\delta(K)$ with $V_{K}^*$ and $\tilde \rho_{K}$, can be stated as follows:
\begin{equation}\label{unwtdtruerumely}
-\log \delta(K)=\frac{1}{d(2\pi)^{d-1}} \int_{\PP^{d-1}}[\tilde \rho_{K}- \tilde \rho_{T}]\sum_{j=0}^{d-1}(dd^c \tilde \rho_{K}+\omega)^j \wedge (dd^c \tilde \rho_{T}+\omega)^{d-j-1}.
\end{equation}
Here, $\omega$ is the standard K\"ahler form on $\PP^d$ and $T$ is the unit torus in $\CC^d$. We will prove a generalization of (\ref{unwtdtruerumely}) in section \ref{bf}.

More generally, let $K\subset \CC^d$ be compact and let
$w$ be an admissible weight function on
$K$.  Given $\zeta_1,...,\zeta_N\in K$, let
$$W(\zeta_1,...,\zeta_N):=VDM(\zeta_1,...,\zeta_N)w(\zeta_1)^{n}\cdots w(\zeta_N)^{n}$$
$$= \det
\left[
\begin{array}{ccccc}
 e_1(\zeta_1) &e_1(\zeta_2) &\ldots  &e_1(\zeta_N)\\
  \vdots  & \vdots & \ddots  & \vdots \\
e_N(\zeta_1) &e_N(\zeta_2) &\ldots  &e_N(\zeta_N)
\end{array}
\right]\cdot w(\zeta_1)^{n}\cdots w(\zeta_N)^{n}$$
be a {\it weighted Vandermonde determinant}. Let
\begin{equation} \label{wn} W_n(K):=\max_{\zeta_1,...,\zeta_N\in K}|W(\zeta_1,...,\zeta_N)|
\end{equation}
and define an {\it $n-$th weighted Fekete set for $K$ and $w$} to be a set of $N$ points $\zeta_1,...,\zeta_ N\in K$ with the property that
$$|W(\zeta_1,...,\zeta_N)|=W_n(K).$$
We also write $\delta^{w,n}(K):=W_n(K)^{\frac{d+1}{dnN}}$ and define
\begin{equation} \label{deltaw} \delta^w(K):=\lim_{n\to \infty}\delta^{w,n}(K):=\lim_{n\to \infty}W_{n}(K)^{\frac{d+1}{dnN}}. \end{equation}
A proof of the existence of the limit may be found in \cite{[BB]} or \cite{[BL]}; we give this latter proof in Appendix 1 (see Proposition \ref{wtdlimit}). 

The weighted Fekete conjecture is that {\it for each $n$, let $x_1^{(n)},...,x_N^{(n)}$ be an $n-$th weighted Fekete set for $K$ and $w$ and let $\mu_n:= \frac{1}{N}\sum_{j=1}^N \delta_{x_j^{(n)}}$. Then}
$$
\mu_n \to \frac{1}{(2\pi)^d}(dd^c V_{K,Q}^*)^d \ \hbox{weak}-*.
$$
This is one of the applications of the Berman-Boucksom theory, to be presented in section \ref{sec:fob}. Indeed, we prove a slightly stronger statement on {\it asymptotically weighted Fekete sets} in Proposition \ref{asympwtdfek}.

A Rumely-type formula relating the weighted transfinite diameter 
$\delta^w(K)$ with $V_{K,Q}^*$ and $\tilde \rho_{K,Q}$, can be stated as follows:
\begin{equation}\label{truerumely}
-\log \delta^w(K)=\frac{1}{d(2\pi)^{d}}\int_{\CC^d} V_{K,Q}^*  (dd^c V_{K,Q}^*)^d\end{equation} $$ +\frac{1}{d(2\pi)^{d-1}}\int_{\PP^{d-1}}[\tilde \rho_{K,Q}- \tilde \rho_{T}]\sum_{j=0}^{d-1}(dd^c \tilde \rho_{K,Q}+\omega)^j \wedge (dd^c \tilde \rho_{T}+\omega)^{d-j-1}.$$
Here we assume $K$ is contained in the unit polydisk $U$. We will prove an equivalent version of (\ref{truerumely}) in section \ref{bf}. Rumely proved a weighted version of (\ref{unwtdtruerumely}): 
\begin{equation}\label{realtruerumely}
-\log d^w(K)=\frac{1}{d(2\pi)^{d-1}} \int_{\PP^{d-1}}[\tilde \rho_{K,Q}- \tilde \rho_{T}]\sum_{j=0}^{d-1}(dd^c \tilde \rho_{K,Q}+\omega)^j \wedge (dd^c \tilde \rho_{T}+\omega)^{d-j-1}
\end{equation}
where $d^w(K)$ is another weighted transfinite diameter, defined in \cite{bloomlev2}. The term $$\int_{\CC^d} V_{K,Q}^*  (dd^c V_{K,Q}^*)^d=\int_{K} Q(dd^cV^*_{K,Q})^{d}$$ in (\ref{truerumely}) arises due to the relationship between $\delta^w(K)$ and $d^w(K)$; this relationship,
\begin{equation} \label{dwdeltw}
\delta^w(K)=[\exp {(-\int_{K} Q(dd^cV^*_{K,Q})^{d})}]^{1/d}\cdot d^w(K),
\end{equation}
as well as the definition of $d^w(K)$, will be deferred to the appendices (Appendix 1 for the definition; Appendix 2 for (\ref{dwdeltw})). Indeed, to be precise, the versions (\ref{unwtdtruerumely}) and (\ref{realtruerumely}) of the Rumely formulas here are symmetrized versions of Rumely's original formulas, due to Demarco and Rumely \cite{demrum}. We prove the equivalence of these symmetrized versions with the ``originals'' in Appendix 2. 

\subsection {Bernstein-Markov properties.}\label{bernmarprop}

Given a compact set $E\subset \CC^d$ and a measure $\nu$ on $E$, we say that $(E,\nu)$ satisfies the Bernstein-Markov inequality if there is a strong comparability between $L^2$ and $L^{\infty}$ norms of holomorphic polynomials on $E$. Precisely, for all $p_n\in \mathcal P_n$, 	
$$||p_n||_E\leq  M_n||p_n||_{L^2(\nu)}  \ \hbox{with} \ \limsup_{n\to \infty} M_n^{1/n} =1.$$
If $E$ is regular, $(E,\mu_E)$ satisfies the Bernstein-Markov inequality (cf., \cite{NTVZ}). It is known if $(E,\nu)$ satisfies the Bernstein-Markov inequality that
\begin{equation}
\label{unwtdweakasym}
\lim_{n\to \infty} \frac{1}{2n} \log K_n^{\nu}(z,z) = V_{E}(z)
\end{equation}
locally uniformly on $\CC^d$ where 
$$
B_n^{\nu}(z):=K_n^{\nu}(z,z)=\sum_{j=1}^N |q_j^{(n)}(z)|^2
$$
is the $n-th$ {\it Bergman function} of $E,\nu$ (cf., \cite{bloomshiff}) and 
$$K_n^{\nu}(z,\zeta):=\sum_{j=1}^N q_j^{(n)}(z)\overline{q_j^{(n)}(\zeta)}$$
where $\{q_j^{(n)}\}_{j=1,...,N}$ is an orthonormal basis for $\mathcal P_n$ with respect to $L^2(\nu)$. 

More generally, for $K\subset \CC^d$ compact, $w=e^{- Q}$ an admissible weight function on $K$, and $\nu$ a measure on $K$, we say that the triple $(K,\nu,Q)$ satisfies a weighted Bernstein-Markov property if there is a strong comparability between $L^2$ and $L^{\infty}$ norms of {\it weighted} polynomials on $K$; precisely, for all $p_n\in \mathcal P_n$, 
writing $||w^np_n||_K:=\sup_{z\in K} |w(z)^np_n(z)|$ and $||w^np_n||_{L^2(\nu)}^2:=\int_K |p_n(z)|^2  |w(z)|^{2n} d\nu(z)$, 
$$||w^np_n||_K \leq M_n ||w^np_n||_{L^2(\nu)} \ \hbox{with} \ \limsup_{n\to \infty} M_n^{1/n} =1.$$
If $K$ is locally regular and $w$ is continuous, taking $\nu = (dd^cV_{K,Q})^d$ we have $(K,\nu,Q)$ satisfies a weighted Bernstein-Markov property (cf., \cite{bloom}). Now if $(K,\nu,Q)$ satisfies a weighted Bernstein-Markov property we have that
\begin{equation}
\label{weakasym}
\lim_{n\to \infty} \frac{1}{2n} \log K_n^{\nu,w}(z,z) = V_{K,Q}(z)
\end{equation}
locally uniformly on $\CC^d$ where 
\begin{equation}
\label{nthberg}
B_n^{\nu,w}(z):=K_n^{\nu,w}(z,z)w(z)^{2n}=\sum_{j=1}^N |q_j^{(n)}(z)|^2w(z)^{2n},
\end{equation}
is the $n-th$ {\it Bergman function} of $K,w,\nu$ (cf., \cite{bloompol}) and 
$$K_n^{\nu,w}(z,\zeta):=\sum_{j=1}^N q_j^{(n)}(z)\overline{q_j^{(n)}(\zeta)}.$$
Here, $\{q_j^{(n)}\}_{j=1,...,N}$ is an orthonormal basis for $\mathcal P_n$ with respect to the weighted $L^2-$norm $p\to ||w^np_n||_{L^2(\nu)}$. A sketch of the proof of (\ref{weakasym}) and/or (\ref{unwtdweakasym}) runs as follows: first, one shows that if 
$$\Phi_{K,Q,n}(z):=\sup \{|p(z)|: ||w^{\deg p}p||_K\leq 1, \ p \in \cup_{k=0}^n{\mathcal P}_k\},$$
then
$$\frac{1}{n}\log \Phi_{K,Q,n}\to V_{K,Q}$$
locally uniformly on $\CC^d$. Next, one verifies the inequality
$$\frac{[\Phi_{K,Q,n}(z)]^2}{N} \leq K_n^{\nu,w}(z,z) \leq N \cdot M_n [\Phi_{K,Q,n}(z)]^2.$$
The left-hand inequality follows simply from the reproducing property of the kernel function $K_n^{\nu,w}(z,\zeta)$; i.e., for any $p\in \mathcal P_n$,
$$p(z)=\int_K K_n^{\nu,w}(z,\zeta) p(\zeta)w(\zeta)^{2n}d\nu(\zeta),$$
and the Cauchy-Schwartz inequality; it is the right-side inequality which utilizes the weighted Bernstein-Markov property. Indeed, for an element $q_j^{(n)}\in {\mathcal P}_n$ in the orthonormal basis,  
$$||w^nq_j^{(n)}||_K \leq M_n \ \hbox{and} \ \frac{|q_j^{(n)}(z)|}{||w^nq_j^{(n)}||_K} \leq \Phi_{K,Q,n}(z)$$
imply
$$|q_j^{(n)}(z)|\leq M_n \Phi_{K,Q,n}(z)$$
so that
$$K_n^{\nu,w}(z,z)=\sum_{j=1}^N |q_j^{(n)}(z)|^2\leq N\cdot M_n^2 [\Phi_{K,Q,n}(z)]^2.$$
This was proved in the unweighted case; i.e., (\ref{unwtdweakasym}), by Bloom and Shiffman \cite{bloomshiff} and in the general (weighted) case; i.e., (\ref{weakasym}), by Bloom \cite{bloompol}. 

From the local uniform convergence in (\ref{weakasym}) follows the weak-* convergence of the Monge-Ampere measures
$$[dd^c\frac{1}{2n} \log K_n^{\nu,w}(z,z)]^d \to (dd^cV_{K,Q}^*)^d \ \hbox{weak-}*.$$
One of the main results of this paper is a much stronger version of ``Bergman asymptotics'' to be proved in Corollary \ref{strongasymp}: 
if $(K,\nu,w)$ satisfies a weighted Bernstein-Markov inequality, then 
$$\frac{1}{N}B_n^{\nu,w} d\nu \to \mu_{K,Q}:=\frac{1}{(2\pi)^d}(dd^cV_{K,Q}^*)^d \ \hbox{weak-}*.$$
This was proved in the one variable case ($d=1$) in \cite{blchrist}.

We will have occasion to consider the case of unbounded sets. If $E\subset \CC^d$ is closed and unbounded, $w=e^{- Q}$ is an admissible weight on $E$ and $\mu$ is a positive measure carried by $E$ we say the triple $(E,\mu,Q)$ satisfies a weighted Bernstein-Markov property if there exists an integer $n_0$ such that for each $n>n_0$ and all $p_n\in \mathcal P_n$, we have 
$$\int_{E}|p_n|^2w^{2n}d\mu< +\infty$$
and
$$\sup_{E}|p_n w^n|\leq M_n [\int_{E}|p_n|^2w^{2n}d\mu]^{1/2}
 \ \hbox{with} \ \limsup_{n\to \infty} M_n^{1/n} =1.$$
(cf., (\ref{weightedleb}) in subsection \ref{sec:bergcd}). Indeed, an important special case of Corollary \ref{strongasymp}, Theorem \ref{bermancn}, involves strong Bergman asymptotics with $E=\CC^d$.

Which pairs $(K,\mu)$ satisfy a Bernstein-Markov property? Which triples $(K,\mu,Q)$ satisfy a weighted Bernstein-Markov property? As mentioned, if $K$ is regular ($V_K$ is continuous), the extremal measure $\mu_K:= \frac{1}{(2\pi)^d}(dd^cV_K)^d$ works. In the weighted case, if $K$ is locally regular and $w$ is continuous (so that $V_{K,Q}$ is continuous), $\mu_{K,Q}:=\frac{1}{(2\pi)^d}(dd^cV_{K,Q})^d$ works. These claims are a special case of a more general result. We begin with a definition.

\begin{definition} \label{determined} Let $K\subset \CC^d$ be compact and let $\mu$ be a probability measure on $K$. We say that $\mu$ is a {\it determined measure} for $K$ if for each Borel set $E\subset K$ with $\mu(E) = \mu(K)$, we have $V_E^* = V_K^*$. If $w$ is an admissible weight on $K$, we say that $\mu$ is a {\it determined measure} for $K,w$ if for each Borel set $E\subset K$ with $\mu(E) = \mu(K)$, we have $V_{E,Q}^* = V_{K,Q}^*$.
\end{definition}

\noindent Equivalently, if $u\in L(\CC^d)$ and $u\leq Q$ $\mu-$a.e. on $K$ then $u\leq V_{K,Q}^*.$

\begin{proposition}
Let $K\subset \CC^d$ be compact and let $w$ be an admissible weight on $K$. Then $\mu_{K,Q}$ is a determined measure for $K,w$.
\end{proposition}

\begin{proof}
Fix $u\in L(\CC^d)$ with $u\leq Q$ $\mu_{K,Q}-$a.e. on $K$. By the  $L^+(\CC^d)$ version of the domination principle, Proposition \ref{lplusdom}, 
$u\leq V_{K,Q}^*.$
\end{proof}

The claim in the paragraph before Definition \ref{determined} follows from the general result below (cf., \cite{sicdet}).

\begin{proposition}\label{detdprop}
If $V_{K,Q}$ is continuous, and if a probability measure $\mu$ is a determined measure for $K,w$, then $(K,\mu,Q)$ satisfies the Bernstein-Markov property.
\end{proposition}

\begin{proof} Suppose that $(K,\mu,Q)$ does not satisfy the Bernstein-Markov property. Then there is an $\epsilon >0$ and a sequence of polynomials $\{p_k\}$ with deg$p_k=k$ satisfying
\begin{equation}\label{l2norm}
||p_ke^{-kQ}||_{L^2(\mu)}=(1+\epsilon)^{-k}
\end{equation}
while
\begin{equation}\label{linfnorm}
||p_ke^{-kQ}||_K=\sup_{z\in K}|p_k(z)e^{-kQ(z)}|\geq k(1+\epsilon)^{k}.
\end{equation}
For $m>0$, define
$$K_m:=\{z\in K: \sup_{k}|p_k(z)e^{-kQ(z)}|\leq m\}.$$
Then clearly for each $k$
$$\frac{1}{k}\log \frac{|p_k|}{m}\leq V_{K_m,Q}\leq V_{K_m,Q}^*$$
so that
\begin{equation}\label{forcont}|p_k(z)e^{-kQ(z)}|\leq me^{k[V_{K_m,Q}^*(z)-Q(z)]} \ \hbox{on} \ K.
\end{equation}
Let 
$$K':=\cup_{m>0}K_m=\{z\in K: \sup_{k}|p_k(z)e^{-kQ(z)}|<  +\infty\}.$$
Then $V_{K_m,Q}^*\downarrow V_{K',Q}^*$ (here we simply restrict $Q$ to $K_m,K'$). Since $K'\subset K$ and $V_{K,Q}$ is continuous, $V_{K,Q}=V_{K,Q}^*\leq V_{K',Q}^*$. We show equality holds; i.e.,
\begin{equation}\label{equality}
V_{K,Q}= V_{K',Q}^*.
\end{equation}

To verify (\ref{equality}), we first claim that $\mu(K\setminus K')=0$. For by (\ref{l2norm}), $\sum_k |p_k(z)e^{-kQ(z)}|^2$ converges in $L^1(\mu)$ and hence $|p_k(z)e^{-kQ(z)}|\to 0$ $\mu-$a.e. as $k\to \infty$; hence $\sup_{k}|p_k(z)e^{-kQ(z)}|< +\infty$ $\mu-$a.e., proving the claim. Now since by hypothesis, $\mu$ is a determined measure for $K,w$, (\ref{equality}) follows.

Finally, we conclude that the usc functions $\{V_{K_m,Q}^*\}$ decrease to $V_{K,Q}$ pointwise on the compact set $K$; since $V_{K,Q}$ is continuous, we conclude, by Dini's theorem, that $V_{K_m,Q}^*\to V_{K,Q}$ uniformly on $K$. In particular, for $m$ sufficiently large, 
$$e^{V_{K_m,Q}^*- Q}\leq 1+\epsilon \ \hbox{on} \ K.$$
Fixing such an $m$ and using (\ref{forcont}), we have
$$|p_k(z)e^{-kQ(z)}|\leq m(1+\epsilon)^k \ \hbox{on} \ K.$$
This contradicts (\ref{linfnorm}) for $k$ sufficiently large.
\end{proof}

\subsection {Gram matrices and optimal measures.} \label{gramsec} We begin with some motivation from \cite{bblw}. Suppose that $K\subset \CC^d$ is compact and non-pluripolar. Note that if we write, for points $\zeta_1,...,\zeta_N\in K$, 
$$V_n:=[e_i(\zeta_j)]_{i,j=1,...,N},$$
then 
$$\frac{1}{N}V_nV_n^* =[\frac{1}{N} \sum_{k=1}^Ne_i(\zeta_k)\overline{e_j(\zeta_k)}]_{i,j=1,...,N}$$
which is a {\it Gram matrix} of inner products using the measure $\mu_n:=\frac{1}{N}\sum_{k=1}^N \delta_{\zeta_k}$ with respect to the monomial basis $\{e_1,...,e_N\}$. More generally, let  $\mu$ be a probability measure on $K.$ We assume that $\mu$ is non-degenerate on ${\mathcal P}_n$. This means that with the associated inner-product
\begin{equation}\label{ip}
\langle f,g\rangle_\mu:=\int_K f\overline{g}d\mu 
\end{equation}
and $L^2(\mu)$ norm, $\|f\|_{L^2(\mu)}=\sqrt{\langle f,f\rangle_\mu}$, we have $\|p\|_{L^2(\mu)}=0$ for $p\in {\mathcal P}_n$ implies that $p=0$. It follows that $\mu$ is
non-degenerate on ${\mathcal P}_n$ if and only if supp$(\mu)$ is not contained in an algebraic variety of degree $n$. 
Then ${\mathcal P}_n$ equipped with the inner-product (\ref{ip}) is a finite dimensional Hilbert
space of dimension $N$ and, for any fixed basis $\beta_n:=\{p_1,...,p_N\}$ of ${\mathcal P}_n$, we can form the {\it Gram matrix of $\mu$ with respect to the basis $\beta_n$}:
$$G_n^{\mu}(\beta_n):=[\langle p_i,p_j\rangle_\mu]_{i,j=1,...,N}.$$
If we change the
basis $\beta_n=\{p_1,...,p_N\}$ to $C_n:=\{q_1,...,q_N\}$ by $p_i=\sum_{j=1}^Na_{ij}q_j,$ 
then the Gram matrices transform (see e.g. \cite{D}, \S8.7) by
\begin{equation}\label{grambasis}
G_n^{\mu}(\beta_n)=AG_n^{\mu}(C_n)A^*
\end{equation}
where $A=[a_{ij}]\in \CC^{N\times N}.$ Hence the following definition is independent of the basis chosen.

\begin{definition}
\label{opt}
If a probability measure $\mu$ has the property that
\begin{equation}\label{(a)}
{\rm det}(G_n^{\mu'})\le {\rm det}(G_n^{\mu})
\end{equation}
for all other probability measures $\mu'$ on $K$ then $\mu$ is said to be an optimal measure of degree $n$ for $K$.
\end{definition}

We make some observations. 
We may compare the uniform norm on $K$ with the $L^2(\mu)$ norm for $p\in{\mathcal P}_n.$ For $\mu$ a probability measure we have
\[\|p\|_{L^2(\mu)}\le \|p\|_K,\]
and since ${\mathcal P}_n$ is finite dimensional there is a constant $C=C(n,\mu,K)$ such that
\[\|p\|_K\le C \|p\|_{L^2(\mu)}.\]
The {\it best} constant $C$ is given by 
\[C=\sup_{p\in {\mathcal P}_n,\,p\ne0}{\|p\|_K\over \|p\|_{L^2(\mu)}}=\max_{z\in K}\sqrt{B_n^\mu(z)}\]
where recall
\begin{equation}\label{Kn}
B_n^\mu(z):=\sum_{j=1}^N|q_j(z)|^2
\end{equation}
is the $n-$th Bergman function for $K,\mu$ and 
$Q_n=\{q_1,q_2,\cdots,q_N\}$ is an orthonormal basis for ${\mathcal P}_n.$ 

It is natural to ask among all probability measures on $K,$ which one provides the smallest such factor, and this leads to an equivalent characterization of optimal measures: {\it Suppose that the probability measure $\mu$ has the property that 
\[\max_{z\in K}\sqrt{B_n^\mu(z)}\le \max_{z\in K}\sqrt{B_n^{\mu'}(z)}\]
for all other probability measures $\mu'$ on $K.$ Then 
$\mu$ is an optimal measure of degree $n$ for $K.$}

Indeed, for {\it any} probability measure $\mu,$  $\displaystyle{\int_K B_n^\mu(z) d\mu=N},$ so that
\[\max _{z\in K} B_n^\mu(z)\ge N.\]
We will show that for an optimal measure 
\begin{equation}\label{max=N}
\max _{z\in K} B_n^\mu(z) = N.
\end{equation}
In fact, we will verify a more general (weighted) version of this (Proposition \ref{KW}). 

We make an observation which will be useful. With a basis $\beta_n=\{p_1,...,p_N\}$ of ${\mathcal P}_n$, if we write
\begin{equation}\label{P}
P(x)=\left[\begin{array}{c}p_1(x)\cr p_2(x)\cr\cdot\cr\cdot\cr p_N(x)\end{array}\right]\in \CC^{ N}
\end{equation}
then 
\begin{equation}
\label{useful}
P(x)^*\bigl(G_n^\mu(\beta_n)\bigr)^{-1}P(x)=B_n^\mu(x).
\end{equation}
To see this, $G:=G_n^\mu(\beta_n)$ and $G^{-1}$ are positive definite, Hermitian matrices; hence $G^{1/2}, \ G^{-1/2}:=(G^{-1})^{1/2}$ exist; writing $P:=P(x)$, we have
$$P^*G^{-1}P = P^*G^{-1/2}G^{-1/2}P=(G^{-1/2}P)^*G^{-1/2}P.$$
To verify that the right-hand-side yields $B_n^\mu(x)$, note that since $G=\int_K PP^*d\mu$, the polynomials $\{\tilde p_1,\tilde p_2,\cdots, \tilde p_N\}$ defined by
\begin{equation}
G^{-1/2}P:=\left[\begin{array}{c}\tilde p_1(x)\cr 
\tilde p_2(x)\cr\cdot\cr\cdot\cr \tilde p_N(x)\end{array}\right]\in \CC^{ N}
\end{equation}
form an {\it orthonormal} basis for ${\mathcal P}_n$ in $L^2(\mu)$: for
$$\int_K G^{-1/2}P \cdot (G^{-1/2}P)^*d\mu =  G^{-1/2}\bigl[\int_K PP^* d\mu \bigr] G^{-1/2}$$
$$ = G^{-1/2} G G^{1/2}=I,$$
the $N\times N$ identity matrix. Thus 
$$B_n^\mu(x)= \sum_{j=1}^N |\tilde p_j(x)|^2 = (G^{-1/2}P)^*G^{-1/2}P.$$

We turn to weighted versions of Gram matrices and optimal measures.  Let $K\subset \CC^d$ be compact and non-pluripolar. Fix $\mu$ a probability measure on $K$ and $w$ an admissible weight on $K$. For each $n$ we have the weighted inner product of degree $n$
\begin{equation}\label{Wip}
\langle f,g\rangle_{\mu,w}:=\int_K f(z)\overline{g(z)}w(z)^{2n}d\mu.
\end{equation}
Again fixing a basis $\beta_n=\{p_1,p_2,\cdots, p_N\}$ of ${\mathcal P}_n$ we form the Gram matrix
$$
G_n^{\mu,w}=G_n^{\mu,w}(\beta_n):=[\langle p_i,p_j\rangle_{\mu,w}]\in\CC^{N\times N}
$$
and the associated $n-$th Bergman function
\begin{equation}\label{WKn}
B_n^{\mu,w}(z):=\sum_{j=1}^N|q_j(z)|^2w(z)^{2n}
\end{equation}
(see (\ref{nthberg})) 
where $Q_n=\{q_1,q_2,\cdots,q_N\}$ is an orthonormal basis for ${\mathcal P}_n$ with respect
to the inner-product (\ref{Wip}). If a probability measure $\mu$ has the property that
\begin{equation}\label{wa} {\rm det}(G_n^{\mu',w})\le {\rm det}(G_n^{\mu,w})
\end{equation}
for all other probability measures $\mu'$ on $K$ then $\mu$ is said to be an optimal measure of
degree $n$ for $K$ and $w.$

From Lemma 2.1 of \cite{KS}, Chapter X, the set of matrices
\[\{G_n^{\mu,w}\,:\,\mu\,\,\hbox{is a probability measure on}\,\,K\}\]
is compact (and convex). Hence {\it an optimal measure of
degree $n$ for $K$ and $w$ always exists}. They need not be unique. An equivalent characterization of optimality is given by the Kiefer-Wolfowitz Equivalence Theorem \cite {[KW]}.

\begin{proposition}
\label{KW}
Let $w$ be an admissible weight on $K.$ A probability measure $\mu$
is an optimal measure of degree $n$ for $K$ and $w$ if and only if
\begin{equation} \label{wb} \max_{z\in K}B_n^{\mu,w}(z)=N.
\end{equation}
\end{proposition} 

\begin{proof} We give a proof of the equivalence of conditions (\ref{wa}) and (\ref{wb}) following [Bo] (see also \cite{KS}, Theorem 2.1, Chapter X).  First, with $P$ defined as in (\ref{P}), the proof of (\ref{useful}) gives 
\begin{equation}\label{OtherKn}
w^{2n}P^*(G_n^{\mu,w})^{-1}P = B_n^{\mu,w}.
\end{equation}

A computation shows that 
$$\mu \to \log \det G_n^{\mu,w}$$
is concave on the space of probability measures; i.e., if 
$$h(t):= \log \det G_n^{t\mu_1+(1-t)\mu_2,w}$$
for two probability measures $\mu_1$ and $\mu_2$, then $h''(t)\leq 0$. Indeed, since $G_n^{\mu_1,w}$ and $G_n^{\mu_2,w}$ are positive definite Hermitian matrices, we can find a nonsingular matrix $A$ so that  
$$A^*G_n^{\mu_1,w}A = {\rm diag}[a_1 \cdots a_N]  \ \hbox{and} \ A^*G_n^{\mu_2,w}A = {\rm diag}[b_1 \cdots b_N] $$
Then
$$\det G_n^{t\mu_1+(1-t)\mu_2,w}= |\det A|^2 \det {\rm diag}[ta_1 +(1-t)b_1\cdots ta_N +(1-t)b_N]$$
and, computing,
$$ \frac{d^2}{dt^2}  \log \det G_n^{t\mu_1+(1-t)\mu_2,w}=-\sum_{j=1}^N \frac{(b_j-a_j)^2}{[ta_j+(1-t)b_j]^2}\leq 0.$$

Hence $\mu_1$ is optimal in the sense of (\ref{wa}) if and only if $h'(t)\leq 0$ for all $\mu_2$. Computing this derivative one sees that $\mu_1$ is optimal in the sense of (\ref{wa}) if and only if
\begin{equation}
\label{kwmain}
{\rm trace} \bigl[\bigl(G_n^{\mu_1,w}\bigr)^{-1}G_n^{\mu_2,w}\bigr]=\int_K w^{2n}P^*(G_n^{\mu_1,w})^{-1}Pd\mu_2 =\int B_n^{\mu_1,w} d\mu_2\leq N
\end{equation}
for all $\mu_2$. Here we use (\ref{OtherKn}) and the fact that, for an $N\times N$ matrix $A$, an $N\times 1$ matrix $B$, and a $1\times N$ matrix $C$, 
$${\rm trace} (ABC) = {\rm trace} (CAB)= CAB;$$
thus, writing $G_j:= G_n^{\mu_j,w}$ and using $G_2=\int_Kw^{2n}PP^*d\mu_2$,  
$${\rm trace} \bigl[\bigl(G_n^{\mu_1,w}\bigr)^{-1}G_n^{\mu_2,w}\bigr]$$
$$= {\rm trace} \bigl[G_1^{-1}\int_Kw^{2n}PP^*d\mu_2\bigr]=\int_Kw^{2n}P^*G_1^{-1}Pd\mu_2.$$
Taking $\mu_2$ to be a point mass at a point $z \in K$ in (\ref{kwmain}) gives $B_n^{\mu_1,w}(z)\leq N$; then taking $\mu_2 = \mu_1$ gives $\int B_n^{\mu_1,w} d\mu_1= N$ by  orthonormality. This proves the equivalence of (\ref{wa}) and (\ref{wb}).
\end{proof}

Indeed, the end of this argument yields the following key property of optimal measures.
\begin{lemma}\label{MaxIsN}
Suppose that $\mu$ is optimal for $K$ and $w.$ Then
\[B_n^{\mu,w}(z)=N,\quad a.e.\,\, [\mu].\]
\end{lemma}

\begin{proof} On the one hand
\[\max_{z\in K}B_n^{\mu,w}(z)=N\]
while on the other hand, again by orthonormality of the $q_j,$
\[ \int_K B_n^{\mu,w}\, d\mu=\int_K\sum_{j=1}^N|q_j(z)|^2w(z)^{2n}\,d\mu(z)=N,\]
and the result follows. 
\end{proof}

We relate the (weighted) Bernstein-Markov property with (weighted) Gram determinants and (weighted) transfinite diameter.

\begin{proposition}
\label{weightedtd}
Let $K\subset \CC^d$ be a compact set and let $w$ be an admissible weight function on $K$. If $\nu$ is a measure on $K$ with $(K,\nu,Q)$ satisfying a weighted Bernstein-Markov property, then 
$$\lim_{n\to \infty} \frac{d+1}{2dnN} \cdot \log \det G_n^{\nu,w} = \log \delta^{w}(K).$$
\end{proposition}

Here $G_n^{\nu,w} =[\int_K p_i \bar p_jw^{2n}d\mu]_{i,j=1,...,N}$ where $\{p_1,...,p_N\}$ is either the standard basis monomials $\{e_1,...,e_N\}$ for $\mathcal P_n$ or, e.g., an orthogonal basis $\{r_1,...,r_N\}$ obtained by applying Gram-Schmidt to the standard basis. Recall that $\frac{dnN}{d+1}$ is the sum of the degrees of a set of basis monomials for $ \mathcal P_n$ and the weighted $n-$th order diameter of $K,w$ is
$$\delta^{w,n}(K):=\bigl(\max_{x_1,...,x_N\in K} \det [e_i(x_j) w(x_1)^n\cdots w(x_N)^n]_{i,j=1,...,N}\bigr)^{\frac{d+1}{dnN}}$$
$$:= \bigl(\max_{x_1,...,x_N\in K} |VDM(x_1,...,x_N)| w(x_1)^{n} \cdots w(x_N)^{n}\bigr)^{\frac{d+1}{dnN}}.$$

\begin{proof} Note first that $\det G_n^{\nu,w} =\prod_{j=1}^N ||r_j||^2_{L^2(w^{2n}\nu)}$ where $\{r_1,...,r_N\}$ are an orthogonal basis of 
$ \mathcal P_n$ obtained by applying Gram-Schmidt to the standard basis monomials of $ \mathcal P_n$. Defining
$$Z_n:=Z_n(K,w,\mu)$$
$$:= \int_K \cdots \int_K |VDM(z_1,...,z_N)|^2 w(z_1)^{2n} \cdots w(z_N)^{2n}d\mu(z_1) \cdots d\mu(z_N)$$
we show that
$$\lim_{n\to \infty} Z_n^{\frac{d+1}{2dnN}} =  \delta^w(K).$$
To see this, clearly 
\begin{equation}
\label{bound1}
Z_n \leq \delta^{w,n}(K)^{\frac{2dnN}{d+1}}\nu(K)^N.
\end{equation}
On the other hand, taking points $x_1,...,x_N$ achieving the maximum in $\delta^{w,n}(K)$, we have, upon applying the weighted Bernstein-Markov property to the weighted polynomial $$z_1 \to VDM(z_1,x_2...,x_N) w(z_1)^{n} \cdots w(x_N)^{n},$$
$$\delta^{w,n}(K)^{\frac{2dnN}{d+1}}=|VDM(x_1,...,x_N)|^2 w(x_1)^{2n} \cdots w(x_N)^{2n}$$
$$\leq M_n^2 \int_K \cdots \int_K |VDM(z_1,x_2...,x_N)|^2 w(z_1)^{2n} \cdots w(x_N)^{2n}d\mu(z_1).$$
Repeating this argument in each variable we obtain
\begin{equation}
\label{bound2}
\delta^{w,n}(K)^{\frac{2dnN}{d+1}}\leq M_n^{2N} Z_n.
\end{equation}
Note that (\ref{bound1}) and (\ref{bound2}) give
$$Z_n \leq \delta^{w,n}(K)^{\frac{2dnN}{d+1}}\nu(K)^N \leq \nu(K)^N M_n^{2N} Z_n.$$
Since $[\nu(K)^N M_n^{2N}]^{\frac{d+1}{2dnN}}\to 1$, using (\ref{deltaw}) 
$$\lim_{n\to \infty} Z_n^{\frac{d+1}{2dnN}}$$
exists and equals
$$\lim_{n\to \infty} \delta^{w,n}(K)^{\frac{d+1}{dnN}}.$$

Using elementary row operations in $|VDM(z_1,...,z_N)|^2$ in the integrand of $Z_n$, we can replace the monomials $\{e_j\}$ by the orthogonal basis $\{r_1,...,r_N\}$ and obtain
\begin{equation}\label{zn}
Z_n = N!\prod_{j=1}^N ||r_j||^2_{L^2(w^{2n}\nu)}.
\end{equation}
Putting everything together gives the result.
\end{proof}

The quantity $Z_n$ occurring in the proof of Proposition \ref{weightedtd} is called the $n-$th {\it free energy} of $K,\mu$. A similar argument shows that the Gram determinants associated to a sequence of weighted optimal measures also converges to $\delta^w(K)$. In this proposition, we again compute the Gram determinant with respect to the standard
basis monomials.  

\begin{proposition}\label{tfd}
Let $K$ be compact and $w$ an admissible weight function. For $n=1,2,...$, let $\mu_n$ be an optimal measure of order $n$ for $K$ and $w.$ Then
\[\lim_{n\to\infty}{\rm det}(G_n^{\mu_n,w})^{\frac{d+1}{2dnN}}=\delta^w(K).\]
\end{proposition}
\begin{proof}
Replacing the the standard
basis monomials by an orthogonal basis $\{r_1,...,r_N\}$ with respect to $L^2(w^{2n}\mu_n)$, the formula (\ref{zn}) says that
$$\int_{K^N}|VDM(z_1,\cdots,z_N)|^2w(z_1)^{2n}\cdots w(z_N)^{2n} d\mu_n(z_1)\cdots d\mu_n(z_N)$$
$$=N!\,{\rm det}(G_n^{\mu_n,w}).$$ 
It follows, since $\mu_n$ is a probability measure, that
\begin{equation}
{\rm det}(G_n^{\mu_n,w})\le {1\over N!}(\delta_n^w(K))^{\frac{2dnN}{d+1}}.
\end{equation}
Now if $f_1,f_2,\cdots,f_N\in K$ are weighted Fekete points of order $n$ for $K,$
i.e., points in $K$ for which 
\[|VDM(z_1,\cdots,z_N)|w^n(z_1)w^n(z_2)\cdots w^n(z_N)\]
is maximal, then the discrete measure 
\begin{equation}
\nu_n={1\over N}\sum_{k=1}^N \delta_{f_k}
\end{equation}
is a candidate for an optimal measure of order $n$; hence
\[ {\rm det}(G_n^{\nu_n,w})\le {\rm det}(G_n^{\mu_n,w}).\]
But
\begin{eqnarray*}
{\rm det}(G_n^{\nu_n,w})&=&{1\over N^N}|VDM(f_1,\cdots,f_N)|^2w(f_1)^{2n}w(f_2)^{2n}\cdots w(f_N)^{2n} \\
&=&\left(\max_{z_i\in K}|VDM(z_1,\cdots,z_N)|w^n(z_1)w^n(z_2)\cdots w^n(z_N)\right)^2 \\
&=&{1\over N^N}(\delta_n^w(K))^{\frac{2dnN}{d+1}}
\end{eqnarray*}
so that
\[{1\over N^N}(\delta_n^w(K))^{\frac{2dnN}{d+1}} \le {\rm det}(G_n^{\mu_n,w}) \]
and the result follows from (\ref{deltaw}) . 
\end{proof}

We will see in section \ref{sec:fob} that if $\{\mu_n\}$ are a sequence of optimal measures for $K,w$ then $\mu_n \to \mu_{K,Q}$ weak-*.

\subsection {Ball volumes.} Given a (complex) $N-$dimensional vector space $V$, and two subsets $A,B$ in $V$, we write
$$[A:B]:=\log \frac{vol(A)}{vol(B)}.$$
Here, ``vol'' denotes any (Haar) measure on $V$ (as taking the ratio makes $[A:B]$ independent of this choice). In particular, if $V$ is equipped with two Hermitian inner products $h,h'$, and $B,B'$ are the corresponding unit balls, then a linear algebra exercise shows that 
\begin{equation}
\label{gramvolume}
[B:B'] = \log \det [h'(e_i,e_j)]_{i,j=1,...,N}
\end{equation}
where $e_1,...,e_N$ is an $h-$orthonormal basis for $V$. In other words, $[B:B'] $ is a {\it Gram determinant} with respect to the $h'$ inner product relative to the $h-$orthonormal basis. Indeed, $[B:B']$ is independent of the $h-$orthonormal basis chosen for $V$, as can be seen by (\ref{grambasis}). 

{\it We will generally take $V=\mathcal P_n$ and our subsets to be unit balls with respect to norms on $\mathcal P_n$}. In particular, let $\mu$ be a probability measure on a compact set $K\subset \CC^d$ which is non-degenerate on $\mathcal P_n$. Observe that for the unit torus $T$, 
$$V_T(z_1,...,z_d)=\max_{j=1,...,d} \log^+|z_j|,$$
and the standard basis monomials $\beta_n:=\{e_1,...,e_N\}$ form an orthonormal basis for $\mathcal P_n$ with respect to $\mu_T:=\frac{1}{(2\pi)^d}(dd^cV_T)^d$, which is the standard Haar measure on $T$. Letting
$$B_n=\{p_n\in \mathcal P_n: ||p_n||_{L^2(\mu)}\leq 1 \}$$
and
$$B_n'=\{p_n\in \mathcal P_n: ||p_n||_{L^2(\mu_T)}\leq 1 \}$$
be $L^{2}-$balls in $\mathcal P_n$, we have
$$[B_n:B_n'] = \log \det G_n^{\mu}(\beta_n).$$
We will also use $L^{\infty}-$balls in $\mathcal P_n$; e.g., with
$$\tilde B_n=\{p_n\in \mathcal P_n: ||p_n||_{K}\leq 1 \}$$
and
$$\tilde B_n'=\{p_n\in \mathcal P_n: ||p_n||_{T}\leq 1 \},$$
we consider $[\tilde B_n:\tilde B_n']$. If $(K,\mu)$ satisfies a Bernstein-Markov property, then the asymptotics ($n\to \infty$) of the sequence of ball-volume ratios
\begin{equation}\label{l2linf}
\frac{1}{2nN}[B_n:B_n'] \ \hbox{and} \ \frac{1}{2nN}[\tilde B_n:\tilde B_n']\end{equation}
will be the same (note $(T,\mu_T)$ satisfies a Bernstein-Markov property) and the limit is related to the relative energy of $V_K^*$ with $V_T$, to be defined in the next section. This is a special (unweighted) case of the main result of these notes, Theorem \ref{keycn}, and will immediately provide us with a version of the Rumely formula.

For future use we note that $[A:B]=-[B:A]$; the ball volume ratios trivially satisfy the cocycle  condition:
$$[A:B]+[B:C]+[C:A]=0;$$
and they are ``monotone'' in the first slot in the sense that for any $B\subset {\mathcal P}_n$, if $E\subset \CC^d$ is closed with admissible weights $Q_1\leq Q_2$ and 
$$\mathcal B^{\infty}(E,nQ_i):=\{p_n\in \mathcal P_n: ||p_ne^{-nQ_i}||_{E}\leq 1 \}, \ i=1,2$$
then
\begin{equation}\label{bvmon}[\mathcal B^{\infty}(E,nQ_1):B]\leq [\mathcal B^{\infty}(E,nQ_2):B].
\end{equation}
The analogous properties will also hold for the relative energies,  defined next.

\section {Energy.}\label{enprop}

We define the {\it Monge-Ampere energy ${\mathcal E} (u,v)$ of $u$ relative to $v$} for $u,v \in L^+(\CC^d)$ as 
\begin{equation}
\label{relendef}
\mathcal E (u,v):= \int_{\CC^d} (u-v)\sum_{j=0}^d (dd^cu)^j\wedge (dd^cv)^{d-j}.
\end{equation}
This will be utilized in the version of the Rumely formula given in section \ref{bf}.  To see the relation with (\ref{truerumely}),  we begin with the following Bedford-Taylor integral formula (Theorem 5.5 of \cite{[BT]}). Given $v_1,v_2,u_1,...,u_{d-1} \in L^+(\CC^d)$ we have
\begin{equation} \label{bedtayfor}
\int_{\CC^d}(v_1dd^cv_2 - v_2dd^cv_1)\wedge dd^cu_1 \wedge \cdots \wedge dd^cu_{d-1}$$ $$ = 2\pi \int_{\PP^{d-1}}(\tilde \rho_{v_1}- \tilde \rho_{v_2})(dd^c \tilde \rho_{u_1}+\omega) \wedge \cdots \wedge (dd^c \tilde \rho_{u_{d-1}}+\omega). 
\end{equation}
Here, given $u\in L^+(\CC^d)$, recall that 
$$\tilde \rho_{u}(z):=\limsup_{|\lambda| \to \infty} [u(\lambda z)- \log |\lambda z|]$$
is the (projectivized) Robin function of $u$ (cf., (\ref{projrob})) and $\omega$ is the standard K\"ahler form on $\PP^{d-1}$. 

Using (\ref{bedtayfor}) we prove an important formula relating $\mathcal E (u,v)$ with a projectivized version.

\begin{proposition}
\label{reduction}
Let $u,v\in L^+(\CC^d)$ with supp$(dd^cu)^d$ and supp$(dd^cv)^d$ compact. Then
\begin{equation}\label{redform}
\mathcal E (u,v)=\int_{\CC^d}u(dd^cu)^d - \int_{\CC^d}v(dd^cv)^d 
\end{equation}
$$+ 2\pi \int_{\PP^{d-1}}(\tilde \rho_{u}- \tilde \rho_{v})\sum_{j=0}^{d-1}(dd^c \tilde \rho_{u}+\omega)^j \wedge (dd^c \tilde \rho_{v}+\omega)^{d-j-1}.$$
\end{proposition}

\begin{proof} We begin with the algebraic formula
\begin{equation}\label{algform}(dd^cu)^d -(dd^cv)^d=dd^c(u-v)\wedge T
\end{equation}
where
$$T:=\sum_{j=0}^{d-1}(dd^c u)^j \wedge (dd^c v)^{d-j-1}.$$
Then we can write
$$\mathcal E (u,v)=\int_{\CC^d}(u-v)\bigl[(dd^cu)^d+dd^cv\wedge T\bigr].$$
We now use the hypothesis that supp$(dd^cu)^d$ and supp$(dd^cv)^d$ are compact to write this as 
$$=\int_{\CC^d}u(dd^cu)^d - \int_{\CC^d}v(dd^cu)^d+ \int_{\CC^d}(u-v)dd^cv\wedge T$$
$$=\int_{\CC^d}u(dd^cu)^d - \int_{\CC^d}v(dd^cv)^d + \int_{\CC^d}v\bigl[(dd^cv)^d -(dd^cu)^d\bigr]$$
$$+\int_{\CC^d}(u-v)dd^cv\wedge T.$$
Using the algebraic formula (\ref{algform}), we obtain
$$\mathcal E (u,v)=\int_{\CC^d}u(dd^cu)^d - \int_{\CC^d}v(dd^cv)^d$$
$$ + \int_{\CC^d}vdd^c(v-u)\wedge T+\int_{\CC^d}(u-v)dd^cv\wedge T$$
$$=\int_{\CC^d}u(dd^cu)^d - \int_{\CC^d}v(dd^cv)^d +\int_{\CC^d}(udd^cv -vdd^cu)\wedge T.$$
The result now follows from (\ref{bedtayfor}).
\end{proof}

Now suppose, as in the setting of (\ref{truerumely}), that $K$ is contained in the unit polydisk $U$. Then, with $T$ the unit torus,
applying Proposition \ref{reduction}, 
$$\mathcal E(V_{K,Q}^*,V_T)=\int_{\CC^d} (V_{K,Q}^* -V_T)\sum_{j=0}^d (dd^c V_{K,Q}^*)^j\wedge (dd^cV_T)^{d-j}$$
$$=\int_{\CC^d} V_{K,Q}^*  (dd^c V_{K,Q}^*)^d$$ $$ +2\pi \int_{\PP^{d-1}}[\tilde \rho_{K,Q}- \tilde \rho_{T}]\sum_{j=0}^{d-1}(dd^c \tilde \rho_{K,Q}+\omega)^j \wedge (dd^c \tilde \rho_{T}+\omega)^{d-j-1}$$
where we have used $\int_{\CC^d} V_T (dd^c V_{K,Q}^*)^d=0$ by the assumption that $K\subset U$. To prove (\ref{truerumely}), we need to verify, then, an ``energy version'' of Rumely's formula; i.e.,
$$-\log \delta^w(K)=  \frac{1}{d(2\pi)^d}\mathcal E(V^*_{K,Q},V_T).$$
This we will do in section \ref{bf}.  

Next we prove a corollary of the Bedford-Taylor formula. We begin with four functions $A,B,C,D\in L^+(\CC^d)$ and let $u_1,...,u_{d-1} \in L^{+}(\CC^d)$ (so that $T:=dd^cu_1 \wedge \cdots \wedge dd^cu_{d-1}$ is a positive closed $(d-1,d-1)$ current). Then
\begin{equation}
\label{ibyp1}
\int_{\CC^d}(A-B)(dd^cC - dd^cD)\wedge dd^cu_1 \wedge \cdots \wedge dd^cu_{d-1}
\end{equation} 
$$ =\int_{\CC^d}(C-D)(dd^cA - dd^cB)\wedge dd^cu_1 \wedge \cdots \wedge dd^cu_{d-1}.$$
To prove this, we have
$$\int_{\CC^d}(A-B)(dd^cC - dd^cD)\wedge dd^cu_1 \wedge \cdots \wedge dd^cu_{d-1}$$ $$-\int_{\CC^d}(C-D)(dd^cA - dd^cB)\wedge dd^cu_1 \wedge \cdots \wedge dd^cu_{d-1}$$ $$ =\int_{\CC^d}(Add^cC-Cdd^cA)\wedge T+\int_{\CC^d}(Ddd^cA-Add^cD)\wedge T$$
$$+\int_{\CC^d}(Bdd^cD-Ddd^cB)\wedge T + \int_{\CC^d}(Cdd^cB-Bdd^cC)\wedge T$$
which equals, by the integral formula (\ref{bedtayfor}), $2\pi$ times 
$$\int_{\PP^{d-1}}(\tilde \rho_{A}- \tilde \rho_{C})\tilde T+\int_{\PP^{d-1}}(\tilde \rho_{D}- \tilde \rho_{A})\tilde T+\int_{\PP^{d-1}}(\tilde \rho_{B}- \tilde \rho_{D})\tilde T+\int_{\PP^{d-1}}(\tilde \rho_{C}- \tilde \rho_{B})\tilde T.$$
Clearly this sum vanishes.

Another formula which will be used is
\begin{equation}
\label{ibyp}
\int_{\CC^d}(A-B)(dd^cC - dd^cD)\wedge dd^cu_1 \wedge \cdots \wedge dd^cu_{d-1}
\end{equation}
$$=-\int_{\CC^d}d(A-B)\wedge d^c(C - D)\wedge dd^cu_1 \wedge \cdots \wedge dd^cu_{d-1}.
$$

\noindent {\bf Remark}.  The following integration by parts ``formula'' {\bf is not valid}:
$$\int_{\CC^d}(A-B)(dd^cC)\wedge dd^cu_1 \wedge \cdots \wedge dd^cu_{d-1}$$ $$ =\int_{\CC^d}C(dd^cA - dd^cB)\wedge dd^cu_1 \wedge \cdots \wedge dd^cu_{d-1}.$$
For take $B=C=u_1=\cdots =u_{d-1}=\log^+|z|$ and $A=\log^+|z| +1$: the top line equals a positive multiple of the area of the unit sphere in $\CC^d$ while the bottom line vanishes. Thus it is {\bf essential} to use {\bf differences} in (\ref{ibyp}).

We prove a fundamental differentiability property of the energy.

\begin{proposition} 
\label{diffprop}
Let $u,u',v\in L^+(\CC^d)$. For $0\leq t\leq 1$, let 
$$f(t):= {\mathcal E} (u+t(u'-u),v).$$
Then $f'(t)$ exists for $0\leq t \leq 1$ and 
\begin{equation}
\label{difft} 
f'(t)= (d+1)\int_{\CC^d} (u'-u) (dd^c(u+t(u'-u)))^d.
\end{equation}
\end{proposition}
\begin{proof} Here we mean the appropriate one-sided derivatives at $t=0$ and $t=1$; e.g., 
\begin{equation}
\label{diff0} 
f'(0):=\lim_{t\to 0^+}\frac{f(t)-f(0)}{t}= (d+1)\int_{\CC^d} (u'-u) (dd^cu)^d.
\end{equation}
We prove this last statement. This implies the first; i.e., (\ref{difft}). For if $s$ is fixed,
$$g(t):=f(s+t)=   {\mathcal E} (u+(s+t)(u'-u),v)= {\mathcal E} (u+s(u'-u)+t(u'-u),v)$$
and applying (\ref{diff0}) to $g$ (so $u\to u+s(u'-u)$) we get 
$$g'(0)=f'(s)= (d+1)\int_{\CC^d} (u'-u) (dd^c(u+s(u'-u)))^d.$$

We begin with the observation that if $w:=u'-u$, then
$$\sum_{j=0}^d [dd^c(u+tw)]^j\wedge (dd^cv)^{d-j}-\sum_{j=0}^d (dd^cu)^j\wedge (dd^cv)^{d-j} $$
$$= t\sum_{j=0}^d j[dd^c w\wedge (dd^cu)^{j-1}\wedge (dd^cv)^{d-j}+0(t^2).$$
Then (all integrals are over $\CC^d$)
$${\mathcal E}(u+t(u'-u),v)-{\mathcal E}(u,v)={\mathcal E}(u+tw,v)-{\mathcal E}(u,v)$$
$$=\int [u+tw-v]\sum_{j=0}^d [dd^c(u+tw)]^j\wedge (dd^cv)^{d-j}$$
$$-\int (u-v)\sum_{j=0}^d (dd^cu)^j\wedge (dd^cv)^{d-j} $$
$$=t\int (u-v)\sum_{j=0}^d j[dd^c w\wedge (dd^cu)^{j-1}\wedge (dd^cv)^{d-j}+0(t^2)$$
$$+\int tw\sum_{j=0}^d [dd^c(u+tw)]^j\wedge (dd^cv)^{d-j}$$
$$=t\bigl[\int (u-v)\sum_{j=0}^d j[dd^c w\wedge (dd^cu)^{j-1}\wedge (dd^cv)^{d-j}]$$
$$+\int w\sum_{j=0}^d(dd^cu)^j\wedge (dd^cv)^{d-j}\bigr]+0(t^2)$$
$$=t\bigl[\int w\sum_{j=0}^d j[dd^c (u-v)\wedge (dd^cu)^{j-1}\wedge (dd^cv)^{d-j}]$$
$$+\int w\sum_{j=0}^d(dd^cu)^j\wedge (dd^cv)^{d-j}\bigr]+0(t^2).$$
Here we have used (\ref{ibyp1}) in the last step. Now check that
$$\sum_{j=0}^d j dd^c (u-v)\wedge (dd^cu)^{j-1}\wedge (dd^cv)^{d-j}+\sum_{j=0}^d(dd^cu)^j\wedge (dd^cv)^{d-j}$$
$$=(d+1)(dd^cu)^d$$
(try the case $d=2$!) and the result follows.
\end{proof}

We sometimes write (\ref{diff0}) in ``directional derivative'' notation as 
\begin{equation}
\label{diff02}
<\mathcal E' (u), u'-u>=(d+1)\int (u'-u)(dd^cu)^d.
\end{equation}
Note that the differentiation formula (\ref{difft}) is independent of $v$. This will also follow from the {\bf cocycle property}, which we now prove using (\ref{difft}).

\begin{proposition} 
\label{cocycleprop}
Let $u,v,w\in L^+(\CC^d)$. Then
$${\mathcal E}(u,v) +{\mathcal E}(v,w) + {\mathcal E}(w,u)=0.$$
\end{proposition}
\begin{proof} Let
$$f(t):={\mathcal E}(u+t(w-u),v)+{\mathcal E}(v,u)$$ 
and
$$g(t):={\mathcal E}(u+t(w-u),w)+{\mathcal E}(w,u).$$
Then $f(0)=g(0)=0$ by antisymmetry of ${\mathcal E}$. 
From the previous proposition, specifically, (\ref{difft}),
$$f'(t)= (d+1)\int_{\CC^d} (w-u) (dd^c(u+t(w-u)))^d$$
and
$$g'(t)= (d+1)\int_{\CC^d} (w-u) (dd^c(u+t(w-u)))^d;$$
i.e., $f'(t)=g'(t)$ for all $t$, Thus $f(1)=g(1)$; i.e.,
$${\mathcal E}(w,v) + {\mathcal E}(v,u)={\mathcal E}(w,w) + {\mathcal E}(w,u)={\mathcal E}(w,u).$$
\end{proof}

The independence of (\ref{difft}) on $v$ also follows from the cocycle property: if $v,v'\in L^+(\CC^d)$, then 
$${\mathcal E} (u+t(u'-u),v')+{\mathcal E} (v',v)+{\mathcal E} (v,u+t(u'-u))=0$$
so that the difference 
$${\mathcal E} (u+t(u'-u),v')-{\mathcal E} (u+t(u'-u),v)={\mathcal E} (v,v')$$
is independent of $t$. Thus one should consider ${\mathcal E}$ as a functional on the first slot with the second fixed. As such, as with the projection operator $P$, it is {\bf increasing and concave}:

\begin{proposition} Let $u,v,w\in L^+(\CC^d)$. Then
$$u\geq v \ \hbox{implies} \ {\mathcal E}(u,w) \geq 
{\mathcal E}(v,w)$$
and for $0\leq t\leq 1$
$${\mathcal E}(tu+(1-t)v,w)\geq t{\mathcal E}(u,w) +(1-t){\mathcal E}(v,w);$$
i.e., $g(t):={\mathcal E}(tu+(1-t)v,w)$ satisfies $g''(t)\leq 0$. 
\end{proposition}
\begin{proof} The monotonicity is trivial from the cocycle property and the definition of ${\mathcal E}$:
$${\mathcal E}(u,w) -{\mathcal E}(v,w)= {\mathcal E}(u,w) +{\mathcal E}(w,v)={\mathcal E}(u,v)$$
$$=\int(u-v)\sum_{j=0}^d (dd^cu)^{d-j}\wedge (dd^cv)^j \geq 0.$$
For the concavity, let 
$$g(t):={\mathcal E}(tu+(1-t)v,w)={\mathcal E}(v+t(u-v),w).$$
We show $g''(t)\leq 0$. From the differentiability result (\ref{difft}) in Proposition \ref{diffprop},
$$g'(t)= (d+1)\int_{\CC^d} (u-v) (dd^c(v+t(u-v)))^d.$$
To compute $g''(t)$, note that
$$\frac {d}{dt}(dd^c(a +tb))^d=d\cdot dd^cb\wedge 
(dd^c(a +tb))^{d-1}$$ so that 
$$g''(t)=(d+1)d\cdot \int_{\CC^d} (u-v) dd^c(u-v) \wedge (dd^c(v+t(u-v)))^{d-1}$$
$$=-(d+1)d\cdot \int_{\CC^d} d(u-v) \wedge d^c(u-v) \wedge (dd^c(v+t(u-v)))^{d-1}\leq 0$$
where the last equality comes from the integration by parts formula (\ref{ibyp}).
\end{proof}

A consequence of concavity is the following. Let $u_1,u_2,v\in L^+(\CC^d)$. Letting
$$g(s):= \mathcal E(u_1+s(u_2-u_1),v)$$
for $0\leq s\leq 1$, we have concavity of $g$ so that $g(s)\leq g(0)+g'(0)s$. In particular, at $s=1$, we have 
\begin{equation}\label{gconcav}
g(1)\leq g(0)+g'(0);
\end{equation}
 i.e.,
\begin{equation}\label{gconcave}\mathcal E(u_2,v)\leq \mathcal E(u_1,v)+(d+1)\int_{\CC^d}(u_2-u_1)(dd^cu_1)^d.
\end{equation}

For future use, we record the following. 

\begin{lemma}
\label{BT63} Let $\{w_j\},\{v_j\}\subset L^+(\CC^d)$ with $w_j\uparrow w\in L^+(\CC^d)$ and $v_j \uparrow v \in L^+(\CC^d)$. Then
$$\mathcal E(w_j,v) \to \mathcal E(w,v) \ \hbox{and} \ \mathcal E(w_j,v_j) \to \mathcal E(w,v).$$
\end{lemma}
\begin{proof} From the cocycle condition (Proposition \ref{cocycleprop}), it suffices to prove the first statement. This follows directly from Lemma 6.3 of \cite{[BT]}: {\it given $$w,\{v_j\},v, \{u_{1,j}\},u_1,...,\{u_{d,j}\},u_d \ \hbox{in} \ L^+(\CC^d)$$
with $v_j\uparrow v, \ u_{1,j}\uparrow u_1, ..., u_{d,j}\uparrow u_d$, 
$$\lim_{j\to \infty} \int_{\CC^d} (w-v_j)dd^cu_{1,j}\wedge \cdots \wedge u_{d,j}=\int_{\CC^d} (w-v)dd^cu_{1}\wedge \cdots \wedge u_{d}.$$}
\end{proof}

We remark that if $w_j\downarrow w\in L^+(\CC^d)$ and $v_j \downarrow v \in L^+(\CC^d)$ then we still have
\begin{equation}
\label{downlemma}
\mathcal E(w_j,v) \to \mathcal E(w,v) \ \hbox{and} \ \mathcal E(w_j,v_j) \to \mathcal E(w,v).
\end{equation}
The first statement is standard and the second follows from the first by Proposition \ref{cocycleprop}.

\begin{remark} A nonnegative functional on $L^+(\CC^d)$ bearing a closer resemblance to a classical ``energy'' is defined in section 5 of \cite{[BB]}. Fix $v\in L^+(\CC^d)$. For $u\in L^+(\CC^d)$, define
\begin{equation}
\label{ien}
I(u)=I_v(u):=\mathcal E(u,v) +(d+1)\int_{\CC^d}(v-u)(dd^cu)^d
\end{equation}
$$=\int_{\CC^d}(v-u)\bigl([(dd^cv)^d+\cdots + dd^cv\wedge (dd^cu)^{d-1}] -d\cdot (dd^cu)^d\bigr).$$
Note that $I(v) = I(v+c)=0$ for any constant $c$. It is not immediately obvious from this definition that $I(u)\geq 0$, but this follows trivially from concavity of $\mathcal E$. Indeed, let 
$$f(t):= \mathcal E(u+t(v-u),v)$$
so that $f(0)=\mathcal E(u,v)$ and $f(1)=\mathcal E(v,v)=0$. Concavity 
of $\mathcal E$ implies that $f(1)\leq f(0)+f'(0)$ from (\ref{gconcav}); i.e., using Proposition \ref{diffprop}, 
$$0\leq  \mathcal E(u,v) +   (d+1)\int_{\CC^d}(v-u)(dd^cu)^d =I(u).$$  
\end{remark}

\section {The Main Theorem.}\label{mainth}

In this section, we state and prove the main result of Berman and Boucksom, which relates asymptotics of certain ball-volume ratios with energies associated with extremal measures. For $E\subset \CC^d$ closed and $\phi$ an admissible weight on $E$, let 
$$\mathcal B^{\infty}(E,n\phi):=\{p_n\in \mathcal P_n: |p_n(z)^2e^{-2n\phi(z)}|\leq 1 \ \hbox{on}  \ E\}$$
be an $L^{\infty}-$ball and, if $\mu$ is a measure on $E$, 
$$\mathcal B^{2}(E,\mu,n\phi):=\{p_n\in \mathcal P_n: \int_E |p_n|^2e^{-2n\phi}d\mu\leq 1 \}$$
be an $L^{2}-$ball in $\mathcal P_n$. The key result is the following.

\begin{theorem}
\label{keycn}
Given $\phi,\phi'$ admissible weights on $E,E'$,
$$\frac{1}{2nN} [\mathcal B^{\infty}(E,n\phi):\mathcal B^{\infty}(E',n\phi')]\to \frac{1}{(d+1)(2\pi)^d}\mathcal E(V_{E,\phi}^*, V_{E',\phi'}^*).$$
If $\mu,\mu'$ are measures on $E,E'$ with $(E,\mu,\phi)$ and $(E',\mu',\phi')$ satisfying a weighted Bernstein-Markov property, then
$$\frac{1}{2nN} [\mathcal B^{2}(E,\mu,n\phi):\mathcal B^{2}(E',\mu',n\phi')]\to \frac{1}{(d+1)(2\pi)^d}\mathcal E(V_{E,\phi}^*,V_{E',\phi'}^*).$$
\end{theorem}

\noindent {\bf Remark}. We will prove the $L^{\infty}-$version and the $L^{2}-$ version follows as noted in the case of (\ref{l2linf}). From this, we will prove a version of strong Bergman asymptotics (Corollary \ref{strongasymp}) which says that if $(K,\mu,w)$ satisfies a weighted Bernstein-Markov inequality, then 
$$\frac{1}{N}B_n^{\mu,w} d\mu \to \mu_{K,Q} \ \hbox{weak-}*.
$$
However, the first step in the proof of Theorem \ref{keycn}, section \ref{sec:bergcd} below, is the special case of strong Bergman asymptotics in the case $K=\CC^d$ (and $Q$ is a sufficiently smooth ``strongly admissible'' weight). From a purely logical point of view, it would be interesting to find either 
\begin{enumerate}
\item a direct proof of Corollary \ref{strongasymp} that did not appeal to Theorem \ref{keycn} and/or
\item a direct proof of Theorem \ref{keycn} which did not require this special case of Corollary \ref{strongasymp}.
\end{enumerate}

\subsection {Bergman asymptotics in $\CC^d$.}\label{sec:bergcd}

Results on Bergman asymptotics in the Berman paper \cite{[B1]}, which apply to global weights on $\CC^d$, form the basis for an essential step in the proof of Theorem \ref{keycn}. The setting in \cite{[B1]} is this: $\phi\in C^{1,1}(\CC^d)$ with 
\begin{equation}
\label{stradm}
\phi (z) \geq (1+\epsilon)\log |z| \ \hbox{for}  \ |z|>>1 \ \hbox{for some}  \ \epsilon >0.
\end{equation}
We write, following (\ref{pq}), $P(\phi):=V_{\CC^d,\phi}$. We will call a global admissible weight $\phi$ satisfying (\ref{stradm}) {\it strongly admissible}. For $p_n\in \mathcal P_n$, our notation for the (varying weighted) $L^2-$norm is 
$$||p_n||_{n\phi}^2:=||p_n||_{\omega_d,n\phi}^2=\int_{\CC^d} |p_n(z)|^2 e^{-2n\phi(z)}\omega_d(z)$$
where $\omega_d(z)=(\frac{dd^c|z|^2}{2\pi })^d/d!$ on $\CC^d$. Under the growth assumption on $\phi$, it is easy to see that if $n>d/\epsilon$ then for each polynomial $p_n\in \mathcal P_n$, $||p_n||_{n\phi}<+\infty$. Next, given an orthonormal basis $\{q_1,...,q_N\}$ of $\mathcal P_n$, in this section we use the notation
$$B_n(z):=B_{n,\phi}(z):=K_{n,\phi}(z,z)e^{-2n\phi(z)}:=[\sum_{j=1}^N |q_j(z)|^2]e^{-2n\phi(z)}$$
for the $n$-th Bergman function; and we recall that
$$B_n(z)=\sup_{p_n\in \mathcal P_n\setminus \{0\}} |p_n(z)|^2e^{-2n\phi(z)}/||p_n||_{n\phi}^2.$$
Finally, let 
$$P:=\{z\in \CC^d: dd^c\phi(z) \ \hbox{exists and} \ dd^c\phi(z)>0\}$$
and if $u$ is a $C^{1,1}$ function such that $(dd^cu)^d$ is absolutely continuous with respect to Lebesgue measure, we write
$$\det (dd^c u)\omega_d :=\frac{1}{(2\pi)^d}(dd^c u)^d.$$

\begin{theorem}
\label{bermancn}
Given $\phi\in C^{1,1}(\CC^d)$ with $\phi (z) \geq (1+\epsilon)\log |z|$ for $|z|$ large, $P(\phi)\in C^{1,1}(\CC^d)\cap L^+(\CC^d)$;  $(dd^cP(\phi))^d$ is absolutely continuous with respect to Lebesgue measure; 
$$(dd^cP(\phi))^d=\det (dd^cP(\phi))\omega_d$$
as $(d,d)-$forms with $L^{\infty}_{loc}(\CC^d)$ coefficients; and a.e. on $D:=\{P(\phi) =\phi\}$ we have $\det (dd^c\phi)=\det (dd^cP(\phi))$. Moreover, 
$$\frac{1}{N}B_n\to \chi_{D\cap P} \det (dd^c\phi) \ \hbox{in} \  L^1(\CC^d)$$ and 
$$\frac{1}{N}B_n\omega_d \to \frac{1}{(2\pi)^d}(dd^cP(\phi))^d \ \hbox{weak}-*.$$
\end{theorem}
\begin{proof} We give the proof on pp. 6-13 in \cite{[B1]}. First of all, it is easy to see that $P(\phi)$ is Lipschitz. Since $P(\phi)\in L(\CC^d)$, 
$$P(\phi)(z+h)\leq \log |z+h| +C\leq  \log |z| + \log [1+|h|/|z|] +C$$
$$\leq \log |z| +C +2|h|/|z|\leq (1+\epsilon)\ln |z| +2|h|/|z|\leq \phi(z) +2|h|/R$$
for $|z|\geq R$ sufficiently large; while for $|z|<R$
$$P(\phi)(z+h)\leq \phi(z+h)\leq \phi(z)+[\sup_{|z|\leq R}|d\phi\|]|h|;$$
combining these estimates, we get that 
 $$P(\phi)(z+h)-\max [2/R,\sup_{|z|\leq R}|d\phi|]\cdot |h| \leq P(\phi)(z) $$
on $\CC^d$ as the left-hand-side is a competitor for $P(\phi)$. Applying this inequality with $z\to z+h$ and $h\to -h$ gives
$$|P(\phi)(z+h)-P(\phi)(z)|\leq \max [2/R,\sup_{|z|\leq R}|d\phi|]\cdot |h|.$$

To verify that the first order partial derivatives of $P(\phi)$ exist and are Lipschitz, we will show that each real second order (weak) partial derivative $\frac{\partial ^2P(\phi)}{\partial x_j  \partial x_k}, \ j,k\in \{1,...,2d\}$ of the real Hessian $D^2P(\phi)$ is representable by an $L^{\infty}_{loc}(\CC^d)$ function; i.e., for short, we say that $D^2P(\phi)$ has an $L^{\infty}_{loc}(\CC^d)$ density. We first claim that there is a constant $C=C(\phi)$ with
\begin{equation}
\label{c11}
[P(\phi)(z+h)+P(\phi)(z-h)]/2 -P(\phi)(z) \leq C|h|^2.
\end{equation}
To see this, let
$$g(z):=[P(\phi)(z+h)+P(\phi)(z-h)]/2 \in L(\CC^d).$$
Arguing as above we want to get 
$$g(z)\leq [\phi(z+h)+\phi(z-h)]/2 \leq P(\phi)(z) +C |h|^2$$
in order to prove (\ref{c11}). Precisely, looking at the first order Taylor polynomial of $\phi$ at $z$, on the ball $\mathcal B_R$ of radius $R$ centered at $0$ we have
$$[\phi(z+h)+\phi(z-h)]/2 \leq \phi(z) +C'|h|^2$$ 
where $C'=C'(R)$ so that
\begin{equation}
\label{c11a}
g(z)=[P(\phi)(z+h)+P(\phi)(z-h)]/2 \leq \phi(z) +C'|h|^2
\end{equation}
on this ball. Here, 
$$C'=\sup_{\mathcal B_R} |D^2\phi|$$
where $D^2\phi$ is the real Hessian of $\phi$. 
Using the growth of $\phi$ as in the previous paragraph we can get an estimate outside a large ball of the form
\begin{equation}
\label{c11b}
g(z)=[P(\phi)(z+h)+P(\phi)(z-h)]/2 \leq \phi(z) +C''|h|^2.
\end{equation}
To see this, recall for $|z|$ large,
$$P(\phi)(z+h)\leq \log |z+h| +C\leq  \log |z| + \log [1+|h|/|z|] +C;$$
similarly,
$$P(\phi)(z-h)\leq \log (|z|-|h|+2|h|)+C$$
$$\leq  \log |z| + \log [1-|h|/|z| +2|h|/|z|] +C$$
$$\leq  \log |z| + \log [1-|h|/|z|]+2\epsilon \log |z| +C$$
for $|z|$ large. Thus 
$$[P(\phi)(z+h)+P(\phi)(z-h)]/2 \leq (1+\epsilon) \log |z| +C_1|h|^2/|z|^2 + C_2$$
$$ \leq \phi(z) +C''|h|^2$$
for $|z|$ large; this gives (\ref{c11b}); together with (\ref{c11a}), we get (\ref{c11}).

To deduce the $C^{1,1}$ regularity from (\ref{c11}), we follow the  arguments of \cite{dem}. Taking regularizations $(P(\phi))_{\epsilon}:=\phi* \chi_{\epsilon}$ of 
$P(\phi)$, we have the same estimate for $(P(\phi))_{\epsilon}$ as we have for $P(\phi)$ in (\ref{c11}):
$$[(P(\phi))_{\epsilon}(z+h)+(P(\phi))_{\epsilon}(z-h)]/2 -(P(\phi))_{\epsilon}(z) \leq C|h|^2.$$
Forming a Taylor expansion of degree $2$ of $(P(\phi))_{\epsilon}$ around $z$ gives 
$$D^2(P(\phi))_{\epsilon}(z)\cdot h^2 \leq C|h|^2.$$
We now use the fact that $(P(\phi))_{\epsilon}$ is psh, so that its complex Hessian is positive semi-definite. Thus
$$D^2(P(\phi))_{\epsilon}(z)\cdot h^2 + D^2(P(\phi))_{\epsilon}(z)\cdot (ih)^2\geq 0$$
so that
$$D^2(P(\phi))_{\epsilon}(z)\cdot h^2 \geq - D^2(P(\phi))_{\epsilon}(z)\cdot (ih)^2\geq -C|h|^2$$
and we conclude that 
$$|D^2(P(\phi))_{\epsilon}(z)|\leq C.$$
Letting $\epsilon \to 0$, we conclude that $D^2P(\phi)$ has an $L^{\infty}_{loc}(\CC^d)$ density and we have a local estimate
$$|D^2P(\phi)|\leq C'=C'(K,\phi)$$
on a compact set $K$ where $C'(K,\phi)$ depends on $K$, $\sup_K|D^2\phi|$, and the growth of $\phi$. This shows that $P(\phi)\in C^{1,1}(\CC^d)$.

Now since, e.g., $(P(\phi))_{\epsilon}\downarrow P(\phi)$, we have the Monge-Ampere measures $(dd^c(P(\phi))_{\epsilon})^d$ converge weak-* to $(dd^cP(\phi))^d$ and hence $$(dd^cP(\phi))^d=\det (dd^cP(\phi))\omega_d$$ as forms with $L^{\infty}_{loc}(\CC^d)$ coefficients. The fact that 
$$\det (dd^cP(\phi))^d=\det (dd^c\phi)^d \ \hbox{a.e. on}  \ D=\{P(\phi) =\phi\}$$ will follow by showing that 
$$\frac{\partial ^2(P(\phi) -\phi)}{\partial z_j  \partial \bar z_k}=0 \ \hbox{a.e. on} \ D.$$
Indeed, we show that for all the real second order derivatives  
\begin{equation}
\label{secondpartials}
\frac{\partial ^2(P(\phi) -\phi)}{\partial x_j  \partial x_k}=0 \ \hbox{a.e. on} \ D.
\end{equation}
We will use the lemma on p. 53 of \cite{KS}. This result implies that if a real-valued function $u$ is Lipschitz on an open set $\Omega$, then $\frac{\partial u}{\partial x_j }=0$ a.e on $\{x\in \Omega:u(x)=0\}$. Apply this to the $C^{1,1}$ function $P(\phi) -\phi$ to conclude $\frac{\partial (P(\phi) -\phi)}{\partial x_j }=0$ a.e on $D$. Since $\frac{\partial (P(\phi) -\phi)}{\partial x_j }$ is Lipschitz, $\frac{\partial (P(\phi) -\phi)}{\partial x_j }=0$ on all of $D$ and we can apply the lemma to $\frac{\partial (P(\phi) -\phi)}{\partial x_j }$ to obtain $\frac{\partial ^2(P(\phi) -\phi)}{\partial x_j  \partial x_k}=0$ a.e on $\{\frac{\partial (P(\phi) -\phi)}{\partial x_j }=0\}$; in particular, we obtain (\ref{secondpartials}).

We turn to the Bergman asymptotics; i.e., the behavior of $\{\frac{1}{N}B_n\}$ for $n$ large. The first step is an asymptotic upper bound (the ``local holomorphic Morse inequality'') which is a general fact about a $C^{1,1}$ function $\phi$ satisfying (\ref{stradm})  (i.e., independent of (weighted) pluripotential theory). Recall that 
$$||p_n||_{n\phi}^2:=\int_{\CC^d} |p_n(z)|^2 e^{-2n\phi(z)}\omega_d(z)$$
and
$$B_n(z)=B_{n,\phi}(z)=\sup_{p_n\in \mathcal P_n\setminus \{0\}} |p_n(z)|^2e^{-2n\phi(z)}/||p_n||_{n\phi}^2.$$
We can also consider $B_n(z)$ as the supremum in the class of all weighted polynomials $\mathcal Q_{n,\phi}:=p_ne^{-n\phi}$:
$$B_{n,\phi}(z)=\sup_{\mathcal Q_{n,\phi}\not \equiv 0}|\mathcal Q_{n,\phi}(z)|^2/||\mathcal Q_{n,\phi}||_{L^2(\CC^d)}.$$
If we let $\tilde{\phi}:=\phi+\Re (g)$ where $g$ is an entire function and 
set $ \tilde{p}_n:=p_n e^{ng}$, then 
$$|\tilde{p}_n(z)|^2e^{-2n\tilde{\phi}(z)} = |p_n(z)|^2e^{-2n\phi(z)}=|\mathcal Q_{n,\phi}(z)|^2$$ 
so that $B_{n,\phi}(z)=B_{n,\tilde \phi}(z)$. In particular, to study asymptotic behavior of $\{\frac{1}{N}B_n\}$, by taking an affine $g$ ($g(z)=a+\sum_{j=1}^d a_jz_j$) we may assume, working near the origin $0$ for convenience,  that $\phi(0)$ vanishes and $d\phi(0)=0$. If the second order partial derivatives exist at $0$ then by taking a quadratic $g$ ($g(z)=\sum_{i,j} a_{ij}z_iz_j$) and a unitary change of coordinates, we can assume that, near $0$,
\begin{equation}
\label{near0}
\phi(z) =\sum_{j=1}^d\lambda_j |z_j|^2 +0(|z|^3).
\end{equation}

We verify two estimates which will be useful. First of all, regardless of the existence of second order partial derivatives at $0$, on a fixed ball $\mathcal B_R$ centered at $0$, there is a constant $C=C(R, \phi)$ with
\begin{equation}
\label{phibound}
|\phi(z)|\leq C|z|^2 \ \hbox{on} \ \mathcal B_R.
\end{equation}
This estimate is sufficient to prove local uniform boundedness of $\{\frac{1}{N} B_n\}$: 
\begin{equation}
\label{locbound}
\frac{1}{N} B_n(z)\leq C=C(K) \ \hbox{on a compact set}  \ K.
\end{equation} 
More to the point, we show that if the second derivatives of $\phi$ {\it do} exist at a point, say $0$, then for {\bf any} $R>0$,
\begin{equation}
\label{c11more}
\lim_{n\to \infty} \bigl[\sup_{z\in \mathcal B_R} \bigl|n\phi(z/\sqrt n)-\sum_{j=1}^d\lambda_j|z_j|^2\bigr|\bigr]=0.
\end{equation}
This will be used to prove the pointwise upper bound asymptotics 
\begin{equation}
\label{asymp0}
\limsup_{n\to \infty} \frac{1}{Nd!}B_n(0)\leq \chi_P(0)\det (dd^c\phi(0)).\end{equation}
To prove (\ref{phibound}) and (\ref{c11more}), define, for $0\leq t\leq 1$, 
$$\psi(t):=\phi(tz)-\sum_{j=1}^d \lambda_j |tz_j|^2.$$
Then $\psi(0)=0$ and $\psi$ is of class $C^1$ so that
$$\psi(1)=\int_0^1 \psi'(t)dt = \phi(z)-\sum_{j=1}^d \lambda_j |z_j|^2.$$
Now 
$$\psi'(t)=\sum_{j=1}^d [z_j\frac{\partial \phi}{\partial z_j}(tz)+\bar z_j\frac{\partial \phi}{\partial \bar z_j}(tz)]-2\sum_{j=1}^d \lambda_j t|z_j|^2.$$
Since $d\phi(0)=0$, $\frac{\partial \phi}{\partial z_j}(tz),\frac{\partial \phi}{\partial \bar z_j}(tz)=0(|z|)$ so that 
$$|\psi'(t)|\leq C|z|^2, \ 0\leq t\leq 1$$
which gives (\ref{phibound}):
$$|\phi(z)|= |\phi(z)-\sum_{j=1}^d \lambda_j |z_j|^2+\sum_{j=1}^d \lambda_j |z_j|^2|$$
$$\le |\phi(z)-\sum_{j=1}^d \lambda_j |z_j|^2| + |\sum_{j=1}^d \lambda_j |z_j|^2| \leq C|z|^2.$$
For (\ref{c11more}), if the second derivatives of $\phi(z)$ exist at $z=0$, then the second derivatives of $\phi(z)-\sum_{j=1}^d \lambda_j |z_j|^2 $ vanish at $z=0$. Hence we get
$$|\psi'(t)|= o(|z|^2), \ 0\leq t\leq 1;$$
i.e., given $\epsilon >0$ there exists $\delta >0$ such that if $|z| < \delta$ then $|\psi'(t)|\leq \epsilon \cdot |z|^2$. This gives
$$|\psi(1)|=|\int_0^1 \psi'(t)dt|= |\phi(z)-\sum_{j=1}^d \lambda_j |z_j|^2|\leq \epsilon \cdot |z|^2.$$
Thus if $|z|\leq R$ and $n$ is sufficiently large
$$|n\phi(z/\sqrt n)-\sum_{j=1}^d\lambda_j|z_j|^2|=n|\phi(z/\sqrt n)-\sum_{j=1}^d\lambda_j|z_j/\sqrt n|^2|$$
$$\leq n\epsilon |z|^2/n \leq n\epsilon (R/\sqrt n)^2=\epsilon R.$$
This is true for all $\epsilon >0$, giving (\ref{c11more}).

 Recall that
\begin{equation}
\label{bn} B_n(z)=\sup_{p_n\in \mathcal P_n\setminus \{0\}} |p_n(z)|^2e^{-2n\phi(z)}/||p_n||_{n\phi}^2.
\end{equation}
Thus, working near a fixed point which we take to be the origin, $0$, and taking an extremal $p_n$, 
$$\frac{1}{N}B_n(0)= |p_n(0)|^2e^{-2n\phi(0)}/N||p_n||_{n\phi}^2=|p_n(0)|^2/N||p_n||_{n\phi}^2$$
$$=\frac{|p_n(0)|^2}{N\int_{\CC^d} |p_n(z)|^2 e^{-2n\phi(z)}\omega_d(z)}$$
$$\leq \frac{|p_n(0)|^2}{N\int_{|z|\leq R/\sqrt n} |p_n(z)|^2 e^{-2n\phi(z)}\omega_d(z)}$$
for any $R>0$. Choosing $R$ as in (\ref{phibound}), we can replace $\phi(z)$ by $C|z|^2$ in the integrand:
$$\frac{1}{N}B_n(0)\leq \frac{|p_n(0)|^2}{N\int_{|z|\leq R/\sqrt n} |p_n(z)|^2 e^{-2nC|z|^2}\omega_d(z)}.$$
Applying the subaveraging property to the psh function $|p_n|^2$ on the ball $\{|z|\leq R/\sqrt n\}$ with respect to the radial probability measure $e^{-2nC|z|^2}\omega_d(z)/\int_{|z|\leq R/\sqrt n}  e^{-2nC|z|^2}\omega_d(z)$, we obtain
$$\frac{1}{N}B_n(0)\leq \frac{1}{N\int_{|z|\leq R/\sqrt n}  e^{-2nC|z|^2}\omega_d(z)}\leq \frac{d!}{\int_{|z'|\leq R}  e^{-2C|z'|^2}\omega_d(z')}$$
which is finite. The last inequality comes via the change of variables $ z\to z':=z\sqrt n$, noting that $\omega_d(z') = n^d\omega_d(z)$ and $N={d+n\choose d}\geq n^d/d!$. Letting the point ``$0$'' vary over a compact set $K$, we get a $C=C(K)$ with
$$\frac{1}{N}B_n(z)\leq C=C(K) \ \hbox{for} \ z\in K;$$
i.e., (\ref{locbound}) holds.

Now start with
$$\frac{1}{N}B_n(0)\leq \frac{|p_n(0)|^2}{N\int_{|z|\leq R/\sqrt n} |p_n(z)|^2 e^{-2n\phi(z)}\omega_d(z)}$$
which is valid for any $R>0$. Letting $ z\to z':=z\sqrt n$ we have
$$\frac{1}{N}B_n(0)\leq \frac{d!|p_n(0)|^2}{\int_{|z'|\leq R} |p_n(z'/\sqrt n)|^2 e^{-2n\phi(z'/\sqrt n)}\omega_d(z')}.$$
Define
$$\rho_{n,R}:=\exp \bigl [2 \sup_{|z'|\leq R} \bigl|n\phi(z'/\sqrt n)-\sum_{j=1}^d\lambda_j|z'_j|^2\bigr|\bigr ].$$
Then
$$\frac{1}{N}B_n(0)\leq \rho_{n,R}\cdot \frac{d!|p_n(0)|^2}{\int_{|z'|\leq R} |p_n(z'/\sqrt n)|^2 e^{-2\sum_{j=1}^d\lambda_j|z'_j|^2}\omega_d(z')}.$$
Applying the subaveraging property to the psh function $|p_n|^2$ on the ball $\{|z'|\leq R\}$ with respect to the radial probability measure 
$$e^{-2\sum_{j=1}^d\lambda_j|z'_j|^2}\omega_d(z)/\int_{|z'|\leq R}  e^{-2\sum_{j=1}^d\lambda_j|z'_j|^2}\omega_d(z')$$ we obtain
$$\frac{1}{N}B_n(0)\leq  \rho_{n,R}\cdot \frac{d!}{\int_{|z'|\leq R}  e^{-2\sum_{j=1}^d\lambda_j|z'_j|^2}\omega_d(z')}.$$
We now assume the second order partial derivatives of $\phi$ exist at $0$. By (\ref{c11more}), $\rho_{n,R}\to 1$ as $n\to \infty$; thus for all $R>0$
$$\limsup_{n\to \infty} \frac{1}{Nd!}B_n(0)\leq \frac{1}{\int_{|z'|\leq R}  e^{-2\sum_{j=1}^d\lambda_j|z'_j|^2}\omega_d(z')}.$$
Letting $R\to \infty$ the Gaussian integral goes to $\frac{\pi^d}{2^d\lambda_1 \cdots \lambda_d}$ if all $\lambda_j >0$ and to $+\infty$ otherwise. Thus we have verified (\ref{asymp0}); i.e., we have shown:
\begin{equation}
\label{morse}
\limsup_{n\to \infty} \frac{1}{Nd!}B_n(z)=\limsup_{n\to \infty} \frac{1}{n^d}B_n(z)\leq \chi_P(z)\det (dd^c\phi(z))
\end{equation}
for a.e. $z\in \CC^d$.

Finally, we utilize (weighted) pluripotential theory to obtain the final results. By (\ref{locbound}), 
\begin{equation}
\label{oldmorse2}
\frac{1}{n^d}|p_n(z)|^2e^{-2n\phi(z)}/||p_n||_{n\phi}^2\leq \frac{1}{n^d}B_n(z)\leq C_D \ \hbox{for} \ z\in D.
\end{equation}
We claim that
\begin{equation}
\label{forld}
 \frac{1}{n^d}B_n(z)\leq C_De^{-2n(\phi(z)-P(\phi)(z))} \ \hbox{on} \ \CC^d.
\end{equation}
To see this, note that for any $p_n\in \mathcal P_n$ with $||p_n||^2_{n\phi}=n^{-d}$, by (\ref{oldmorse2}) 
$$  |p_n(z)|^2e^{-2n\phi(z)}\leq C_D \ \hbox{on} \ D.$$
But then 
$$\frac{1}{2n}\log |p_n(z)|^2\leq \phi(z)+\frac{1}{2n}\log C_D  \ \hbox{on} \ D$$
so that, from (\ref{equivdef}), 
$$\frac{1}{2n}\log |p_n(z)|^2\leq P(\phi)(z)+\frac{1}{2n}\log C_D  \ \hbox{on} \ \CC^d.$$
Thus from (\ref{bn})
$$ \frac{1}{n^d}B_n(z)=\sup_{||p_n||^2_{n\phi}=n^{-d}} |p_n(z)|^2e^{-2n\phi(z)}\leq 
C_De^{-2n(\phi(z)-P(\phi)(z))} \ \hbox{on} \ \CC^d.$$
In particular, 
$$ \lim_{n\to \infty}\frac{1}{n^d}B_n(z) =0 \  \hbox{on} \ \CC^d \setminus D.$$
Note then (\ref{morse}) can be strengthened to 
\begin{equation}
\label{morse2}
\limsup_{n\to \infty} \frac{1}{Nd!}B_n(z)=\limsup_{n\to \infty} \frac{1}{n^d}B_n(z)\leq \chi_{D\cap P}(z)\det (dd^c\phi(z))
\end{equation}
for a.e. $z\in \CC^d$.

From (\ref{forld}) and the growth assumption on $\phi$, for a sufficiently large $R$, there is a $C$ with 
$$\frac{1}{n^d}B_n(z)\leq C|z|^{-2n\epsilon} \ \hbox{for} \ |z|>R;$$ 
then dominated convergence (recall (\ref{locbound})) shows that
\begin{equation}
\label{offd}
\lim_{n\to \infty}\int_{\CC^d \setminus D}\frac{1}{n^d}B_n\omega_d =0.\end{equation}

Next we show
\begin{equation}
\label{ond}
\lim_{n\to \infty}\int_{D}\frac{1}{n^d}B_n\omega_d =\int_{D\cap P}
(2\pi dd^c\phi)^d/d!.
\end{equation}
To prove (\ref{ond}), we know that
$$\int_{\CC^d}B_n\omega_d = N\asymp n^d/d!$$
and using (\ref{offd}) we have
$$1/d! = \lim_{n\to \infty}\int_{\CC^d}\frac{1}{n^d}B_n\omega_d= \lim_{n\to \infty}\int_{D}\frac{1}{n^d}B_n\omega_d.$$
On the other hand, by (\ref{morse}) applied to $D$ and Fatou's lemma, we have
$$1/d! = \lim_{n\to \infty}\int_{D}\frac{1}{n^d}B_n\omega_d\leq \int_{D\cap P}\frac{(dd^c\phi)^d}{(2\pi)^dd!}.$$
But from the first part of this theorem, we can replace $(dd^c\phi)^d$ by $(dd^cP(\phi))^d$ which has total mass $(2\pi)^d$ on $D\cap P$; hence
$$1/d! = \lim_{n\to \infty}\int_{D}\frac{1}{n^d}B_n\omega_d\leq \int_{D\cap P}\frac{(dd^cP(\phi))^d}{(2\pi)^dd!} =1/d!.$$
This gives (\ref{ond}). We will use this relation, together with (\ref{morse2}), to show that 
$$\frac{1}{n^d}B_n \to \chi_{D\cap P}\det (dd^c\phi) \ 
\hbox{in} \ L^1(\CC^d).$$
First we prove the following measure-theoretic lemma (cf., \cite{[toep]}).
\begin{lemma}
Let $(X,\mu)$ be a measure space and let $\{f_n\}$ be a sequence of uniformly bounded, integrable functions on $X$. If $f$ is a bounded, integrable function on $X$ with
\begin{enumerate}
\item $\lim_{n\to \infty} \int_X f_n d\mu =  \int_X f d\mu$ and
\item $\limsup_{n\to \infty} f_n \leq f \ \hbox{a.e.} \ \mu,$
\end{enumerate}
then $f_n\to f$ in $L^1(X,\mu)$.
\end{lemma}
\begin{proof}
Let $\chi_n$ be the characteristic function of $\{x\in X: f_n - f\geq 0\}$. Since
$$\int_X |f_n-f|d\mu =  \int_X \chi_n (f_n-f)d\mu +  \int_X (1-\chi_n) (f-f_n)d\mu$$
$$ = 2 \int_X \chi_n (f_n-f)d\mu  + \int_X (f-f_n)d\mu,$$ 
using (1) we have 
$$\limsup_{n\to \infty} \int_X |f_n-f|d\mu =2\limsup_{n\to \infty} \int_X \chi_n (f_n-f)d\mu.$$
By Fatou's lemma, 
$$\limsup_{n\to \infty} \int_X \chi_n (f_n-f)d\mu \leq \int_X [\limsup_{n\to \infty} \chi_n (f_n-f)]d\mu.$$
Now the result follows from (2).
\end{proof}
We set $f_n := \frac{1}{n^d}B_n$ and $f:= \chi_{D\cap P}\det (dd^c\phi)$ then from (\ref{ond}) and (\ref{morse2}) we get the convergence $\frac{1}{n^d}B_n \to \chi_{D\cap P}\det (dd^c\phi)$ in $L^1(\CC^d)$. This implies weak-* convergence of $\frac{1}{n^d}B_n\omega_d$ to $\chi_{D\cap P}\det (dd^c\phi)\omega_d$ and completes the proof of the theorem.
\end{proof}

We make two important remarks. 
\begin{enumerate}
\item The above argument yields, since $\lim_{n\to \infty}\int_{\CC^d \setminus D}\frac{1}{n^d}B_n\omega_d =0$, that the compactly supported measures
\begin{equation}
\label{weakstarleb}
\frac{1}{N}B_n\cdot \chi_S\omega_d \to \frac{1}{(2\pi)^d}(dd^cP(\phi))^d \ \hbox{weak}-*
\end{equation}
where $S$ is any set containing $D$.
\item We have a weighted Bernstein-Markov property:
\begin{equation}
\label{weightedleb}
\sup_{\CC^d}|p_n e^{-n\phi}|\leq C_n [\int_{\CC^d}|p_n|^2e^{-2n\phi}\omega_d]^{1/2}<+\infty
\end{equation}
for $p_n\in \mathcal P_n$ and $n>n_0(\epsilon,d)=d/\epsilon$ where $C_n^{1/n}\to 1$. For 
$$\sup_{\CC^d}|p_n e^{-n\phi}|=\sup_{D}|p_n e^{-n\phi}|$$
$$\leq C_n[\int_{D}|p_n|^2e^{-2n\phi}\omega_d]^{1/2}\leq C_n [\int_{\CC^d}|p_n|^2e^{-2n\phi}\omega_d]^{1/2}$$
since $(D,\omega_d|_D,\phi|_D)$ satisfies a weighted Bernstein-Markov property. This follows from Proposition \ref{detdprop} or see \cite{bloom}.

\end{enumerate}

\subsection {Differentiability of $\mathcal E \circ P$.} \label{diffep}

We turn to the main ``differentiability'' result. Generally we will fix a function $v\in L^+(\CC^d)$ which will be in the second slot of all energy terms and we simply write, for any $\tilde v\in L^+(\CC^d)$,
$${\mathcal E}(\tilde v):={\mathcal E}(\tilde v,v).$$
If we take a specific $v$, we revert to the notation on the right-hand-side of this equation. For a closed subset $E\subset \CC^d$ and an admissible weight $a$ on $E$, we write $P(a)$ (sometimes $P_E(a)$) to denote the regularized weighted extremal function $V_{E,a}^*$.

We state two versions of differentiability of $\mathcal E \circ P$. One version, Proposition \ref{mainprop}, is for a second admissible weight $b$ on $E$ where we consider the perturbed weight $a+t(b-a)$ and the associated weighted extremal function $P(a+t(b-a))$ and show the differentiability of 
$$F(t):=\mathcal E( P(a+t(b-a))).$$
If $E$ is unbounded, we will need to make an additional assumption on $u:=b-a$ so that (\ref{loclip}) below holds; also, in this case, we restrict to $0\leq t\leq 1$ so that $a+t(b-a)=tb+(1-t)a$, being a convex combination of $a, b$, is admissible on $E$. The other version, Proposition \ref{diffpropcor}, is for a compact set $E$ and an arbitrary real $t$. We take a function $u\in C(E)$, consider the perturbed weight $a+tu$, and show the differentiability of 
$$F(t):=\mathcal E( P(a+tu)).$$
Apriori, since $t\in \RR$, we must assume $u$ is continuous so that $a+tu$ is an admissible (lowersemicontinuous) weight. However, to rigorously prove the results, we assume some regularity of $a,b$ and/or $u$: $C^2-$smoothness (or at least $C^{1,1}-$smoothness) will suffice.

\begin{proposition} 
\label{mainprop}
Let $v\in L^+(\CC^d)$. For admissible weights $a,b\in C^2(E)$ on a closed set $E\subset \CC^d$, let $u:=b-a$ and let 
$$F(t):={\mathcal E}(P(a+tu),v))$$
for $t\in \RR$. If $E$ is unbounded, we assume (\ref{unbhyp}) holds and $0\leq t\leq 1$. Then
\begin{equation}
\label{Pdifft}
F'(t)= (d+1)\int_{\CC^d} u (dd^c P(a+tu))^d.
\end{equation}

\end{proposition}

\begin{proposition}
\label{diffpropcor}
Let $v\in L^+(\CC^d)$. For an admissible weight $a$ on a compact set $E\subset \CC^d$ and $u\in C^2(E)$, let 
$$F(t):={\mathcal E}(P(a+tu),v)$$
for $t\in \RR$. Then
\begin{equation}
\label{Pdifftcor}
F'(t)= (d+1)\int_{\CC^d} u (dd^c P(a+tu))^d.
\end{equation}
\end{proposition}

We prove Proposition \ref{mainprop} and Proposition \ref{diffpropcor} simultaneously.

\begin{proof} We may take $v=P(a)$. As in the proof of Proposition \ref{diffprop} we prove only the one-sided limit as $t\to 0^+$:  
\begin{equation}
\label{Pdiff0} 
F'(0):=\lim_{t\to 0^+}\frac{F(t)-F(0)}{t}= (d+1)\int_{\CC^d} u (dd^c P(a))^d.
\end{equation}
This implies (\ref{Pdifft}). For if $s$ is fixed,
$$G(t):=F(s+t)=   {\mathcal E}(P(a+(s+t)u),v))$$
$$= {\mathcal E}(P(a+su+tu),v))$$
and applying (\ref{Pdiff0}) to $G$ (so $a\to a+su$) we get 
$$G'(0)=F'(s)= (d+1)\int_{\CC^d} u (dd^c P(a+su))^d.$$

Note that $F(0)=0$ and to verify (\ref{Pdiff0}) it suffices to prove 
\begin{equation}
\label{smooth5.7at0}
\mathcal E(P(a+tu),P(a))=(d+1)t\int_{\CC^d} u(dd^cP(a))^d +o(t).
\end{equation}
We need two ingredients for (\ref{smooth5.7at0}): 
\begin{equation}
\label{item1}
\mathcal E(P(a+tu),P(a))=(d+1)\int_{\CC^d} [P(a+tu)-P(a)](dd^cP(a))^d +o(t)
\end{equation}
and
\begin{equation}
\label{item2}
\lim_{t\to 0} \int_{D(0)-D(t)} (dd^cP(a))^d  =0
\end{equation}
where
$$D(t):=\{z \in \CC^d: \ P(a+tu)(z) = (a+tu)(z)\}.$$
We will state and prove (\ref{item1}) and (\ref{item2}) as separate lemmas.

Given (\ref{item1}) and (\ref{item2}), and observing from (\ref{suppw}) that 
\begin{equation}
\label{support}
\hbox{supp}(dd^cP(a))^d  \subset D(0),
\end{equation}
(\ref{smooth5.7at0}) follows as in \cite{[BB]}, p. 28:
$$\mathcal E(P(a+tu),P(a))=(d+1)\int_{\CC^d} [P(a+tu)-P(a)](dd^cP(a))^d +o(t)$$
$$=(d+1)\int_{D(0)-D(t)} [P(a+tu)-P(a)](dd^cP(a))^d $$
$$+(d+1)\int_{D(0)\cap D(t)} [P(a+tu)-P(a)](dd^cP(a))^d +o(t)$$
$$=(d+1)\int_{D(0)-D(t)} [P(a+tu)-P(a)](dd^cP(a))^d$$
$$ +(d+1)t\int_{D(0)\cap D(t)} u(dd^cP(a))^d +o(t)$$
$$=(d+1)\int_{D(0)-D(t)} [P(a+tu)-P(a)-tu](dd^cP(a))^d$$
$$ +(d+1)t\int_{D(0)} u(dd^cP(a))^d +o(t)$$
since $P(a+tu)-P(a)= tu$ on $D(0)\cap D(t)$. Now (\ref{loclip}) or (\ref{loclip2}) implies 
$$|P(a+tu)-P(a)-tu| =0(t)$$
on the {\it bounded} set $D(0)-D(t)$
(recall if $E$ is unbounded we assume (\ref{unbhyp}) holds in the setting of Proposition \ref{mainprop}) and this fact, combined with (\ref{item2}) and (\ref{support}), finishes the proof. 
\end{proof}

In (\ref{item1}), since $(dd^cP(a))^d$ is supported in $D(0)$, 
$$\int_{\CC^d} [P(a+tu)-P(a)](dd^cP(a))^d= \int_{D(0)} [P(a+tu)-P(a)](dd^cP(a))^d;$$
and, on $D(t)\cap D(0)$, we have $P(a+tu)-P(a)=tu$. Then the  content of (\ref{item2}) is that the contribution to this integral on $D(0)- D(t)$ is negligible. It is in proving (\ref{item2}) that the $C^2-$smoothness (or at least $C^{1,1}-$smoothness) is needed.

\begin{lemma} 
\label{mamass}
Let $a$ be an admissible weight on a compact set $E$ and let $u\in C^2(E)$. Then
\begin{equation}
\lim_{t\to 0} \int_{D(0)-D(t)} (dd^cP(a))^d  =0
\end{equation}
where $D(t)=\{P(a +tu)=a +tu\}$ for $t\in \RR$.
\end{lemma}

\begin{proof} The hypothesis $u\in C^2(E)$ means that $u$ is the restriction to $E$ of a $C^{\infty}$ function (which we also denote by $u$) on $\CC^d$; clearly we can take this function to have compact support. We prove the result for $t>0$; i.e $t\to 0^+$. We can find $M>0$ sufficiently large depending on $u$ and its support so that $u+M\psi$ is psh where $\psi(z) = \frac{1}{2} \log (1+|z|^2)$. Observing that 
$$D(0)\setminus D(t)\subset S$$
where 
$$S:=\{P(a+tu)< P(a) +tu\}=\{P(a+tu)+tM\psi < P(a) + t(u+M\psi)\}$$
and
$$D(t) \cap \{P(a+tu)< P(a) +tu\}=\emptyset,$$
we have
$$\int_{D(0)-D(t)} (dd^cP(a))^d  \leq \int_S (dd^cP(a))^d $$
$$\leq \int_S [dd^c(P(a)+t(u+M\psi)]^d\leq \int_S [dd^c(P(a+tu)+tM\psi)]^d$$
$$=\int_S [dd^c(P(a+tu))]^d +0(t) =0(t).$$
Here, the inequality in the second line comes from the $L^+-$comparison principle, Proposition \ref{L+comp}, since each of $(\frac{1}{1+tM})[P(a+tu)+tM\psi]$ and $(\frac{1}{1+tM})[P(a) + t(u+M\psi)]$ belong to $L^+(\CC^d)$.
\end{proof}
%\noindent {\sl Question: Is Lemma \ref{mamass} still true for a locally regular compacta $K$ and an admissible continuous weight $Q$?YES -- hedelmalm question?}

\begin{corollary} 
\label{mamassbis}
Let $a,b\in C^2(E)$ be admissible weights on a closed, unbounded set $E$. If (\ref{unbhyp}) holds then
\begin{equation}
\lim_{t\to 0} \int_{D(0)-D(t)} (dd^cP(a))^d  =0
\end{equation}
where $D(t)=\{P(a +t(b-a))=a +t(b-a)\}$ for $0\leq t\leq 1$.
\end{corollary}

\begin{proof} First of all, $(dd^cP(a))^d$ has compact support. Also, by (\ref{unbhyp}), the extremal functions $P(a+t(b-a))$ for all $0\leq t\leq 1$ are independent of the values of $a,b$ outside a large ball. Thus  we may assume that $a=b$ outside a fixed ball. In other words, this  case is reduced to the case of Lemma \ref{mamass} where $u=b-a$.
\end{proof}

The content of (\ref{item1}), Lemma \ref{keylemma2} below, is that the contribution of each of the $d+1$ terms in the energy $\mathcal E(P(a+tu),P(a))$ is the same, up to $o(t)$, as that involving the pure Monge-Ampere term $(dd^cP(a))^d$. Again we write
$$F(t):=\mathcal E(P(a+tu)) =\mathcal E(P(a+tu),P(a))$$
$$=\int [P(a+tu)-P(a)][(dd^c P(a+tu))^d +... + (dd^c P(a))^d].$$
Another interpretation of (\ref{item1}) is that to prove the differentiability of $\mathcal E \circ P$, we can replace $\mathcal E$ by its ``linearization'' at $P(a)$. As in the previous arguments, we only give the proof at $t=0$ and for the one-sided limit in (\ref{Pdifftcor}) as $t\to 0^+$. Note in this next result we don't require smoothness of the perturbation $u$.

\begin{lemma} 
\label{keylemma2}
For an admissible weight $a$ on $E$ and $u\in C(E)$, let
$$F(t)=\mathcal E(P(a+tu)) $$
$$=\int [P(a+tu)-P(a)][(dd^c P(a+tu))^d +... + (dd^c P(a))^d]$$
and
$$G(t):= (d+1)\int [P(a+tu)-P(a)](dd^c P(a))^d.$$ 
Then
$$\lim_{t\to 0^+} \frac{F(t)-F(0)}{t}=\lim_{t\to 0^+} \frac{G(t)-G(0)}{t}=\int u(dd^cP(a))^d.$$
\end{lemma}

\begin{proof} Note that $F(0)=\mathcal E(P(a))=0$ and $G(0)=0$. Next, by concavity of $P$ (recall (\ref{projcon})) and linearity of 
$f\to \int f (dd^c P(a))^d$, the function $G(t)$ is concave so that 
$$A:=\lim_{t\to 0^+} \frac{G(t)-G(0)}{t}$$
exists. By concavity of $\mathcal E$, we have (recall (\ref{diff02}))
$$\mathcal E(P(a+tu))\leq \mathcal E(P(a))+<\mathcal E'(P(a)), P(a+tu)-P(a)>;$$
i.e., from (\ref{gconcave}) with $u_1=P(a), \ u_2 = P(a+tu)$ and $v= P(a)$, 
$$\mathcal E(P(a+tu))\leq \mathcal E(P(a))+ (d+1)\int [P(a+tu)-P(a)](dd^c P(a))^d].$$
Thus
$$\limsup_{t\to 0^+} \frac{F(t)-F(0)}{t}\leq A.$$
We prove 
$$\liminf_{t\to 0^+} \frac{F(t)-F(0)}{t}\geq A.$$
Since $A:=\lim_{t\to 0^+} \frac{G(t)-G(0)}{t}$ exists, given $\epsilon >0$ we can choose $\delta >0$ sufficiently small so that 
$$\frac{G(\delta)-G(0)}{\delta}=\frac{d+1}{\delta}\int [P(a+\delta u)-P(a)](dd^c P(a))^d\geq A -\epsilon;$$
i.e.,
$$(d+1)\int [P(a+\delta u)-P(a)](dd^c P(a))^d\geq \delta (A -\epsilon).$$
From Proposition \ref{diffprop}, for $t>0$ sufficiently small we have
$$\frac{\mathcal E(P(a) +t[P(a+\delta u)-P(a)])-\mathcal E(P(a))}{t}$$
$$\geq (d+1)\int [P(a+\delta u)-P(a)](dd^cP(a))^d-\delta \epsilon;$$
i.e.,
$$\mathcal E((1-t)P(a) +tP(a+\delta u))=\mathcal E(P(a) +t[P(a+\delta u)-P(a)])$$
$$\geq \mathcal E(P(a))+t(d+1) \int [P(a+\delta u)-P(a)](dd^cP(a))^d-t\delta \epsilon.$$
Combining these last two inequalities, we have
$$\mathcal E((1-t)P(a) +tP(a+\delta u))\geq \mathcal E(P(a))+t\delta A -2t\delta \epsilon.$$
By concavity of $P$,
$$P(a+t\delta u)=P((1-t)a +t(a+\delta u ))\geq (1-t)P(a) +tP(a +\delta u)$$
so that, by monotonicity of $\mathcal E$,
$$\mathcal E (P(a+t\delta u))\geq \mathcal E((1-t)P(a) +tP(a +\delta u))\geq \mathcal E(P(a))+t\delta A -2t\delta \epsilon$$
for $t>0$ sufficiently small. Thus,
$$\liminf_{t\to 0^+} \frac{F(t)-F(0)}{t}\geq A - 2\epsilon$$
for all $\epsilon >0$, yielding the result.

We now finish the proof of Proposition \ref{mainprop} and Proposition  \ref{diffpropcor} by verifying that $A= \int u(dd^cP(a))^d$. The proof was essentially given in the verification of (\ref{smooth5.7at0}) assuming (\ref{item1}) and (\ref{item2}); for the reader's convenience, we recall the details. We write $S_a:=$supp$(dd^cP(a))^d$.
For each $t$, set
$$D(t):=\{z\in \CC^d: P(a+tu)(z) = a(z) +tu(z)\}.$$
This is a bounded set. Since $P(a)=a \ (dd^cP(a))^d \ \hbox{a.e. on } \ S_a\subset D(0),$
$$\int [P(a+tu)-P(a)] (dd^cP(a))^d = \int_{S_a} [P(a+tu)-P(a)] (dd^cP(a))^d$$
$$=\int_{D(t)\cap S_a} [P(a+tu)-P(a)] (dd^cP(a))^d$$
$$ + \int_{S_a\setminus D(t)} [P(a+tu)-P(a)](dd^cP(a))^d$$
$$=\int_{D(t)\cap S_a} [a+tu-a] (dd^cP(a))^d + \int_{S_a\setminus D(t)} [P(a+tu)-P(a)] (dd^cP(a))^d$$
$$=\int_{D(t)\cap S_a} tu (dd^cP(a))^d + \int_{S_a\setminus D(t)} [P(a+tu)-P(a)] (dd^cP(a))^d$$
$$=\int_{S_a}tu (dd^cP(a))^d + \int_{S_a\setminus D(t)} [P(a+tu)-P(a)-tu] (dd^cP(a))^d.$$
Now we use the observation (\ref{loclip}) (or (\ref{loclip2})) to see that 
$$|P(a+tu) - P(a)-tu|=0(t)$$
on the bounded set $S_a\setminus D(t)$; then the conclusion of the lemma follows from Lemma \ref{mamass}.
\end{proof}

\begin{remark}
There is a nice interpretation of the differentiabilty of $\mathcal E\circ P$ in one dimension (see section 9.3 of \cite{[BB]}). Indeed, if $d=1$ and  $K$ is a closed subset of the unit disc in $\CC$ with admissible weight $Q=-\log w$, then, for $T=\{z\in \CC:|z|=1\}$ we have $V_T(z)=\log^+|z|$ so that $V_T=0$ on $K\cup T$. Taking $v=V_T$, we have
$$(\mathcal E\circ P)(Q)=\mathcal E(V_{K,Q}^*,V_T)=\int_{\CC}(V_{K,Q}^*-V_T)dd^c(V_{K,Q}^*+V_T)$$
$$=\int_{\CC}V_{K,Q}^*dd^c(V_{K,Q}^*)+\int_{\CC}V_{K,Q}^*dd^cV_T$$
$$=\int_{\CC}Qdd^c(V_{K,Q}^*)+\rho_{K,Q}$$
where $$\rho_{K,Q}=\lim_{|z|\to \infty} [V_{K,Q}^*(z)-\log |z|]$$ is the Robin constant of $V_{K,Q}^*$. Let $\mathcal M(K)$ denote the probability measures supported in $K$. It is classical that $\mathcal E(V_{K,Q}^*,V_T)$ coincides with the minimal {\it weighted logarithmic energy}
$$\inf_{\mu\in \mathcal M(K)} \int_K \int_K \log \frac{1}{|x-y|w(x)w(y)}d\mu(x)d\mu(y)$$
$$=\inf_{\mu\in \mathcal M(K)} \bigl[\mathcal I(\mu)+2\int_K Qd\mu\bigr]$$
over $\mu \in \mathcal M(K)$ where $\mathcal I(\mu)=\int_K \int_K \log \frac{1}{|x-y|}d\mu(x)d\mu(y)$ is the (unweighted) logarithmic energy of $\mu$. It is also classical that 
$$\mu\in \mathcal M(K) \to \mathcal I(\mu)\in \RR$$
is strictly convex on $\mathcal M(K)$; i.e., if $0<t<1$, then 
$$\mathcal I(t\mu_1+(1-t)\mu_2) < t\mathcal I(\mu_1)+(1-t) \mathcal I(\mu_2)$$
for $\mu_1,\mu_2\in \mathcal M(K)$. In a sense that can be made precise, the strict convexity of $\mu\to \mathcal I(\mu)$ on $\mathcal M(K)$ is related to the differentiability of the {\it Legendre transform} $\mathcal I^*$ of $\mathcal I$ on $C(K)$, where, for $p\in C(K)$,
$$\mathcal I^*(p):=\sup_{\mu\in \mathcal M(K)}[\int _K pd\mu -\mathcal I(\mu)]=-\inf_{\mu\in \mathcal M(K)}[\mathcal I(\mu)-\int_K pd\mu].$$
Setting $p=-2Q$, we have
$$\mathcal I^*(-2Q)=-\inf_{\mu\in \mathcal M(K)}[\mathcal I(\mu)+2\int_K Qd\mu]=-\mathcal E(V_{K,Q}^*,V_T)=-(\mathcal E\circ P)(Q).$$

\end{remark}

We record an integrated version of Proposition \ref{mainprop} and Proposition \ref{diffpropcor} which we will use. 

\begin{proposition} 
\label{smoothample}
For admissible weights $a,b\in C^2(E)$ on an unbounded closed set $E$ satisfying (\ref{unbhyp}),
\begin{equation}
\label{smooth5.7}
\mathcal E(P(b),P(a))=(d+1)\int_{t=0}^1 dt \int_{\CC^d} (b-a)(dd^cP(a+t(b-a)))^d;
\end{equation}
and for a compact set $E$ with admissible weight $a$ and $u\in C^2(E)$, 
\begin{equation}
\label{smooth5.8}
\mathcal E(P(a+u),P(a))=(d+1)\int_{t=0}^1 dt \int_{\CC^d} u(dd^cP(a+tu))^d.
\end{equation}
\end{proposition}

\begin{proof} We prove (\ref{smooth5.7}) as (\ref{smooth5.8}) is similar. We begin with Proposition \ref{mainprop} using $v=P(a)$ so that 
$F(t)={\mathcal E}(P(a+t(b-a)),P(a)))$ and (\ref{Pdifft}) becomes
$$F'(t)=(d+1)\int_{\CC^d} (b-a)(dd^cP(a+t(b-a)))^d.$$
Integrating this expression from $t=0$ to $t=1$ gives (\ref{smooth5.7}) since $F(1)-F(0)=\mathcal E(P(b),P(a))$.
\end{proof}

\subsection {Proof of the Main Theorem.} \label{pomt}

Following the ideas in \cite{[BB]}, \cite{BBnew}, \cite{[BB2]} and \cite{[BBN]}, given a closed set $K$, an admissible weight $w$ on $K$, and a function $u\in C(K)$, we consider the weight $w_t(z):=w(z)\exp(-tu(z)),$ $t\in\RR,$ and let $\{\mu_n\}$ be a sequence of measures on $K$. Fixing a basis $\beta_n:=\{p_1,...,p_N\}$ of ${\mathcal P}_n$, we set
\begin{equation}\label{fn}
f_n(t):=-{(d+1)\over 2 dnN}\log\,{\rm det}(G_n^{\mu_n,w_t})
\end{equation}
where $G_n^{\mu_n,w_t}=G_n^{\mu_n,w_t}(\beta_n)$ and we begin with the following general result (see Lemma 6.4 in [BB]).

\begin{lemma} \label{1stderiv}We have
\[ f_n'(t)={d+1\over dN}\int_K u(z)B_n^{\mu_n,w_t}(z)d\mu_n.\]
\end{lemma}
\begin{proof}
Recall that $G_n^{\mu_n,w_t}$ is a positive definite Hermitian matrix; hence it can be diagonalized by a unitary matrix and we can define $\log (G_n^{\mu_n,w_t})$. Using $\log\,{\rm det}(G_n^{\mu_n,w_t})={\rm trace}\log(G_n^{\mu_n,w_t})$, we calculate
\begin{eqnarray*}
\frac{2dnN}{d+1}f_n'(t)&=&-{d\over dt}{\rm trace}\left(\log(G_n^{\mu_n,w_t})\right) \\
&=& -{\rm trace} \left({d\over dt}\log(G_n^{\mu_n,w_t})\right) \\
&=& -{\rm trace}\left( (G_n^{\mu_n,w_t})^{-1}{d\over dt} G_n^{\mu_n,w_t}\right) \\
\end{eqnarray*}
$$=2n\, {\rm trace}  \left( (G_n^{\mu_n,w_t})^{-1} \left[\int_K p_i(z)\overline{p_j(z)}u(z)w(z)^{2n}\exp(-2ntu(z))d\mu_n\right]\right).$$
As in the proof of Proposition \ref{KW} we use $${\rm trace} (ABC) = {\rm trace} (CAB)= CAB$$ to write the previous line as
\begin{eqnarray*}
&=&2n\int_K P^*(z)(G_n^{\mu_n,w_t})^{-1}P(z) u(z)w(z)^{2n}\exp(-2ntu(z))d\mu_n\\
&=&2n\int_K u(z) P^*(z)(G_n^{\mu_n,w_t})^{-1}P(z) w_t(z)^{2n} d\mu_n \\
&=& 2n\int_K u(z) B_n^{\mu_n,w_t}(z) d\mu_n
\end{eqnarray*}
where the last equality follows from the remark (\ref{OtherKn}):
$$
w^{2n}P^*(G_n^{\mu_n,w_t})^{-1}P = B_n^{\mu_n,w_t}.
$$
\end{proof}

We are ready for the proof of Theorem \ref{keycn}.

\begin{proof} We begin in the $L^2-$case  $E=E'=\CC^d$ and $\phi, \phi'\in C^2(\CC^d)$ strongly admissible. We note that (\ref{unbhyp}) holds for then all of the weights $\phi+t(\phi'-\phi)$ are strongly admissible with a uniform $\epsilon$ (recall (\ref{stradm})). 
Here, we are using $d\mu=d\mu'=\omega_d$; i.e., Lebesgue measure; recall from (\ref{weightedleb}) that we have a weighted Bernstein-Markov property in this setting. Take $u=\phi'-\phi$ which is, in particular, continuous. {\it We first assume that $\phi'=\phi$ outside a ball ${\mathcal B}_R$ for some $R$}; i.e., $u$ has compact support. For $0\leq t\leq 1$ let 
$$\phi_t:=\phi +tu= \phi +t(\phi'-\phi)=(1-t)\phi + t\phi'$$
so that $\phi_0=\phi$ and $\phi_1=\phi'$; equivalently, $w_t(z):=w(z)\exp(-tu(z))$ (note $w_0 =w=e^{-\phi}$ and $w_1=w'=e^{-\phi'}$). Then from Theorem \ref{bermancn} and remark (\ref{weakstarleb}), for each $t$, 
$$\frac{1}{N}B_{n,\phi+tu}\cdot  \omega_d \to \frac{1}{(2\pi)^d}(dd^c  P(\phi+tu))^d \ \hbox{weak}-*.$$
Now set $$f_n(t):= -{(d+1)\over 2 dnN}\log\,{\rm det}(G_n^{\mu,w_t}(\beta_n))$$ where $\mu=\mu_n:= \omega_d$ for all $n$ and the basis $\beta_n:=\{p_1,...,p_N\}$ of ${\mathcal P}_n$ is chosen to be an orthonormal basis with respect to the weighted $L^2-$norm $p\to ||w^np||_{L^2(\mu)}$. Then $G_n^{\mu,w}(\beta_n)$ is the $N\times N$ identity matrix so that we have $f_n(0)=0$; and, using Lemma \ref{1stderiv} and the fact that $u$ has compact support, 
$$\lim_{n\to \infty}\frac{d}{d+1} f_n'(t) = \int u\frac{1}{(2\pi)^d}(dd^c  P(\phi+tu))^d $$
(recall our notation for the $n-$th Bergman function in using Lebesgue measure and a global, admissible weight $\phi+tu$ on $\CC^d$ is $B_{n,\phi+tu}$). 
We now integrate the corresponding expression for $f_n'(t)$ from $t=0$ to $t=1$:
$$\frac{d}{d+1}[f_n(1) -f_n(0)]=\frac{d}{d+1}[f_n(1)]= \frac{-1}{2nN}\log\,{\rm det}(G_n^{\mu,w'}(\beta_n))$$
$$=\frac{-1}{2nN}[\mathcal B^{2}(\CC^d,\mu,n\phi):\mathcal B^{2}(\CC^d,\mu,n\phi')] \ \hbox{(from (\ref{gramvolume}))}$$
$$=\frac{1}{N}\int_{t=0}^1 dt \int  B_{n,\phi+tu}(\phi-\phi')\omega_d$$
$$\to \int_{t=0}^1 dt \int (\phi-\phi')\frac{1}{(2\pi)^d}(dd^c  P(\phi+tu))^d.$$
But by (\ref{smooth5.7}), since (\ref{unbhyp}) holds, 
$$(d+1)\int_{t=0}^1 dt \int (\phi-\phi')\frac{1}{(2\pi)^d}(dd^c  P(\phi+tu))^d=\frac{1}{(2\pi)^d}\mathcal E(P(\phi'),P(\phi))$$
which proves Theorem \ref{keycn} in the $L^2-$case when $E=E'=\CC^d$ with $\phi, \phi'\in C^2(\CC^d)$ strongly admissible, $\phi'=\phi$ outside a ball ${\mathcal B}_R$, and $d\mu=d\mu'=\omega_d$. By the weighted Bernstein-Markov property this also proves the $L^{\infty}-$case when $E=E'=\CC^d$ and $\phi, \phi'\in C^2(\CC^d)$ are strongly admissible with $\phi'=\phi$ outside a ball ${\mathcal B}_R$. 

Next, we claim that the $L^{\infty}-$case when $E=E'=\CC^d$ and $\phi, \phi'\in C^2(\CC^d)$ are strongly admissible follows, without the assumption that $\phi'=\phi$ outside a ball ${\mathcal B}_R$. For recall from (\ref{equivdef}) that $$P(\phi)=P_{S_w}(\phi|_{S_w})=\sup \{\frac{1}{{\rm deg }p} \log |p|: p\in \cup_n \mathcal P_n, \ ||pe^{-{\rm deg }p \phi}||_{S_w}\leq 1\}$$
where $S_w=\hbox{supp}(dd^cP(\phi))^d$ is compact; moreover, 
for $p_n\in \mathcal P_n$, 
$$||p_ne^{-n\phi}||_{S_w}=||p_ne^{-n\phi}||_{\CC^d}$$
so that 
$${\mathcal B}^{\infty}(S_w,n\phi|_{S_w})={\mathcal B}^{\infty}(\CC^d,n\phi).$$
Thus modifyiing $\phi, \phi'$ outside a large ball in such a way to make them equal outside a perhaps larger ball, we don't change the $L^{\infty}-$ball volume ratios nor the extremal functions $P(\phi), P(\phi')$. By the weighted Bernstein-Markov property (\ref{weightedleb}) this also proves the $L^{\infty}-$case when $E=E'=\CC^d$ and $\phi, \phi'\in C^2(\CC^d)$ are strongly admissible, without the assumption that $\phi'=\phi$ outside a ball ${\mathcal B}_R$.

To prove Theorem \ref{keycn} in the general case when $E,E'\subset \CC^d$ are closed with admissible weights $\phi,\phi'$, we consider the $L^{\infty}$ situation only; the $L^2-$case follows from the definition of the weighted Bernstein-Markov property. We claim that by the cocycle property for the ball volume ratios $[A:B]$ and energies $\mathcal E(u_1,u_2)$, we may assume that one of the sets is $\CC^d$  
with a strongly admissible $C^2(\CC^d)$ weight $\hat \phi$. For, using the notation $P_E(\phi):=V_{E,\phi}^*$, we have
$$\mathcal E(P_E(\phi),P_{E'}(\phi')) + \mathcal E(P_{E'}(\phi'),P_{\CC^d}(\hat \phi)) + \mathcal E(P_{\CC^d}(\hat \phi),P_E(\phi))=0$$
so that
$$\mathcal E(P_E(\phi),P_{E'}(\phi')) =- \mathcal E(P_{E'}(\phi'),P_{\CC^d}(\hat \phi)) + \mathcal E(P_E(\phi),P_{\CC^d}(\hat \phi)).$$
Both terms on the right have the second term being $P_{\CC^d}(\hat \phi)$. Similarly, with respect to the ball volume ratios, for each $n$ we have
$$[{\mathcal B}^{\infty}(E,n\phi):{\mathcal B}^{\infty}(E',n\phi')] $$
$$=-[{\mathcal B}^{\infty}(E',n\phi'): {\mathcal B}^{\infty}(\CC^d,n\hat \phi)]+ [{\mathcal B}^{\infty}(E,n\phi):{\mathcal B}^{\infty}(\CC^d,n\hat \phi)].$$

Now to deduce the case where one of the sets is $\CC^d$  
with a strongly admissible $C^2(\CC^d)$ weight $\hat \phi$ and the other is a general closed set $E$ with admissible weight $\phi$ from the case where both sets are $\CC^d$ 
with strongly admissible $C^2(\CC^d)$ weights $\hat \phi, \psi$, we first observe that we may assume $E$ is compact (i.e., bounded). For
recall again from (\ref{equivdef}) that if $w=e^{-\phi}$, 
$$P_E(\phi)=P_{S_w}(\phi|_{S_w})=\sup \{\frac{1}{{\rm deg }p} \log |p|: p\in \cup_n \mathcal P_n, \ ||pe^{-{\rm deg }p \phi}||_{S_w}\leq 1\}$$
where $S_w=\hbox{supp}(dd^cP_E(\phi))^d$ is compact; moreover, 
for $p_n\in \mathcal P_n$, 
$$||p_ne^{-n\phi}||_{S_w}=||p_ne^{-n\phi}||_E$$
so that 
$${\mathcal B}^{\infty}(S_w,n\phi|_{S_w})={\mathcal B}^{\infty}(E,n\phi).$$

Thus we assume $E$ is compact; since $V_{E,\phi}^*\in L^+(\CC^d)$, we can also assume $\phi$ is bounded above on $E$. We take a large ball $\mathcal B_R$ containing $E$ and extend $\phi$ from $E$ to $\hat \psi$ on $\CC^d$:
$$\hat \psi := \phi \ \hbox{on} \ E; \ \hat \psi =2\log R  \ \hbox{on} \ \mathcal B_R -E; \ \hat \psi =2\log |z| \ \hbox{on} \ \CC^d- \mathcal B_R$$
We have $\hat \psi$ is lsc and by taking $R$ sufficiently big $P_{\CC^d}(\hat \psi)=P_E(\phi)$; then   we take a sequence of strongly admissible $C^2(\CC^d)$ weights $\{\phi_j\}$ with $\phi_j \uparrow \hat \psi$. We can apply the first case of Theorem 3.1 to $(\CC^d, \phi_j)$ and $(\CC^d,\hat \phi)$ to conclude
$$\frac{1}{2nN} [\mathcal B^{\infty}(\CC^d,n \phi_j):\mathcal B^{\infty}(\CC^d,n\hat \phi)]\to \frac{1}{(d+1)(2\pi)^d}\mathcal E(P_{\CC^d}(\phi_j), P_{\CC^d}(\hat \phi))$$
as $n\to \infty$. But $\phi_j \uparrow \hat \psi$ implies $P_{\CC^d}(\phi_j)\uparrow P_{\CC^d}(\hat \psi)=P_E(\phi)$ and hence 
\begin{equation}
\label{ephij}
\mathcal E(P_{\CC^d}(\phi_j), P_{\CC^d}(\hat \phi)) \ \hbox{converges to} \ \mathcal E(P_E(\phi), P_{\CC^d}(\hat \phi))
\end{equation}
as $j\to \infty$ by Lemma \ref{BT63}. We want to conclude that
\begin{equation}
\label{ephin}
\frac{1}{2nN} [\mathcal B^{\infty}(E,n \phi):\mathcal B^{\infty}(\CC^d,n\hat \phi)]\to \frac{1}{(d+1)(2\pi)^d}\mathcal E(P_E(\phi), P_{\CC^d}(\hat \phi))
\end{equation}
as $n\to \infty$. To make this precise, first observe that 
$$\frac{1}{(d+1)(2\pi)^d}\mathcal E(P_{\CC^d}(\phi_j), P_{\CC^d}(\hat \phi))$$
$$=\lim_{n\to \infty} \frac{1}{2nN} [\mathcal B^{\infty}(\CC^d,n \phi_j):\mathcal B^{\infty}(\CC^d,n\hat \phi)]$$
$$\leq \liminf_{n\to \infty} \frac{1}{2nN} [\mathcal B^{\infty}(E,n \phi):\mathcal B^{\infty}(\CC^d,n\hat \phi)]$$
$$\leq \limsup_{n\to \infty} \frac{1}{2nN} [\mathcal B^{\infty}(E,n \phi):\mathcal B^{\infty}(\CC^d,n\hat \phi)]$$
since $P_{\CC^d}(\phi_j)\uparrow P_E(\phi)$ implies from (\ref{bvmon}) that 
$$[\mathcal B^{\infty}(\CC^d,n \phi_j):\mathcal B^{\infty}(\CC^d,n\hat \phi)] \leq [\mathcal B^{\infty}(E,n \phi):\mathcal B^{\infty}(\CC^d,n\hat \phi)].$$
But if we take a sequence of smooth, admissible weights $\{\psi_j\}$ on $\CC^d$ with $\psi_j \downarrow P_E(\phi)$ -- for example, we may take $\psi_j =(1+\epsilon_j)[(P_E(\phi))_{\epsilon_j}]$ where $(P_E(\phi))_{\epsilon_j}$ is a smoothing of $P_E(\phi)$ -- then $P_{\CC^d}(\psi_j)\downarrow P_E(\phi)$ and
$$\limsup_{n\to \infty} \frac{1}{2nN} [\mathcal B^{\infty}(E,n \phi):\mathcal B^{\infty}(\CC^d,n\hat \phi)]$$
$$\leq \lim_{n\to \infty} \frac{1}{2nN} [\mathcal B^{\infty}(\CC^d,n \psi_j):\mathcal B^{\infty}(\CC^d,n\hat \phi)]$$
again by (\ref{bvmon}); this limit equals
$$\frac{1}{(d+1)(2\pi)^d}\mathcal E(P_{\CC^d}(\psi_j), P_{\CC^d}(\hat \phi))$$
by applying the first case of Theorem 3.1 this time to $(\CC^d, \psi_j)$ and $(\CC^d,\hat \phi)$. Now 
\begin{equation}
\label{epsi}
\mathcal E(P_{\CC^d}(\psi_j), P_{\CC^d}(\hat \phi)) \ \hbox{converges to} \ \mathcal E(P_E(\phi), P_{\CC^d}(\hat \phi))
\end{equation}
as $j\to \infty$ by (\ref{downlemma}). Then (\ref{ephij}) and (\ref{epsi}) imply (\ref{ephin}) which completes the proof of Theorem \ref{keycn}.
\end{proof}

\section{Proof of Rumely.}
\label{bf}
We use Proposition \ref{weightedtd}, in conjunction with Theorem \ref{keycn} and the observation (\ref{gramvolume}), to prove an ``energy version'' of Rumely's formula.

\begin{theorem} \label{energyrumely}
Let $K\subset \CC^d$ be compact and $w$ an admissible weight on $K$. Then
\begin{equation} \label{enrum}
-\log \delta^w(K)=  \frac{1}{d(2\pi)^d}\mathcal E(V^*_{K,Q},V_T).
\end{equation}
\end{theorem}

\begin{proof} Note that for the unit torus $T$, 
$$V_T(z_1,...,z_d)=\max_{j=1,...,d} \log^+|z_j|$$
and that the standard basis monomials $e_1,...,e_N$ form an orthonormal basis $\beta_n$ for $\mathcal P_n$ with respect to $\mu_T:=\frac{1}{(2\pi)^d}(dd^cV_T)^d$, which is the standard Haar measure on $T$. 

We first assume that there exists a measure $\nu$ on $K$ such that 
$(K,\nu,Q)$ satisfies a weighted Bernstein-Markov property. Now from (\ref{gramvolume}), 
$$\log \det G_n^{\nu,w}(\beta_n) = [\mathcal B^{2}(T,\mu_T,n\cdot 0):\mathcal B^{2}(K,\nu,-n\log w)]$$
so that, on the one hand, by Proposition \ref{weightedtd},  
$$\lim_{n\to \infty} \frac{d+1}{2dnN} \cdot \log \det G_n^{\nu,w}(\beta_n) = \log \delta^w(K);$$
while on the other hand, by Theorem \ref{keycn}, 
$$\lim_{n\to \infty}\frac{1}{2nN} [\mathcal B^{2}(T,\mu_T,n\cdot 0):\mathcal B^{2}(K,\nu,-n\log w)]$$
$$= \frac{1}{(d+1)(2\pi)^d}\mathcal E(V_T, V^*_{K,Q})$$
so that, since $\mathcal E(V_T, V^*_{K,Q})=-\mathcal E(V^*_{K,Q},V_T)$, we obtain 
$$-\log \delta^w(K)=  \frac{1}{d(2\pi)^d}\mathcal E(V^*_{K,Q},V_T)$$
as desired.

In the general case, we can find a sequence of locally regular compacta $\{K_j\}$ decreasing to $K$ and a sequence of weights 
$\{w_j\}$ with $w_j$ continuous and admissible on $K_j$ such that $w_{j+1}\leq w_j|_{K_{j+1}}$ and $w_j \downarrow w$ on $K$. Then by (\ref{wtdapprox}) we have
$$V_{K_j,Q_j} \uparrow V_{K,Q}$$
so that from Lemma \ref{BT63} we have
$$\lim_{j\to \infty} \mathcal E(V^*_{K_j,Q_j},V_T) = \mathcal E(V^*_{K,Q},V_T)$$
where $Q_j := -\log w_j$. 
Thus to verify (\ref{enrum}), it suffices to prove that 
\begin{equation} \label{decrlimit}\lim_{j\to \infty}\delta^{w_j}(K_j) = \delta^w(K).
\end{equation}
We defer the proof of (\ref{decrlimit}) to Appendix 1 (Proposition \ref{polypoly}).
\end{proof}

\section{Asymptotic weighted Fekete measures, weighted optimal measures and strong Bergman asymptotics.}\label{sec:fob}

We will apply the following calculus lemma to an appropriate sequence of real-valued functions $\{f_n\}$ in order to prove the main applications of the differentiability result, Proposition \ref{diffpropcor}.

\begin{lemma}\label{calc}
Let $f_n$ be a sequence of concave functions on $\RR$ and let $g$ be a function on $\RR$. Suppose
$$\liminf f_n(t) \geq g(t) \ \hbox{for all} \ t \ \hbox{and} \ \lim f_n(0) = g(0)$$
and that $f_n$ and $g$ are differentiable at $0$. Then $\lim f_n'(0) = g'(0)$.
\end{lemma}

\begin{proof}
 By concavity of $f_n$, we have
$$f_n(0) +t f' _n(0) \geq f_n(t).$$
From the hypotheses, we see that  
$$\liminf_{n \to \infty}tf'_n(0)\geq g(t)-g(0).$$ 
For $t>0$, we get
$$\liminf_{n \to \infty}f'_n(0)\geq \lim_{t\to 0^+}\frac{g(t)-g(0)}{t}
=g'(0);$$
for $t<0$, we observe first that 
$$\limsup_{n \to \infty}tf'_n(0) \geq \liminf_{n \to \infty}tf'_n(0)\geq g(t)-g(0)$$
so
$$\limsup_{n \to \infty}f'_n(0)\leq \lim_{t\to 0^-}\frac{g(t)-g(0)}{t}
=g'(0).$$
\end{proof}

\noindent Note that here ``differentiable at the origin'' means the usual (two-sided) limit of the difference quotients exists; the conclusion is not true with one-sided limits.

As in Lemma \ref{1stderiv} in section \ref{pomt}, given a closed set $K$, an admissible weight $w$ on $K$, and a function $u\in C(K)$, we consider the weight $w_t(z):=w(z)\exp(-tu(z)),$ $t\in\RR,$ and let $\{\mu_n\}$ be a sequence of measures on $K$. Fixing a basis $\beta_n:=\{p_1,...,p_N\}$ of ${\mathcal P}_n$, define (see (\ref{fn}))
$$f_n(t)=-{(d+1)\over 2 dnN}\log\,{\rm det}(G_n^{\mu_n,w_t})$$
where $G_n^{\mu_n,w_t}=G_n^{\mu_n,w_t}(\beta_n)$. Then $f_n(0)= -{(d+1)\over 2 dnN}\log\,{\rm det}(G_n^{\mu_n,w})$. 
The next result was proved in a different way in [BN], Lemma 2.2. The proof we present is in \cite{bblw}.

\begin{lemma} \label{2ndderiv}
The functions $f_n(t)$ are concave, i.e., $f_n''(t)\le0.$
\end{lemma}
\begin{proof} First, let
\[g_n(h):=\frac{2dnN}{d+1}f_n(t+h)\]
so that $\displaystyle{f_n''(t)={(d+1)\over 2 dnN}g_n''(0).}$ Also, note that if we change the
basis $B_n=\{p_1,\cdots,p_N\}$ to $C_n:=\{q_1,\cdots,q_N\}$ by $p_i=\sum_{j=1}^Na_{ij}q_j,$ 
then the Gram matrices transform (recall (\ref{grambasis}) by
\[G_n^{\mu_n,w_t}(B_n)=AG_n^{\mu_n,w_t}(C_n)A^*\]
where $A=[a_{ij}]\in \CC^{N\times N}.$ Hence,
$$ g_n(h)=-\log({\rm det}(G_n^{\mu_n,w_{t+h}}(B_n)))$$
$$=-\log({\rm det}(G_n^{\mu_n,w_{t+h}}(C_n)))-\log(|{\rm det}(A)|^2)$$
and we see that the derivatives of $g_n$ are independent of the basis chosen.

Let us choose $C_n$ to be an orthonormal basis for ${\mathcal P}_n$ with respect to the
inner-product $\langle \cdot,\cdot\rangle_{\mu_n,nw_t}$ (recall (\ref{Wip})). 

Now, for convenience, write $G(h)=G_n^{\mu_n,w_{t+h}}$ and set $F(h)=\log(G(h))$ so
that $G(h)=\exp(F(h)).$ By our choice of basis $C_n$ we have $G(0)=I\in \CC^{N\times N},$ 
the identity matrix, and
$F(0)=[0]\in\CC^{N\times N},$ the zero matrix.
Then, (see e.g. [Bh, p. 311]),
\[{dG\over dh}={d\over dh}\exp(F(h))=\int_0^1e^{(1-s)F(h)}{dF\over dh}e^{sF(h)}ds.\]
In particular
\[{dG\over dh}(0)={dF\over dh}(0).\]
Further, 
\begin{eqnarray*}
{d^2G\over dh^2}&=&\int_0^1 \left\{\left[{d\over dh}e^{(1-s)F(h)}\right]{dF\over dh}e^{sF(h)}
+e^{(1-s)F(h)}{d^2F\over dh^2}e^{sF(h)} \right. \\
&& \left. \quad +e^{(1-s)F(h)}{dF\over dh}\left[{d\over dh}e^{sF(h)}\right] \right\}ds.
\end{eqnarray*}
Evaluating at $h=0,$ using the fact that $F(0)=[0],$ we obtain
\begin{eqnarray*}
{d^2G\over dh^2}(0)&=&\int_0^1 \left\{(1-s){dF\over dh}(0)\times {dF\over dh}(0)\times I
+I\times {d^2F\over dh^2}(0)\times I \right.\\
&&\left. \quad+I\times {dF\over dh}(0)\times s{dF\over dh}(0)\, \right\}ds\\
&=&\int_0^1 \left\{(1-s+s)\left({dF\over dh}(0)\right)^2+{d^2F\over dh^2}(0)\right\}\,ds \\
&=&\left({dF\over dh}(0)\right)^2+{d^2F\over dh^2}(0).
\end{eqnarray*}
Hence,
$${d^2F\over dh^2}(0){d^2G\over dh^2}(0)-({dF\over dh}(0))^2$$
$$=[\int_Kq_i(z)\overline{q_j(z)}(-2nu(z))^2w_t(z)^{2n}d\mu_n]$$
$$-[\int_Kq_i(z)\overline{q_j(z)}(-2nu(z))w_t(z)^{2n}d\mu_n]^2.$$
Since $g_n'(h)=\frac{d}{dh}[-\log({\rm det}(G(h))]$ and $$\log({\rm det}(G(h))={\rm trace}(\log(G(h)))={\rm trace}(F(h))$$ it follows that
\begin{eqnarray*}
g_n''(0)&=&-{\rm trace}\left(\left[\int_Kq_i(z)\overline{q_j(z)}(-2nu(z))^2w_t(z)^{2n}d\mu_n\right]\right)\\
&&\qquad+{\rm trace}\left(\left[\int_Kq_i(z)\overline{q_j(z)}(-2nu(z))w_t(z)^{2n}d\mu_n\right]^2\right)\\
&=&-\sum_{i=1}^N\int_K|q_i(z)|^2w_t(z)^{2n}(2nu(z))^2d\mu_n \\
&&\qquad+\sum_{i=1}^N\sum_{j=1}^N\left|\int_Kq_i(z)\overline{q_j(z)}w_t(z)^{2n}(2nu(z))d\mu_n\right|^2\\
&=&-\sum_{i=1}^N \left\{\int_K|q_i(z)|^2w_t(z)^{2n}(2nu(z))^2d\mu_n - \right. \\
&&\qquad  \left.\sum_{j=1}^N \left|\int_Kq_i(z)\overline{q_j(z)}w_t(z)^{2n}(2nu(z))d\mu_n\right|^2\right\}.
\end{eqnarray*}But notice that $\displaystyle{\int_Kq_i(z)\overline{q_j(z)}w_t(z)^{2n}(2nu(z))d\mu_n}$ is the
$j$th Fourier coefficient of the function $2nu(z)q_i(z)$ with respect to the orthonormal basis
$C_n,$ and also that $\displaystyle{\int_K|q_i(z)|^2w_t(z)^{2n}(2nu(z))^2d\mu_n}$ is the $L^2$ norm
squared of this same function. Hence, by Parseval's inequality, 
$g_n''(0)\le 0$. 
\end{proof}

Now let $\mu$ be a probability measure on $K$ and let $u\in C^2(K)$. Define 
$$g(t)=-\log(\delta^{w_t}(K))$$
so that $g(0)=-\log(\delta^{w}(K))$
From Proposition \ref{diffpropcor} and Theorem \ref{energyrumely},
$$g'(0)= {d+1\over d(2\pi)^d}\int_Ku(z)(dd^cV_{K,Q}^*)^d.$$
Note that for each $n$, $\mu_n$ is a candidate to be an optimal measure of order $n$ for $K$ and $w_t$. {\it For the rest of this section, in computing Gram matrices, we fix the standard monomial basis $\beta_n=\{e_1,...,e_N\}$ of ${\mathcal P}_n$}. Thus, if $\mu_n^t$ is an optimal measure of order $n$ for $K$ and $w_t$, we have 
$$\det G_n^{\mu_n,w_t} \leq \det G_n^{\mu_n^t,w_t}$$
and, from Proposition \ref{tfd},
$$\lim_{n\to \infty}  \frac{d+1}{2dnN} \cdot \log \det G_n^{\mu_n^t,w_t}= \log \delta^{w_t}(K).$$
Thus with 
$$f_n(t):=-{(d+1)\over 2 dnN}\log\,{\rm det}(G_n^{\mu_n,w_t})$$
as in (\ref{fn}),
$$\liminf f_n(t) \geq g(t) \ \hbox{for all} \ t.$$
From Lemma \ref{1stderiv}, we have
$$f_n'(0)={d+1\over dN}\int_K u(z)B_n^{\mu_n,w}(z)d\mu_n.$$
Using Lemma \ref{calc}, we have the following general result.

\begin{proposition}
\label{genres}
Let $K\subset \CC^d$ be compact with admissible weight $w$. Let $\{\mu_n\}$ be a sequence of probability measures on $K$ with the property that 
\begin{equation}\label{fnhyp}
\lim_{n\to \infty}{(d+1)\over 2 dnN}\log\,{\rm det}(G_n^{\mu_n,w})= \log(\delta^{w}(K))
\end{equation}
i.e., $\lim_{n\to \infty}f_n(0)=g(0)$. Then 
\begin{equation}\label{strongasym}
\frac{1}{N}B_n^{\mu_n,w} d\mu_n \to \mu_{K,Q}=\frac{1}{(2\pi)^d}(dd^cV_{K,Q}^*)^d \ \hbox{weak-}*.
\end{equation}
\end{proposition}

Note that since all $\mu_n$ are probability measures on $K$, to verify weak-* convergence, it suffices to test with $C^2-$functions on $K$. In particular, from Proposition \ref{weightedtd} we have the general strong Bergman asymptotic result.

\begin{corollary} \label{strongasymp}
{\bf [Strong Bergman Asymptotics]} If $(K,\mu,w)$ satisfies a weighted Bernstein-Markov inequality, then 
$$\frac{1}{N}B_n^{\mu,w} d\mu \to \mu_{K,Q} \ \hbox{weak-}*.
$$
\end{corollary}

Next, suppose $\mu_n$ is an optimal measure of order $n$ for $K$ and $w$. 

\begin{corollary}
{\bf [Weighted Optimal Measures]} 
Let $K\subset \CC^d$ be compact with admissible weight $w$. Let $\{\mu_n\}$ be a sequence of optimal measures for $K,w$. Then
$$\mu_n  \to \mu_{K,Q} \ \hbox{weak-}*.
$$
\end{corollary}
\begin{proof} We have $B_n^{\mu_n,w}=N$ a.e. $\mu_n$ on $K$ from Lemma \ref{MaxIsN} so that the result follows immediately from Proposition \ref{genres}, specifically, equation (\ref{strongasym}). Equivalently, using results from subsection \ref{gramsec}, 
$$f_n'(0)={d+1\over dN}\int_K u(z)B_n^{\mu_n,w}(z)d\mu_n={d+1\over d}\int_K u(z)d\mu_n.$$
Again, if $\mu_n^t$ is an optimal measure of order $n$ for $K$ and $w_t$, we have 
$$\det G_n^{\mu_n,w_t} \leq \det G_n^{\mu_n^t,w_t}$$
and hence
$$\liminf f_n(t) \geq g(t) \ \hbox{for all} \ t;$$
and for $t=0$ by Proposition \ref{tfd}
\[\lim_{n\to\infty}f_n(0)=-\log(\delta^{w}(K))=g(0).\]
By Theorem \ref{energyrumely}, 
\begin{equation}\label{RumelyFormula}
-\log(\delta^w(K))={1\over d (2\pi)^d}{\mathcal E}(V_{K,Q}^*,V_T).
\end{equation}
Thus from Lemma \ref{calc} and Proposition \ref{diffpropcor}
$$\lim_{n\to \infty} f_n'(0)=g'(0)={d+1\over d(2\pi)^d}\int_Ku(z)(dd^cV_{K,Q}^*)^d$$
which says that $\mu_n\to \mu_{K,Q}$ weak-*.
\end{proof}

Finally, we prove a strengthened version of the weighted Fekete conjecture.

\begin{corollary} \label{asympwtdfek} {\bf [Asymptotic Weighted Fekete Points]} Let $K\subset \CC^d$ be compact with admissible weight $w$. For each $n$, take points $x_1^{(n)},x_2^{(n)},\cdots,x_N^{(n)}\in K$ for which 
\begin{equation}\label{wam}
 \lim_{n\to \infty}\bigl[|VDM(x_1^{(n)},\cdots,x_N^{(n)})|w(x_1^{(n)})^nw(x_2^{(n)})^n\cdots w(x_N^{(n)})^n\bigr]^{{(d+1)\over  dnN}}=\delta^w(K)
\end{equation}
({\it asymptotically} weighted Fekete points) and let $\mu_n:= \frac{1}{N}\sum_{j=1}^N \delta_{x_j^{(n)}}$. Then
$$
\mu_n \to \mu_{K,Q} \ \hbox{weak}-*.
$$
\end{corollary}
\begin{proof} By direct calculation, we have $B_n^{\mu_n,w}(x_j^{(n)})=N$ for $j=1,...,N$ and hence a.e. $\mu_n$ on $K$. Indeed, this property holds for {\it any} discrete, equally weighted measure $\mu_n:= \frac{1}{N}\sum_{j=1}^N \delta_{x_j^{(n)}}$ with $$|VDM(x_1^{(n)},\cdots,x_N^{(n)})|w(x_1^{(n)})^nw(x_2^{(n)})^n\cdots w(x_N^{(n)})^n\not =0.$$ Again, the result follows immediately from Proposition \ref{genres}, specifically, equation (\ref{strongasym}). Alternately, if $\mu_n^t$ is an optimal measure of order $n$ for $K$ and $w_t$, we have 
$$\det G_n^{\mu_n,w_t} \leq \det G_n^{\mu_n^t,w_t}$$
and hence
$$\liminf f_n(t) \geq g(t) \ \hbox{for all} \ t;$$
finally, by definition of asymptotically weighted Fekete points, 
$$\lim_{n\to\infty}f_n(0)=-\log(\delta^{w}(K))=g(0).$$
Thus the same proof as in the previous proposition is valid to show $\mu_n\to \mu_{K,Q}$ weak-*. Indeed, in this case, the reader should note that the functions $f_n(t)$ are {\it affine} in $t$ so that $f_n''(t)=0$ is immediate.
\end{proof}

\section{Appendix 1: Transfinite diameter notions in $\CC^d$.}
	
	This section is adapted from \cite{[BL]}. We begin by considering a function $Y$ from the set of multiindices $\alpha \in {\bf N}^d$ to the nonnegative real numbers satisfying:
\begin{equation} \label{subadd} Y(\alpha + \beta) \leq Y(\alpha)\cdot Y(\beta) \ \hbox{for all} \ \alpha, \ \beta \in {\bf N}^d. \end{equation}
We call a function $Y$ satisfying (\ref{subadd}) {\it submultiplicative};  we have three main examples below. Let $e_1(z),...,e_j(z),...$ be a listing of the monomials
$\{e_i(z)=z^{\alpha(i)}=z_1^{\alpha_1}\cdots z_d^{\alpha_d}\}$ in
$\CC^d$ indexed using a lexicographic ordering on the multiindices $\alpha=\alpha(i)=(\alpha_1,...,\alpha_d)\in {\bf N}^d$, but with deg$e_i=|\alpha(i)|$ nondecreasing. We write $|\alpha|:=\sum_{j=1}^d\alpha_j$. 

We define the following integers:
\begin{enumerate}
\item $m_n^{(d)}=m_n:=$ the number of monomials $e_i(z)$ of degree at most $n$ in $d$ variables;
\item $h_n^{(d)}=h_n:=$ the number of monomials $e_i(z)$ of degree exactly $n$ in $d$ variables;
\item $l_n^{(d)}=l_n:=$ the sum of the degrees of the $m_n$ monomials $e_i(z)$ of degree at most $n$ in $d$ variables.
\end{enumerate}
Note that $m_n^{(d)} ={d+n \choose n}$ and $l_n=\frac{d+1}{dnN}$. We will also need 
$$r_n^{(d)}=r_n:=nh_n^{(d)}=n(m_n^{(d)}-m_{n-1}^{(d)})$$
which is the sum of the degrees of the $h_n$ monomials $e_i(z)$ of degree exactly $n$ in $d$ variables.

\noindent We have the following relations:

\begin{equation} \label{precrucial} m_n^{(d)} ={d+n \choose n}; \ h_n^{(d)}=m_n^{(d)}-m_{n-1}^{(d)}= {d-1+n \choose n}; \end{equation}
\begin{equation} \label{crucial} h_n^{(d+1)} =  {d+n \choose n}= m_n^{(d)}; \ l_n^{(d)} = d{d+n \choose d+1}=(\frac{d}{d+1})\cdot nm_n^{(d)}; \end{equation}
and
\begin{equation} \label{postcrucial} l_n^{(d)}=\sum_{k=1}^nr_k^{(d)}=\sum_{k=1}^n kh_k^{(d)}.\end{equation}
The elementary fact that the dimension of the space of homogeneous polynomials of degree $n$ in $d+1$ variables equals the dimension of the space of polynomials of degree at most $n$ in $d$ variables will be  crucial.

Let $K\subset \CC^d$ be compact. Here are three natural constructions associated to $K$:

\begin{enumerate}
\item {\it Chebyshev constants}: Define the class of polynomials $$P_i=P(\alpha(i)):=\{e_i(z)+\sum_{j<i}c_je_j(z)\};  $$
and the Chebyshev constants
$$Y_1(\alpha):= \inf \{||p||_K:p\in P_i\}.$$ 
We write $t_{\alpha,K}:=t_{\alpha(i),K}$ for a Chebyshev polynomial; i.e., $t_{\alpha,K} \in P(\alpha(i))$ and $||t_{\alpha,K}||_K=Y_1(\alpha)$.  
\item {\it Homogeneous Chebyshev constants}: Define the class of homogeneous polynomials $$P_i^{(H)}=P^{(H)}(\alpha(i)):=\{e_i(z)+\sum_{j<i, \ {\rm deg}(e_j)={\rm deg}(e_i)}c_je_j(z)\};$$
and the homogeneous Chebyshev constants
$$Y_2(\alpha):= \inf \{||p||_K:p\in P^{(H)}_i\}.$$
We write $t^{(H)}_{\alpha,K}:=t^{(H)}_{\alpha(i),K}$ for a homogeneous Chebyshev polynomial; i.e., $t^{(H)}_{\alpha,K} \in P^{(H)}(\alpha(i))$ and $||t^{(H)}_{\alpha,K}||_K=Y_2(\alpha)$.
\item {\it Weighted Chebyshev constants}: Let $w$ be an admissible weight function on $K$ and let 
$$Y_3(\alpha):= \inf \{||w^{|\alpha (i)|}p||_K: \ p\in P_i\}$$ 
be the weighted Chebyshev constants. Note we use the polynomial classes $P_i$ as in (1). We write $t^w_{\alpha,K}$ for a weighted Chebyshev polynomial; i.e., $t^w_{\alpha,K}$ is of the form $w^{\alpha(i)}p$ with $p\in P(\alpha(i))$ and 
$||t^w_{\alpha,K}||_K= Y_3(\alpha)$.
\end{enumerate}

Let
$\Sigma$ denote the standard
$(d-1)-$simplex  in $\RR^d$; i.e.,
$$\Sigma = \{\theta =(\theta_1,...,\theta_d)\in \RR^d: \sum_{j=1}^d\theta_j=1, \ \theta_j\geq 0, \ j=1,...,d\},$$
and let
$$\Sigma^0 :=\{\theta \in \Sigma:  \ \theta_j > 0, \ j=1,...,d\}.$$
Given a submultiplicative function $Y(\alpha)$, define, as with the above examples, a new function  
\begin{equation} \label{tau}\tau( \alpha):= Y(\alpha)^{1/|\alpha|}.\end{equation} 
An examination of lemmas 1, 2, 3, 5, and 6 in \cite {zah} shows that (\ref{subadd}) is the only property of the numbers $Y( \alpha)$ needed to establish those lemmas. That is,
we have the following results for $Y:{\bf N}^d \to \RR^+$ satisfying (\ref{subadd}) and the associated function $\tau( \alpha)$ in (\ref{tau}).

\begin{lemma} For  all $\theta \in \Sigma^0$, the limit
$$T(Y,\theta):= \lim_{\alpha/|\alpha|\to \theta} Y(\alpha)^{1/|\alpha|}=\lim_{\alpha/|\alpha|\to \theta}\tau(\alpha)$$
exists. \end{lemma}

\begin{lemma} The function $\theta \to T(Y,\theta)$ is log-convex on $\Sigma^0$ (and hence continuous). \end{lemma}

\begin{lemma}  Given $b\in \partial \Sigma$,
$$\liminf_{\theta \to b, \ \theta \in \Sigma^0}
T(Y,\theta) =\liminf_{i\to \infty, \ \alpha(i)/|\alpha(i)|\to b}\tau(\alpha(i)).$$
\end{lemma}

\begin{lemma}  Let $\theta(k):= \alpha(k)/|\alpha(k)|$ for $k=1,2,...$ and let $Q$ be a compact subset of $\Sigma^0$. Then
$$\limsup_{|\alpha|\to \infty} \{\log \tau(\alpha(k)) - \log T (Y(\theta(k))): |\alpha(k)|=\alpha, \ \theta(k)\in Q\}=0.$$
\end{lemma}

\begin{lemma} \label{lemma2.5} Define 
$$\tau(Y):= \exp \bigl[\frac{1}{\hbox{meas}(\Sigma)}\int_{\Sigma} \log T(Y,\theta)d\theta\bigr]$$
Then
$$\lim_{d\to \infty} \frac{1}{h_n}\sum_{|\alpha|=d}\log \tau(\alpha)= \log \tau (Y);$$
i.e., using (\ref{tau}),
$$\lim_{d\to \infty}\bigl[\prod_{|\alpha|= d}Y(\alpha)\bigr]^{1/dh_n} = \ \tau (Y).$$
 \end{lemma}

One can incorporate all of the $Y(\alpha)'$s for $|\alpha|\leq d$; this is the content of the next result. 

\begin{theorem}\label{geommean} We have 
$$\lim_{d\to \infty}\bigl[\prod_{|\alpha|\leq d}Y(\alpha)\bigr]^{1/l_n} \ \hbox{exists and equals} \ \tau (Y).$$
\end{theorem}

\begin{proof} Define the geometric means 
$$\tau^0_d := \bigl(\prod_{|\alpha| =d}\tau(\alpha)\bigr)^{1/h_n}, \ d=1,2,...$$
The sequence
$$\log \tau^0_1,\log \tau^0_1,... (r_1 \ \hbox{times}), ... , \log \tau^0_d, \log \tau^0_d,... (r_d \ \hbox{times}), ...$$
converges to $\log \tau(Y)$ by the previous lemma; hence the arithmetic mean of the first $l_n = \sum_{k=1}^d r_k$ terms (see (\ref{postcrucial})) converges to $\log \tau(Y)$ as well. Exponentiating this arithmetic mean gives
\begin{equation} \label{mainest}\bigl(\prod_{k=1}^d (\tau^0_k)^{r_k} \bigr)^{1/l_n}= \bigl(\prod_{k=1}^d\prod_{|\alpha|=k}\tau(\alpha)^k\bigr)^{1/l_n}=\bigl(\prod_{|\alpha|\leq d}Y(\alpha)\bigr)^{1/l_n}\end{equation}
and the result follows.
\end{proof}

Returning to our examples (1)-(3), example (1) was the original setting of Zaharjuta \cite{zah} which he utilized to prove the existence of the limit in the definition of the transfinite diameter $\delta(K)$ of a compact set $K\subset \CC^d$. Recall the notation
$$V_n =V_n(K):=\max_{\zeta_1,...,\zeta_n\in K}|VDM(\zeta_1,...,\zeta_n)|$$
and
\begin{equation} \label{tdlim}\delta(K)=\lim_{d\to \infty}V_{m_n}^{1/l_n}. \end{equation}
Zaharjuta \cite{zah} showed that the limit exists by showing that one has
\begin{equation} \label{zahthm}\delta(K)=\exp\bigl[\frac{1}{\hbox{meas}(\Sigma)}\int_{\Sigma^0}\log {\tau(K,\theta)}d\theta \bigr]  \end{equation} 
where $\tau(K,\theta)=T(Y_1,\theta)$ from (1); i.e., the right-hand-side of (\ref{zahthm}) is $\tau(Y_1)$. This follows from Theorem \ref{geommean} for $Y=Y_1$ and the estimate
$$\bigl(\prod_{k=1}^d (\tau^0_k)^{r_k} \bigr)^{1/l_n} \leq V_{m_n}^{1/l_n}   \leq (m_n!)^{1/l_n} \bigl(\prod_{k=1}^d (\tau^0_k)^{r_k} \bigr)^{1/l_n}$$
in \cite{zah} (compare (\ref{mainest})).

For a compact {\it circled} set $K\subset \CC^d$; i.e., $z\in K$ if and only if $e^{i\phi}z\in K, \ \phi \in [0,2\pi]$, one need only
consider homogeneous polynomials in the definition of the directional Chebyshev constants $\tau(K,\theta)$. In other words, in the notation of (1) and (2), $Y_1(\alpha)=Y_2(\alpha)$ for all $\alpha$ so that 
$$T(Y_1,\theta)=T(Y_2,\theta) \ \hbox{for circled sets} \  K.$$
This is because for such a set, if we write a polynomial $p$ of degree $d$ as $p=\sum_{j=0}^dH_j$ where $H_j$ is a homogeneous
polynomial of degree $j$, then, from the Cauchy integral formula, $||H_j||_K\leq ||p||_K, \ j=0,...,d$. Moreover, a slight modification of Zaharjuta's arguments prove the existence of the limit of appropriate roots of maximal {\it homogeneous} Vandermonde determinants; i.e., the homogeneous transfinite diameter $d^{(H)}(K)$ of a compact set (cf., \cite{jed}). From the above remarks, it follows that 
\begin{equation} \label{circled} \hbox{for circled sets} \  K, \  \delta(K)=d^{(H)}(K). \end{equation}
Since we will be using the homogeneous transfinite diameter, we amplify the discussion. We relabel the standard basis monomials $\{e_i^{(H,d)}(z)=z^{\alpha(i)}=z_1^{\alpha_1}\cdots z_d^{\alpha_d}\}$ where $|\alpha(i)|=d, \ i=1,...,h_n$, we define the $d-$homogeneous Vandermonde determinant 
\begin{equation} \label{vdmh}VDMH_d((\zeta_1,...,\zeta_{h_n}):=\det \bigl[e_i^{(H,d)}(\zeta_j)\bigr]_{i,j=1,...,h_n}.\end{equation}
Then 
\begin{equation} \label{htdlim}d^{(H)}(K)=\lim_{d\to \infty}\bigl[\max_{\zeta_1,...,\zeta_{h_n}\in K}|VDMH_d(\zeta_1,...,\zeta_{h_n})|\bigr]^{1/dh_n} \end{equation} is the homogeneous transfinite diameter of $K$; the limit exists and equals 
$$\exp\bigl[\frac{1}{\hbox{meas}(\Sigma)}\int_{\Sigma^0}\log {T(Y_2,\theta)}d\theta \bigr]$$
where $T(Y_2,\theta)$ comes from (2).

Finally, related to example (3), there are similar properties for the weighted version of directional Chebyshev constants and transfinite diameter. To define weighted notions, let $K\subset {\bf
C}^d$ be closed and let $w$ be an admissible weight function on $K$. We define a {\it weighted transfinite diameter}
$$d^w(K):=\exp\bigl[\frac{1}{\hbox{meas}(\Sigma)}\int_{\Sigma^0}\log {\tau^w(K,\theta)}d\theta \bigr] $$
as in \cite{bloomlev2} where $\tau^w(K,\theta)= T(Y_3,\theta)$ from (3); i.e., the right-hand-side of this equation is the quantity $\tau(Y_3)$.

We remark that if $\{K_j\}$ is a decreasing sequence of locally regular compacta with $K_j \downarrow K$, and if $w_j$ is a continuous admissible weight function on $K_j$ with $w_j \downarrow w$ on $K$ where $w$ is an admissible weight function on $K$, then the argument in Proposition 7.5 of  \cite{bloomlev2} shows that $\lim_{j\to \infty} \tau^{w_j}(K_j,\theta) = \tau^w(K,\theta)$ for all $\theta \in \Sigma^0$  and hence
\begin{equation}\label{decrlimit1} \lim_{j\to \infty} d^{w_j}(K_j) = d^w(K). \end{equation}
In particular, (\ref{decrlimit1}) holds in the unweighted case ($w\equiv 1$) for any decreasing sequence $\{K_j\}$ of compacta with $K_j \downarrow K$; i.e., 
\begin{equation}\label{decrlimit2} \lim_{j\to \infty} \delta (K_j) = \delta(K) \end{equation}
(cf., \cite{bloomlev2} equation (1.13)). Another useful fact is that 
\begin{equation} \label{hateqn} \delta(K)=\delta (\hat K) \ \hbox{and} \  d^{(H)}(K)=d^{(H)}(\hat K) \end{equation}
for $K$ compact where
$$\hat K :=\{z\in \CC^d: |p(z)|\leq ||p||_K, \ \hbox{all polynomials} \ p\}$$ 
is the polynomial hull of $K$.

Another natural definition of a weighted transfinite diameter has been described using weighted Vandermonde determinants. Let $K\subset \CC^d$ be compact and let
$w$ be an admissible weight function on
$K$.  We write
\begin{equation} \label{wn} W_n:=\max_{\zeta_1,...,\zeta_n\in K}|W(\zeta_1,...,\zeta_n)|
\end{equation}
and, apriori, 
\begin{equation}  \delta^w(K):=\limsup_{d\to \infty}W_{m_n}^{1/ l_n}. \end{equation}
To illustrate a useful technique, we give another proof of the existence of this limit.
		
	\begin{proposition} \label{wtdlimit} Let $K\subset \CC^d$ be a compact set with an admissible weight function $w$. The limit 
	$$\lim_{d\to \infty} \bigl[\max_{\lambda^{(i)}\in K}|VDM(\lambda^{(1)},...,\lambda^{(m_n^{(d)})})|\cdot w(\lambda^{(1)})^d\cdots w(\lambda^{(m_n^{(d)})})^d\bigr] ^{1/l_n^{(d)}}$$
	exists and equals $\delta^w(K)$.
		\end{proposition}
		
		\begin{proof} Following \cite{bloom}, we define the circled set 
	$$F=F(K,w):=\{(t,z)=(t,t \lambda )\in \CC^{d+1}: \lambda \in K, \ |t|=w(\lambda)\}.$$
	We first relate weighted Vandermonde determinants for $K$ with homogeneous Vandermonde determinants for the compact set 
\begin{equation} 
\label{FD}F(D):=\{(t,z)=(t,t \lambda )\in \CC^{d+1}: \lambda \in K, \ |t|\leq w(\lambda)\}.\end{equation}
Note that $F\subset \bar F \subset F(D) \subset \hat {\bar F}$ (cf., \cite{bloom}, (2.4)) where $\hat {\bar F}$ is the polynomial hull of $\bar F$ (recall (\ref{hateqn})); thus 
\begin{equation} 
\label{fd} d^{(H)}(\bar F) =d^{(H)}(F(D)). 
\end{equation} 
To this end, for each positive integer $d$, choose 
	$$m_n^{(d)}=  {n+d \choose d}$$
	(recall (\ref{precrucial})) points $\{(t_i,z^{(i)})\}_{i=1,...,m_n^{(d)}}=\{(t_i,t_i\lambda^{(i)})\}_{i=1,...,m_n^{(d)}}$ in $F(D)$ and form the $d-$homogeneous Vandermonde determinant
	$$VDMH_d((t_1,z^{(1)}),...,(t_{m_n^{(d)}},z^{(m_n^{(d)})})).$$ 
	We extend the lexicographical order of the monomials in $\CC^d$ to $\CC^{d+1}$ by letting $t$ precede any of $z_1,...,z_d$. Writing the standard basis monomials of degree $d$ in 
	$\CC^{d+1}$ as 
	$$\{t^{d-j} e_{k}^{(H,d)}(z): j=0,...,d; \ k=1,...,h_j\};$$ 
	i.e., for each power $d-j$ of $t$, we multiply by the standard basis monomials of degree $j$ in $\CC^{d}$, and dropping the superscript $(d)$ in $m_n^{(d)}$, we have the $d-$homogeneous Vandermonde matrix
$$\left[\begin{array}{ccccc}
 t_1^d &t_2^d&\ldots  &t_{m_n}^d\\
 t_1^{d-1}e_{2}(z^{(1)})  &t_2^{d-1}e_{2}(z^{(2)})&\ldots  &t_{m_n}^{d-1}e_{2}(z^{(m_n)})\\
  \vdots  & \vdots & \ddots  & \vdots \\
e_{m_n}(z^{(1)}) &e_{m_n}(z^{(2)} ) &\ldots  &e_{m_n}(z^{(m_n)})
\end{array}\right]$$
$$=\left[\begin{array}{ccccc}
 t_1^d &t_2^d&\ldots  &t_{m_n}^d\\
 t_1^{d-1}z_1^{(1)}  &t_2^{d-1}z_1^{(2)}&\ldots  &t_{m_n}^{d-1}z_1^{(m_n)}\\
  \vdots  & \vdots & \ddots  & \vdots \\
(z^{(1)}_d)^d &(z^{(2)}_d)^d  &\ldots  &(z^{(m_n)}_d)^d
\end{array}\right].$$
	Factoring $t_i^d$ out of the $i-$th column, we obtain
	$$VDMH_d((t_1,z^{(1)}),...,(t_{m_n},z^{(m_n)}))=t_1^d\cdots t_{m_n}^d\cdot VDM(\lambda^{(1)},...,\lambda^{(m_n)});$$
thus, writing $|A|:=|\det A|$ for a square matrix $A$,  
\begin{align} \label{matrix}
\left|\begin{array}{ccccc}
 t_1^d &t_2^d&\ldots  &t_{m_n}^d\\
 t_1^{d-1}z_1^{(1)}  &t_2^{d-1}z_1^{(2)}&\ldots  &t_{m_n}^{d-1}z_1^{(m_n)}\\
  \vdots  & \vdots & \ddots  & \vdots \\
(z^{(1)}_d)^d &(z^{(2)}_d)^d  &\ldots  &(z^{(m_n)}_d)^d
\end{array}\right|\end{align}
$$=|t_1|^d\cdots |t_{m_n}|^d \left|\begin{array}{ccccc}
 1 &1 &\ldots  &1\\
 \lambda^{(1)}_1 &\lambda^{(2)}_1 &\ldots  &\lambda^{(m_n)}_1\\
  \vdots  & \vdots & \ddots  & \vdots \\
(\lambda^{(1)}_d)^d &(\lambda^{(2)}_d)^d &\ldots  &(\lambda^{(m_n)}_d)^d
\end{array}\right|,$$
where $\lambda^{(j)}_k= z^{(j)}_k/t_j$ provided $t_j\not = 0$.  By definition of $F(D)$, since $(t_i,z^{(i)})=(t_i,t_i\lambda^{(i)})\in F(D)$, we have $ |t_i|\leq w(\lambda^{(i)})$. Clearly the maximum of $$|VDMH_d((t_1,z^{(1)}),...,(t_{m_n},z^{(m_n)}))|$$ over points in $F(D)$ will occur when all $|t_j|=w(\lambda^{(j)})>0$ (recall $w$ is an admissible weight) so that from (\ref{matrix})
	$$\max_{(t_i,z^{(i)})\in F(D)}|VDMH_d((t_1,z^{(1)}),...,(t_{m_n},z^{(m_n)}))|=$$
	$$\max_{\lambda^{(i)}\in K}|VDM(\lambda^{(1)},...,\lambda^{(m_n)})|\cdot w(\lambda^{(1)})^d\cdots w(\lambda^{(m_n)})^d.$$
	As mentioned in the discussion of (\ref{htdlim}) the limit
	$$\lim_{d\to \infty} \bigl[\max_{(t_i,z^{(i)})\in F(D)}|VDMH_d((t_1,z^{(1)}),...,(t_{m_n},z^{(m_n)}))|\bigr]^{1/dh_n^{(d+1)}}$$
	exists \cite{jed} and equals $d^{(H)}(F(D))$; thus the limit 
	$$\lim_{d\to \infty} \bigl[\max_{\lambda^{(i)}\in K}|VDM(\lambda^{(1)},...,\lambda^{(m_n)})|\cdot w(\lambda^{(1)})^d\cdots w(\lambda^{(m_n)})^d\bigr] ^{1/l_n^{(d)}}$$
	exists and equals $\delta^w(K)$. \end{proof}
	
	\begin{corollary} For $K\subset \CC^d$ a nonpluripolar compact set with an admissible weight function $w$ and $$F=F(K,w):=\{(t,z)=(t,t \lambda )\in \CC^{d+1}: \lambda \in K, \ |t|=w(\lambda)\},$$
		\begin{equation} \label{homvswtd} \delta^w(K)=d^{(H)}(\bar F)^{\frac{d+1}{d}}=\delta(\bar F)^{\frac{d+1}{d}}. \end{equation}
		\end{corollary}
		\begin{proof} The first equality follows from the proof of Proposition \ref{wtdlimit} using (\ref{fd}) and the relation 
	$$l_n^{(d)}=(\frac{d}{d+1})\cdot dh_n^{(d+1)}$$
	(see (\ref{crucial})). The second equality is (\ref{circled}). 
	\end{proof}

We use this corollary to prove (\ref{decrlimit}).

\begin{proposition} \label{polypoly}
Let $K\subset \CC^d$ be compact and let $w$ be an admissible weight on $K$. There exist a sequence of locally regular compacta $\{K_j\}$ decreasing to $K$ and a sequence of weights 
$\{w_j\}$ with $w_j$ continuous and admissible on $K_j$ such that $w_{j+1}\leq w_j|_{K_{j+1}}$ and $w_j \downarrow w$ on $K$, and we have 
\begin{equation}\label{wtddecrlimi}
\lim_{j\to \infty}\delta^{w_j}(K_j) = \delta^w(K).
\end{equation}
\end{proposition}

\begin{proof} If $K$ is polynomially convex, we can take the  $\{K_j\}$ to be polynomial polyhedra; it is known (cf., \cite{[P]}) that any nonempty analytic polyhedron $P$ in $\CC^d$ is locally regular at every point in $\overline{P}$. Note that the extension of
$w$ by $0$ on
$K_1\setminus K$ is usc on
$K_1$ so that we can find a sequence of continuous functions
$\{\phi_j\}$ on $K_1$ which decrease to $w$ on $K$ and to $0$ on $K_1\setminus K$. Define $w_j:=\phi_j|_{K_j}$ so that $w_j$ is a continuous, admissible
weight function on the locally regular set $K_j$ (when necessary, we consider $w_j$ as an admissible weight on
$K_1$ by setting $w_j =0$ on $K_1\setminus K_j$). Writing $Q_j:=-\log w_j$, we have $V_{K_j,Q_j}$ is continuous \cite{sic}. 

If $K$ is not necessarily polynomially convex, a recent result of N. Q. Dieu \cite{NQD} can be used to show that the closure $\overline D$ of any bounded domain $D\subset \CC^d$ with $C^1-$boundary is locally regular at each point of $\overline D$. (We remark that the weaker result that $\overline D$ is {\it regular} follows from arguments in \cite{[P]}; cf., \cite{[K]} p. 202). Thus we take a decreasing sequence $\{K_j\}$ with each $K_j$ being a finite union of closures of bounded domains with $C^1-$boundary and repeat the construction of $\{w_j\}$.

 Since $w_{j+1}\leq w_j|_{K_{j+1}}$ the circled sets 
$$F_j(D):=\{(t,z)=(t,t \lambda )\in \CC^{d+1}: \lambda \in K_j, \ |t|\leq w_j(\lambda)\}$$
are decreasing and the result follows from (\ref{homvswtd}) and (\ref{decrlimit2}). 
\end{proof}

\section{Appendix 2: Relations between Rumely formulas, and $d^w(K)$ vs. $\delta^w(K)$.}\label{sec:rimplieswr}
As mentioned, the Rumely formulas which we proved are symmetrized versions of Rumely's original formulas, due to Demarco and Rumely \cite{demrum}. We begin by proving the equivalence of these 
formulas.  First, recall for $u\in L^+(\CC^d)$, we defined
$$\rho_u(z):=\limsup_{|\lambda| \to \infty} [u(\lambda z)- \log |\lambda|]$$
and 
$$\tilde \rho_u(z):=\limsup_{|\lambda| \to \infty} [u(\lambda z)- \log |\lambda z|];$$
the latter we can consider as a function on $\PP^{d-1}$. We make the convention in this section, until Corollary \ref{lastcor}, that $dd^c = \frac{1}{2\pi}(2i \partial \bar \partial)$ so that in any dimension $d=1,2,...$, 
$$\int_{\CC^d} (dd^cu)^d =1$$
for any $u\in L^+(\CC^d)$. In terms of the Robin function $\rho_K$, Rumely's original formula (cf., \cite{rumely}) takes the form 
\begin{align} \label{rum}  -\log \delta(K) = \frac{1}{d}\bigl[\int_{\CC^{d-1}}\rho_K(1,t_2,...,t_d)(dd^c\rho_K(1,t_2,...,t_d))^{d-1}\end{align}
$$+ \int_{\CC^{d-2}}\rho_K(0,1,t_3,...,t_d)(dd^c\rho_K(0,1,t_3,...,t_d))^{d-2}$$
$$+\cdots + \int_{\CC}\rho_K(0,..,0,1,t_d)(dd^c\rho_K(0,..,0,1,t_d) + \rho_K(0,..,0,1)\bigr].$$
To relate this with (\ref{unwtdtruerumely}), we prove a generalization of this equivalence. Taking $u=V_K^*$ and $v=V_T$ in (\ref{symm}) below (recall $T$ is the unit torus so that $\rho_{T}(z_1,...,z_d)=\max [\log|z_1|,...,\log |z_d|]$) proves that (\ref{unwtdtruerumely}), with ``$2\pi$'' replaced by ``$1$,'' equals (\ref{rum}). 

\begin{proposition}
\label{demrumsym}
Let $u,v\in L^+(\CC^d)$. Then
\begin{equation}\label{symm}\int_{\PP^{d-1}}[\tilde \rho_{u}- \tilde \rho_{v}]\sum_{j=0}^{d-1}(dd^c \tilde \rho_{u}+\omega)^j \wedge (dd^c \tilde \rho_{v}+\omega)^{d-j-1}\end{equation}
$$=\int_{\CC^{d-1}}\rho_{u}(1,t_2,...,t_d)(dd^c\rho_{u}(1,t_2,...,t_d))^{d-1} $$ $$ +\int_{\CC^{d-2}}\rho_{u}(0,1,t_3,...,t_d)(dd^c\rho_{u}(0,1,t_3,...,t_d))^{d-2}$$ $$ +\cdots + \rho_{u}(0,0,...,0,1)$$
$$-\int_{\CC^{d-1}}\rho_{v}(1,t_2,...,t_d)(dd^c\rho_{v}(1,t_2,...,t_d))^{d-1} $$ $$ -\int_{\CC^{d-2}}\rho_{v}(0,1,t_3,...,t_d)(dd^c\rho_{v}(0,1,t_3,...,t_d))^{d-2}$$ $$ -\cdots - \rho_{v}(0,0,...,0,1).$$
\end{proposition}

\begin{proof} 
We first relate the $\CC^d$ integrals in (\ref{symm}) to $\PP^{d-1}$ integrals. We consider $\PP^d=\CC^d\cup \PP^{d-1}$ where, if $[Z_0:\cdots:Z_d]$ are homogeneous coordinates on $\PP^d$, we identify $\CC^d$ with $U_0:=\{Z_0\not = 0\}$ and $\PP^{d-1}$ with the hyperplane $\{Z_0=0\}$. On $U_0$ we have local coordinates 
$$z:=(z_1,...,z_d):=(Z_1/Z_0,...,Z_d/Z_0).$$ With a slight abuse of notation, we write
$$\tilde \rho_u([z_1:\cdots:z_d])=\rho_u(z_1,...,z_d);$$
we give a local coordinate relationship in equation (\ref{tilderho}) below.

Writing $U_j:=\{Z_j\not = 0\}$, on 
$U_0\cap U_1$, for example, by letting
$$t_0=Z_0/Z_1, \ t_2=Z_2/Z_1=(Z_2/Z_0)t_0, \ ..., \ t_d=Z_d/Z_1=(Z_d/Z_0)t_0,$$
we have
$$(z_1,...,z_d)=(1/t_0,t_2/t_0,...,t_d/t_0).$$
Thus we see that for $u\in L^+(\CC^d)$,
$$u(1/t_0,t_2/t_0,...,t_d/t_0)+\log |t_0| = u(z_1,...,z_d)-\log |z_1|$$
which extends (since $u\in L^+(\CC^d)$) across $t_0=0$ via 
$$\limsup_{t_0\to 0}[u(1/t_0,t_2/t_0,...,t_d/t_0)+\log |t_0|]$$
$$=\limsup_{|z_1|\to \infty}[u(z_1,...,z_d)-\log |z_1|]=\rho_u(1,t_2,...,t_d).
$$
In particular, in local coordinates $t_2,...,t_d$ on $\PP^{d-1}\cap U_1$ with $|t|^2=|t_2|^2+\cdots |t_d|^2$, we have 
\begin{equation}\label{tilderho}
\tilde \rho_u([1:t_2:\cdots :t_d])=\rho_u(1,t_2,...,t_d)- \frac{1}{2}\log(1+|t|^2)
\end{equation}
so that
$$
dd^c \tilde \rho_u +\omega = dd^c \rho_u(1,t_2,...,t_d).
$$

We can use this to rewrite
\begin{equation}\label{mess}\int_{\CC^{d-1}}\rho_{u}(1,t_2,...,t_d)(dd^c\rho_{u}(1,t_2,...,t_d))^{d-1} \end{equation} $$ +\int_{\CC^{d-2}}\rho_{u}(0,1,t_3,...,t_d)(dd^c\rho_{u}(0,1,t_3,...,t_d))^{d-2}$$ $$ +\cdots + \rho_{u}(0,0,...,0,1)$$
as integrals over $\PP^{d-1}$. 
Following \cite{demrum}, we introduce coordinate-wise Robin functions on $\PP^{d-1}$. We write the projectivized Robin function as 
$$g_u([z_1:\cdots:z_d]):= \tilde \rho_u([z_1:\cdots:z_d]) = \limsup_{|\lambda|\to \infty}[u(\lambda z)-\log |\lambda z|]$$
and, for $j=1,...,d$ we define the coordinate-wise Robin functions
$$g_j([z_1:\cdots:z_d])=g_{u,j}([z_1:\cdots:z_d]):= \limsup_{|\lambda|\to \infty}[u(\lambda z)-\log |\lambda z_j|].$$
Clearly on $\PP^{d-1}\cap U_j$, we have
$$g_u = g_j -\phi$$
where $\phi$ is a potential for $\omega$; i.e., $dd^c\phi =\omega$. This relation, together with (\ref{tilderho}), shows that on $\PP^{d-1}\cap U_1$, we have
$$g_1 = \rho_u(1,t_2,...,t_d).$$

We write $T_j:=[Z_j=0]$ for the current of integration over the hyperplane $\{Z_j=0\}$. 
It is straightforward to see that
$$dd^cg_j + T_j$$
is a globally defined $(1,1)-$current on $\PP^{d-1}$; indeed, writing 
$$T_u:=dd^c g_u +\omega,$$
we have
$$dd^cg_j +T_j= T_u.$$
In particular, 
$$T_u=dd^cg_j \ \hbox{on} \ \PP^{d-1}\cap U_j.$$
We leave it as an exercise to show that (\ref{mess}) can be rewritten as
\begin{equation}\label{mess2}\int_{\PP^{d-1}}\bigl[g_1T_u^{d-1}+g_2T_u^{d-2}\wedge T_1 +\cdots + g_dT_1\wedge \cdots \wedge T_{d-1}\bigr].\end{equation}
We give the proof in the case $d=2$. We want to show 
\begin{equation}\label{mess3}\int_{\CC}\rho_{u}(1,t)dd^c\rho_{u}(1,t) +\rho_u(0,1) = \int_{\PP}(g_1T_u+g_2T_1).\end{equation}
We saw that on $\CC=\PP\cap U_1$, $g_1 = \rho_u(1,t)$ and $$T_u=dd^c g_u +\omega=dd^c g_1 =dd^c\rho_{u}(1,t) $$
so that 
$$\int_{\CC}\rho_{u}(1,t)dd^c\rho_{u}(1,t)=\int_{\PP}g_1T_u.$$
In a similar fashion, we can write $g_2$ in local coordinates on $\CC=\PP\cap U_2$ as $g_2 =\rho_u(t,1)$ so that $g_2 T_1 =\rho_u(0,1)$, yielding (\ref{mess3}).

Next, using the identities
$$dd^cg_j = T_u -T_j, \ j=1,...,d,$$
for any $n=0,1,...,d$ we can rewrite
$$\omega^n\wedge T_u^{d-n-1} -\omega^n \wedge T_1 \wedge \cdots \wedge T_{d-n-1}$$
as
$$\omega^n\wedge \bigl[ dd^c g_1\wedge T_u^{d-n-2}+dd^c g_2\wedge T_u^{d-n-3}\wedge T_1 +$$
$$\cdots + dd^cg_{d-n-1}\wedge T_1 \wedge \cdots \wedge T_{d-n-2}\bigr].$$
Integrate $g_u$ with respect to this $(d-1,d-1)$ current over $\PP^{d-1}$:
$$\int_{\PP^{d-1}} g_u[\omega^n\wedge T_u^{d-n-1} -\omega^n \wedge T_1 \wedge \cdots \wedge T_{d-n-1}];$$
the integral is finite as it is a difference; then using the identity and 
integrating by parts to move $dd^c$ from $g_j$ to $g_u$ and using $dd^cg_u =T_u -\omega$, we get  
$$\int_{\PP^{d-1}}\sum_{j=0}^{d-1} g_{j+1}(T_u-\omega)\wedge \omega^n\wedge T_u^{d-n-j-2}\wedge T_1 \wedge \cdots \wedge T_j.$$
Now we leave it as an exercise, using this calculation, to show that (\ref{mess2}) can be rewritten as 
$$ \int_{\PP^{d-1}}\sum_{j=0}^{d-1} g_{j+1} T_u^{d-1-j}\wedge T_1 \wedge \cdots \wedge T_j$$
$$=\int_{\PP^{d-1}}\sum_{j=0}^{d-1}g_u\omega^j \wedge T_u^{d-j-1} $$
$$+\int_{\PP^{d-1}}\sum_{j=0}^{d-1}(g_{d-j}-g_u)\omega^j \wedge T_1 \wedge \cdots \wedge T_{d-j-1}$$
(see \cite{demrum}, p. 149). But since $g_k-g_u= \log |z| -\log |z_k|$ this last sum is {\it independent of $u$}. Thus if we apply this rewriting of (\ref{mess2}) to two functions  $u$ and $v$ in $L^+(\CC^d)$, the difference becomes 
$$\int_{\PP^{d-1}}\sum_{j=0}^{d-1}g_u\omega^j \wedge T_u^{d-j-1} -\int_{\PP^{d-1}}\sum_{j=0}^{d-1}g_v\omega^j \wedge T_v^{d-j-1} .$$
However, using $\omega = T_v -dd^cg_v$ in the {\it first} sum, and using $\omega = T_u -dd^cg_u$ in the {\it second} sum, it is straightforward to rewrite this difference as
$$\int_{\PP^{d-1}}\sum_{j=0}^{d-1}(g_u - g_v) T_u^{d-j-1} \wedge T_v^j$$
which gives (\ref{symm}).
\end{proof}

We next prove ({\ref{dwdeltw}), the relationship between the two weighted transfinite diameters $d^w(K)$ and $\delta^w(K)$, following \cite{[BL]}. We use the version of Rumely's formula in (\ref{rum}). We begin by rewriting (\ref{rum}) for {\it regular circled} sets $K$ using an observation of Sione Ma'u. Note that for such sets, $V_K^*=\rho_K^+:=\max(\rho_K,0)$ (cf., \cite{blm}, Lemma 5.1). If we intersect $K$ with a hyperplane ${\mathcal H}$ through the origin, e.g., by rotating coordinates, we take ${\mathcal H}=\{z=(z_1,...,z_d)\in \CC^d: z_1 =0\}$, then $K\cap {\mathcal H}$ is a regular, compact, circled set in $\CC^{d-1}$ (which we identify with ${\mathcal H}$). Moreover, we have
$$\rho_{{\mathcal H}\cap K} (z_2,...,z_d)=\rho_K(0,z_2,...,z_d)$$
since each side is logarithmically homogeneous and vanishes for all points $(z_2,...,z_d)\in \partial ({\mathcal H}\cap K)$. Thus the terms 
$$\int_{\CC^{d-2}}\rho_K(0,1,t_3,...,t_d)(dd^c\rho_K(0,1,t_3,...,t_d))^{d-2}$$
$$+\cdots + \int_{\CC}\rho_K(0,..,0,1,t_d)(dd^c\rho_K(0,..,0,1,t_d) + \rho_K(0,..,0,1)$$
in (\ref{rum}) are seen to equal 
$$(d-1)\delta^{\CC^{d-1}}({\mathcal H}\cap K)$$ (where we temporarily write $\delta^{\CC^{d-1}}$ to denote the transfinite diameter in $\CC^{d-1}$ for emphasis) by applying (\ref{rum}) in $\CC^{d-1}$ to the set ${\mathcal H}\cap K$. Hence we have
\begin{align}  \label{newrum}  -\log \delta(K) = \frac{1}{d}\int_{\CC^{d-1}}\rho_K(1,t_2,...,t_d)(dd^c\rho_K(1,t_2,...,t_d))^{d-1} \end{align}
$$+\bigl(\frac{d-1}{d}\bigr)[-\log \delta^{\CC^{d-1}}({\mathcal H}\cap K)].$$

	\begin{theorem} 
	\label{dwdelw}
	For $K\subset \CC^d$ a nonpluripolar compact set with an admissible weight function $w$,
	 \begin{equation} \label{conj1}\delta^w(K)=[\exp {(-\int_{K} Q(dd^cV^*_{K,Q})^{d})}]^{1/d}\cdot d^w(K).\end{equation}
	\end{theorem}
\begin{proof} We first assume that $K$ is locally regular and $Q$ is continuous so that $V_{K,Q}=V_{K,Q}^*$. As before, we define the circled set 
	$$F=F(K,w):=\{(t,z)=(t,t \lambda )\in \CC^{d+1}: \lambda \in K, \ |t|=w(\lambda)\}.$$
	We claim this is a regular compact set; i.e., $V_F$ is continuous. First of all, $V_F^*(t,z)=\max[\rho_F(t,z),0]$ (cf., Proposition 2.2 of \cite{bloom}) so that it suffices to verify that $\rho_F(t,z)$ is continuous. From Theorem 2.1 and Corollary 2.1 of \cite{bloom}, 
	\begin{equation} \label{relation}V_{K,Q}(\lambda)=  \rho_F(1,\lambda) \ \hbox{on} \ \CC^d\end{equation}  
which implies, by the logarithmic homogeneity of $\rho_F$, that $\rho_F(t,z)$ is continuous on $\CC^{d+1}\setminus \{t=0\}$. Corollary 2.1 and equation (2.6) in \cite{bloom} give that
\begin{equation} \label{relation2}\rho_F(0,\lambda)= \rho_{K,Q}(\lambda) \ \hbox{for} \ \lambda \in \CC^d \end{equation}
and $\rho_{K,Q}$ is continuous by Theorem 2.5 of \cite{bloomlev2}. Moreover, the limit exists in the definition of $\rho_{K,Q}$:
$$\rho_{K,Q}(\lambda):=\limsup_{|t|\to \infty} [V_{K,Q}(t\lambda)-\log |t|]=\lim_{|t|\to \infty} [V_{K,Q}(t\lambda)-\log |t|];$$
and the limit is uniform in $\lambda$ (cf., Corollary 4.4 of \cite{blm}) which implies, from (\ref{relation}) and (\ref{relation2}), that $\lim_{t\to 0} \rho_F(t,\lambda)=\rho_F(0,\lambda)$ so that $\rho_F(t,z)$ is continuous. 
In particular, from Theorem 2.5 in Appendix B of \cite{safftotik},
$$V_{K,Q}(\lambda)=Q(\lambda) = \rho_F(1,\lambda) \ \hbox{on the support of} \  (dd^cV_{K,Q})^{d}$$ so that
\begin{equation} \label{qandv} \int_{K} Q(\lambda)(dd^cV_{K,Q}(\lambda))^{d}=\int_{\CC^{d}} \rho_F(1,\lambda)(dd^c\rho_F(1,\lambda))^{d}.\end{equation}
 On the other hand, $K^w_{\rho}:=\{\lambda\in \CC^{d}: \rho_{K,Q}(\lambda)\leq 0\}$ is a circled set, and, according to eqn. (3.14) in \cite{bloomlev2}, 
\begin{equation}\label{nontriv}
d^w(K)=\delta(K^w_{\rho}).
\end{equation}
We remark that (\ref{nontriv}) is highly nontrivial; however, as Theorem \ref{dwdelw} is not our main objective in these notes, we omit the proof. But 
$$\rho_{K,Q}(\lambda)=\limsup_{|t|\to \infty} [V_{K,Q}(t\lambda)-\log |t|]$$
$$=\limsup_{|t|\to \infty} [\rho_F(1,t\lambda)-\log |t|]=\limsup_{|t|\to \infty} \rho_F(1/t,\lambda)=\rho_F(0,\lambda).$$
Thus
$$K^w_{\rho}=\{\lambda\in \CC^{d}: \rho_F(0,\lambda)\leq 0\}=F\cap {\mathcal H}$$
where ${\mathcal H}=\{(t,z)\in \CC^{d+1}: t=0\}$ and hence 
\begin{equation} \label{step1} d^w(K)=\delta(K^w_{\rho})=\delta(F\cap {\mathcal H}).\end{equation}
From (\ref{newrum}) applied to $F\subset \CC^{d+1}$,
\begin{equation} \label{step2} -\log \delta(F)=  \frac{1}{d+1}\int_{\CC^{d}}\rho_F(1,\lambda)(dd^c\rho_F(1,\lambda))^{d}\end{equation}
$$+ (\frac{d}{d+1})[-\log \delta(F\cap {\mathcal H})].$$
Finally, from (\ref{homvswtd}) (note $F=\bar F$ since $w$ is continuous),
\begin{equation} \label{step3}  \delta^w(K)=\delta(F)^{\frac{d+1}{d}};\end{equation}
putting together (\ref{qandv}), (\ref{step1}),  (\ref{step2}) and (\ref{step3}) gives the result if $K$ is locally regular and $Q$ is continuous. 

The general case follows from approximation. Using Proposition \ref{polypoly} we can take a sequence of locally regular compacta $\{K_j\}$ decreasing to $K$ and a sequence of weight functions $\{w_j\}$ with $w_j$ continuous and admissible on $K_j$ and $w_j \downarrow w$ on $K$ and by (\ref{wtddecrlimi}) we have 
$$\lim_{j\to \infty} \delta^{w_j}(K_j) =\delta^w(K).$$  From (\ref{decrlimit1}) we have
$$
\lim_{j\to \infty} d^{w_j}(K_j) =d^w(K). 
$$
Applying (\ref{conj1}) to $K_j, w_j, Q_j$ and using these facts, we conclude that 
$$\lim_{j\to \infty} \int_{K_j} Q_j(dd^cV_{K_j,Q_j})^{d}=\lim_{j\to \infty} \int_{K_j} V_{K_j,Q_j}(dd^cV_{K_j,Q_j})^{d}$$
exists. Since $V_{K_j,Q_j}\uparrow V_{K,Q}^*$ a.e., by Lemma 3.6.2 of \cite{[K]}, the measures $V_{K_j,Q_j}(dd^cV_{K_j,Q_j})^{d}$ converge weak-* to $V_{K,Q}^*(dd^cV^*_{K,Q})^{d}$ so that
$$\lim_{j\to \infty} \int_{K_j} Q_j(dd^cV_{K_j,Q_j})^{d}= \int_{K} Q(dd^cV^*_{K,Q})^{d},$$ completing the proof of (\ref{conj1}).
\end{proof}

As an immediate corollary, we obtain the symmetrized version of Rumely's original weighted formula, equation (\ref{realtruerumely}). Here we revert to the convention used prior to this appendix that $dd^c = 2i \partial \bar \partial$.

\begin{corollary}\label{lastcor} Let $K\subset \CC^d$ be compact and $w$ be an admissible weight. Then
$$-\log d^w(K)=\frac{1}{d(2\pi)^{d-1}} \int_{\PP^{d-1}}[\tilde \rho_{K,Q}- \tilde \rho_{T}]\sum_{j=0}^{d-1}(dd^c \tilde \rho_{K,Q}+\omega)^j \wedge (dd^c \tilde \rho_{T}+\omega)^{d-j-1}.$$
\end{corollary}

\noindent Another proof of Theorem \ref{dwdelw} has been given by Nystrom in \cite{ny}.

\end{document}